\documentclass[10pt, a4paper]{amsart}

\usepackage{amsmath,amssymb,enumerate}

\usepackage[T1]{fontenc}
\usepackage[all,cmtip]{xy}

\usepackage{babel}
\usepackage{amstext}
\usepackage{amsmath}
\usepackage{amsfonts}
\usepackage{latexsym}
\usepackage{ifthen}
\usepackage{amssymb}
\usepackage{xypic}
\usepackage{array}
\usepackage{makecell}
\xyoption{all}
\newcommand{\ma}{\mathcal}

\usepackage[top=2.5cm,bottom=3cm,left=3.2cm,right=3.2cm]{geometry}

\newtheorem{lemma1}{}[section]

\newenvironment{lemma}{\begin{lemma1}{\bf Lemma.}}{\end{lemma1}}
\newenvironment{example}{\begin{lemma1}{\bf Example.}\rm}{\end{lemma1}}

\newenvironment{theorem}{\begin{lemma1}{\bf Theorem.}}{\end{lemma1}}

\newenvironment{proposition}{\begin{lemma1}{\bf Proposition.}}{\end{lemma1}}
\newenvironment{corollary}{\begin{lemma1}{\bf Corollary.}}{\end{lemma1}}
\newenvironment{remark}{\begin{lemma1}{\bf Remark.}\rm}{\end{lemma1}}
\newenvironment{remarks}{\begin{lemma1}{\bf Remarks.}\rm}{\end{lemma1}}
\newenvironment{definition}{\begin{lemma1}{\bf Definition.}}{\end{lemma1}}

\newenvironment{setup}{\begin{lemma1}{\bf Setup.}}{\end{lemma1}}

\newenvironment{fact}{\begin{lemma1}{\bf Fact.}}{\end{lemma1}}

\newenvironment{remark*}{{\bf Remark.}}{}
\newenvironment{remarks*}{{\bf Remarks.}}{}
\newenvironment{example*}{{\bf Example.}}{}
\newenvironment{assumption*}{{\bf Assumption.}}{}

\newcommand{\R}{\ensuremath{\mathbb{R}}}
\newcommand{\Q}{\ensuremath{\mathbb{Q}}}
\newcommand{\Z}{\ensuremath{\mathbb{Z}}}
\newcommand{\C}{\ensuremath{\mathbb{C}}}
\newcommand{\N}{\ensuremath{\mathbb{N}}}
\newcommand{\PP}{\ensuremath{\mathbb{P}}}

\newcommand{\F}{\ensuremath{\mathbb{F}}}

\usepackage{eurosym}

\newcommand{\holom}[3]{\ensuremath{#1\colon#2  \rightarrow #3}}
\newcommand{\fibre}[2]{\ensuremath{#1^{-1} (#2)}}

\makeatletter
\ifnum\@ptsize=0  \fi \ifnum\@ptsize=2  \fi 

\newcommand\sF{{\mathcal F}}

\newcommand\sI{{\mathcal I}}

\newcommand\sO{{\mathcal O}}

\DeclareMathOperator*{\sing}{sing}

\DeclareMathOperator*{\Exc}{Exc}

\newcommand{\Chow}[1]{\ensuremath{\mbox{\rm Chow}(#1)}}

\newcommand{\NEX}{\overline{\mbox{NE}}(X)}

\newcommand{\Ne}[1]{\operatorname{NE}(#1)}
\newcommand{\NE}[1]{\overline{\operatorname{NE}}(#1)}

\usepackage{xcolor}

\usepackage{hyperref}
\hypersetup{
	colorlinks=true,
	linkcolor=red,
	citecolor=blue}

\newcommand{\AX}{\ensuremath{|-K_X|}}
\newcommand{\BsAX}{\ensuremath{\mbox{\rm Bs}}(|-K_X|)}

\author{Andreas H\"oring}
\author{Saverio Andrea Secci}

\address{Andreas H\"oring, Universit\'e C\^ote d'Azur, CNRS, LJAD, France}
\email{Andreas.Hoering@univ-cotedazur.fr}
\address{Saverio Andrea Secci, SISSA - Scuola Internazionale Superiore di Studi Avanzati, Via Bonomea 265, 34136 Trieste, Italy}
\email{ssecci@sissa.it}

\makeatletter
\@namedef{subjclassname@2020}{\textup{2020} Mathematics Subject Classification}
\makeatother

\subjclass[2020]{14J45,14J35,14E30,14J17}
\keywords{Fano manifold, anticanonical divisor, base locus, MMP}

\title{Classification of Fano fourfolds with large anticanonical base locus} 
\date{October 24th, 2025} 

\begin{document}

\begin{abstract} 
We give a classification of smooth Fano fourfolds such that
the base scheme of the anticanonical system is a smooth surface. 
As a consequence we show that there are exactly 22 deformation families of
such manifolds and they are all obtained by the same geometric construction.
These 22 families are closely related to the list of smooth Fano threefolds
that admit a $\PP^1$-bundle structure.
\end{abstract} 

\maketitle

\section{Introduction}

\subsection{Motivation and main result}
A del Pezzo surface in the most classical sense is a smooth projective surface $X$
such that the anticanonical system $|-K_X|$ defines an {\em embedding} of $X$ into a projective space \cite{DP87}, i.e.\ $-K_X$ is very ample. From the point of view of modern classification theory it is more natural to assume that a Fano manifold has ample anticanonical bundle and study the situation where $|-K_X|$ does not define an embedding. The most immediate obstruction is the presence of base points and it turns out that low-dimensional
Fano manifolds with non-empty base locus are rather rare:

\begin{theorem} \cite{Isk77} \label{thm:iskovskikh}
Let $X$ be a Fano manifold such that the base locus $\BsAX$ is not empty. 
\begin{itemize}
\item If $\dim X=2$, the surface $X$
has degree one and can be a realised as a sextic in $\PP(1,1,2,3)$.
\item If $\dim X=3$ one of the 
following holds:
\begin{itemize}
\item $X \simeq \PP^1 \times S$ where $S$ is del Pezzo surface of degree one;
\item $X$ is the blow-up of a smooth sextic in $\PP(1,1,1,2,3)$ along an elliptic curve.
\end{itemize}
\end{itemize}
\end{theorem}

In our paper \cite{HS25} we started a systematic investigation of smooth Fano fourfolds with large anticanonical base
locus, i.e.\ Fano fourfolds where $\BsAX$ has dimension two (which is the maximal dimension). We showed
that these manifolds exhibit a behaviour that is completely opposite to the three-dimensional case: all the anticanonical divisors
are singular. In this paper we build on \cite[Thm.1.3]{HS25} to give a complete classification:

\begin{theorem}
\label{theorem-main}
Let $X$ be a smooth Fano fourfold such that $h^0(X, \sO_X(-K_X)) \geq 4$.
Assume that the scheme-theoretic base locus
$\BsAX$ is a smooth irreducible surface $B$. 
Let $X' := \mbox{\rm Bl}_B X$ be the blow-up and denote by $E \simeq \PP(\ma N_{B/X}^*)$ the exceptional divisor.

Then $X'$ admits a flat Weierstra\ss\ fibration
$$
\holom{\phi}{X'}{T}
$$
with distinguished section $E$. 
Moreover $B$ and $\PP(\ma N_{B/X}^*)$ are Fano, and
the manifolds $X'$ and $X$ can be constructed from the data $(B,\, \ma N_{B/X}^*)$
by the construction in Proposition \ref{construction:general}.
\end{theorem}

An important feature of our classification is that it is effective, i.e.\ we are able to realise all the possibilities
appearing in our list. In fact, our construction can be applied to all the Fano threefolds with a $\PP^1$-bundle structure.
We can summarise this as follows:

\begin{corollary}  \label{corollary-main} 
There are exactly twenty-two deformation families of smooth 
 Fano fourfolds such that $h^0(X, \sO_X(-K_X)) \geq 4$ and the scheme-theoretic base locus
$\BsAX$ is a smooth irreducible surface $B$.
\end{corollary}

We list the invariants of the Fano fourfolds in Table \ref{table} and describe the birational geometry  in Section \ref{section-examples}, we will see that the Mori contractions are much more complicated than in dimension three. While the number of families is surprisingly high, we would like to stress that they all arise from the same construction. Moreover they share some striking geometric properties listed in Theorem \ref{theorem-nefdivisor}.

\subsection{General strategy and structure of the proof}
The proof of Iskovskikh's theorem \ref{thm:iskovskikh} is based on two important results: the existence of a smooth anticanonical divisor by Shokurov \cite{Sho80} and the classification of ample divisors
with base points on K3 surfaces by Mayer \cite{May72}. With this information at hand one can describe the fibre space structure on $X' := \mbox{\rm Bl}_{\tiny \BsAX} X$  and deduce the structure of the Fano threefold $X$. In the situation of Theorem \ref{theorem-main}
there is never a smooth anticanonical divisor and so far there is no analogue of Mayer's theorem for Calabi-Yau threefolds. Our strategy is to inverse the procedure, i.e.\ we start with the geometry of the blow-up, then study linear systems on elliptic Calabi-Yau threefolds and finally describe the geometry of the base scheme. In more detail, we proceed as follows:
let
$$
\holom{\mu}{X'}{X}
$$
be the blow-up of the base scheme $B$ and denote by $E$ the exceptional divisor. Let
$$
\holom{\varphi}{X'}{T'}
$$
be the fibration defined by the anticanonical system $|-K_{X'}|$. Using the assumption $h^0(X, \sO_X(-K_X)) \geq 4$
we show that $T'$ is a threefold, so $\varphi$ is an elliptic fibration. Our first goal is to show that the exceptional divisor
$E$ is a rational section of this fibration. 
This requires a general result for fixed divisors
on Calabi-Yau threefolds with an elliptic fibration which is of independent interest:

\begin{theorem} \label{thm-birational}
Let $Z$ be a normal $\Q$-factorial projective threefold with terminal singularities that is Calabi-Yau, i.e.\ we have
$$
\sO_Z(K_Z) \simeq \sO_Z \qquad \mbox{and} \qquad
h^1(Z, \sO_Z)=h^2(Z, \sO_Z)=0.
$$
Let 
$$
\holom{\varphi_Z}{Z}{U}
$$ 
be a fibration onto a normal surface $U$, and let $A$ be a nef and
big divisor on $Z$ such that the following holds: 
\begin{enumerate}
\item we have a decomposition into fixed and mobile part $|A| = S + |M|$ with 
$S$ a smooth prime divisor;
\item the restriction of $A$ to $S$ is big;
\item we have $M \simeq \varphi_Z^* M_U$ where $M_U$ is an ample globally generated Cartier divisor on $U$.
\end{enumerate}
Then $S$ is a rational section of the fibration $\varphi_Z$.
\end{theorem}

It seems likely that the statement also holds for a singular surface $S$ and $M_U$ a not necessarily globally generated ample line bundle, but these assumptions allow us to avoid further case distinctions in a long and technical proof.
 
We then prove a structure result for the Fano fourfold which is the basis for the rest of our investigation:

\begin{theorem} \label{theorem-nefdivisor}
Let $X$ be a smooth Fano fourfold such that $h^0(X, \sO_X(-K_X)) \geq 4$.
Assume that the scheme-theoretic base locus
$\BsAX$ is a smooth irreducible surface $B$. Let
$$
\holom{\mu}{X'}{X}
$$
be the blow-up of $B$ and denote by $E$ the exceptional divisor.
Then the following holds:
\begin{enumerate}
\item There exists a flat fibration 
$$
\holom{\pi}{X}{V}
$$
with integral fibres such that the general fibre $F$ is a del Pezzo surface of degree one
and $B \subset X$ is a $\pi$-section.
\item There exists a flat elliptic fibration
$$
\holom{\phi}{X'}{T}
$$
with integral fibres such that $E \subset X'$ is a $\phi$-ample section.
\item We have a commutative diagram
$$
\xymatrix{
X' \ar[r]^{\mu} \ar[d]_{\phi} & X \ar[d]^{\pi} \\
T \ar[r]_{\mu_T} & V 
}
$$
where $\mu_T$ is the $\PP^1$-bundle structure given by
$$
\holom{\mu_E :=\mu|_E}{E \simeq T}{B \simeq V}.
$$
\end{enumerate}
\end{theorem}

Theorem \ref{theorem-nefdivisor} 
provides an explanation for the existence of the base locus $B$: since $F = \fibre{\pi}{v}$ is a del Pezzo surface of degree one, its anticanonical system
$|-K_{F}|$ has a unique base point $p_v$. Yet $-K_X|_{F} \simeq - K_F$, so $p_v$ is also a basepoint of $|-K_X|$. Thus the union of these points defines
a rational section of $\pi$ that is contained in $\BsAX$. This section is exactly
the surface $B$.
It is tempting to believe that $\phi$ coincides with the fibration
$\varphi = \varphi_{|-K_{X'}|}$ but this can only be checked as a consequence of Theorem \ref{theorem-main}.

If $V \simeq B$ admits a birational Mori contraction we can use Theorem \ref{theorem-nefdivisor} to construct a very special birational Mori contraction on $X$, this allows us
to reduce to the case where $B$ is either the projective plane $\PP^2$ or a smooth quadric $\PP^1 \times \PP^1$. In these cases there are many more options,
for example $B$ can be the exceptional locus of a small contraction or a higher-dimensional fibre of an elementary contraction onto a threefold, see Examples \ref{example-new0}, \ref{example:flagvariety} and \ref{example:bidegreeonetwo}.
At this point the description of elementary Mori contractions with two-dimensional fibres
by Kawamata, Andreatta-Wi\'sniewski and Takagi \cite{Kaw89, AW98, Tak99} allows us to classify
the possibilities for the conormal bundle $\ma N_{B/X}^*$: it turns out that they are all Fano bundles
in the sense of \cite{SW90}. Once we have established this list we are able to determine the Weierstra\ss\ model of $X'$ and deduce the structure of $X$.

Finally let us note that 
the technical assumption  
$h^0(X, \sO_X(-K_X)) \geq 4$ in Theorem \ref{theorem-main}
is satisfied by all the known Fano fourfolds where $\BsAX$ is a surface.
We use the assumption in the proof of Lemma \ref{lemma-dimensionT}, studying the cases
$h^0(X, \sO_X(-K_X)) \in \{2, 3\}$ seems possible (cf. Remark \ref{remark-dimensionT}) but leads to numerous questions about surfaces with small invariants that we wanted to avoid in this paper.

\noindent\textbf{Acknowledgements.} 
A.H.\ was partially supported by the ANR-DFG project ``Positivity on K-trivial varieties'', ANR-23-CE40-0026 and DFG Project-ID 530132094. A.H.\ was also partially supported  by the France 2030 investment plan managed by the ANR, as part of the Initiative of Excellence of Universit\'e C\^ote d'Azur, reference ANR-15-IDEX-01.

S.A.S.\ thanks Universit\'a di Milano and Universit\'e C\^ote d'Azur for supporting his visit to Nice, while he was a postdoc at Universit\'a di Milano, in order to start this project; moreover, he is a member of GNSAGA, INdAM.

\section{Notation and basic results}

We work over $\C$, for general definitions we refer to \cite{Har77}.
All the schemes appearing in this paper are projective, manifolds and normal varieties will always be supposed to be irreducible.
For notions of positivity of divisors and vector bundles we refer to Lazarsfeld's books \cite{Laz04a, Laz04b}. 
We denote by $\simeq$ (resp.\ $\sim_\Q$) the linear equivalence of Weil divisors (resp.\ Weil $\Q$-divisors) and by $\equiv$ the numerical equivalence of $\Q$-Cartier divisors.
 Given a Cartier divisor $D$ we will denote by $\sO_X(D)$ both the associated invertible sheaf and the corresponding line bundle.
Somewhat abusively we will say that a Cartier divisor class is effective if it contains an effective divisor.

We use the terminology of \cite{KM98} for birational geometry and of \cite{Kol96} for rational curves.
We refer to \cite{Kol13} for the definitions and basic facts about singularities of the MMP.

\subsection{Technical preliminaries}

We recall some preliminary results:

\begin{lemma} \label{lemma-pullback-weil}
Let $\holom{\psi}{A}{B}$ a surjective morphism between normal projective varieties that does not contract a divisor. Then there exists a well-defined map
$$
\psi^* \colon \mbox{\rm Cl}(B) \rightarrow \mbox{\rm Cl}(A)
$$
such that the restriction to the smooth locus of $B$ is the pull-back of Cartier divisors.
\end{lemma}

\begin{proof}
Let $B_0 \subset B$ be the regular locus and set $A_0 := \fibre{\psi}{B_0}$.
Since $B$ is normal the complement of $B_0 \subset B$ has codimension at least two.
Since $\psi$ does not contract a divisor, the same holds for the complement of
$A_0 \subset A$. Thus we have isomorphisms
$$
 \mbox{Cl}(B) \simeq \mbox{Cl}(B_0) \simeq \mbox{Pic}(B_0) \qquad \mbox{Cl}(A) \simeq \mbox{Cl}(A_0).
$$
and $\psi^*$ is simply the composition
$$
 \mbox{Cl}(B) \simeq \mbox{Cl}(B_0) \simeq \mbox{Pic}(B_0) \stackrel{\psi^*}{\rightarrow}
 \mbox{Pic}(A_0) \hookrightarrow \mbox{Cl}(A_0) \simeq \mbox{Cl}(A).
$$
\end{proof}

\begin{lemma} \label{lemma:fano}
Let $\holom{\mu}{M'}{M}$ be a birational morphism between normal $\Q$-factorial projective varieties with klt singularities such that
the exceptional locus is a prime divisor $E$.
Assume that $K_{M'} \sim_\Q \mu^* K_M + \alpha E$ with $\alpha \geq 0$ (e.g.\ $\mu$ is a divisorial Mori contraction on a projective manifold). 
Set $B := \mu(E)$. Then the following holds
\begin{enumerate}
\item For every curve $C \subset M$ such that $C \not\subset B$ we have
\begin{equation}
\label{KMinequality}
-K_{M'} \cdot C' \leq -K_M \cdot C
\end{equation}
where $C' \subset M'$ is the strict transform.
\item If $-K_{M'}$ is nef and $-K_M|_{B}$ is nef, then $-K_M$ is nef.
\item If $-K_{M'}$ is big, then $-K_M$ is big.
\item If $-K_{M'}$ is ample and $-K_M|_{B}$ is ample, then $-K_M$ is ample.
\end{enumerate}
\end{lemma}

\begin{proof}
The inequality \eqref{KMinequality} is straightforward from the projection formula. Item ii.\ then follows from \eqref{KMinequality}. For item iii.\ just note that by Hartog's property one has
$$
H^0(M', \sO_{M'}(-dK_{M'})) \subsetneq H^0(M, \sO_{M}(-dK_{M})). 
$$
For iv.\ note that $-K_M$ is nef and big by ii.\ and iii., thus by the basepoint-free theorem $-K_M$ is ample if and only if it has
strictly positive degree on every curve. Yet this is guaranted by the assumption and  
\eqref{KMinequality}.
\end{proof}

\begin{remark} \label{remark-nulllocus}
Let us recall that for a nef line bundle $L$ on a projective manifold $M$ the null locus is defined as
$$
\mbox{\rm Null}(L) = \cup_{Z \subset X,  c_1(L)^{\dim Z} \cdot Z=0} \ Z.
$$
If $L$ is semiample and big the null locus is exactly the exceptional locus of the birational morphism $\varphi_{|kL|}$,
in this case it is exactly the locus covered by curves $C$ such that $c_1(L) \cdot C=0$.
\end{remark}

\subsection{Some surface theory}

For the convenience of the reader we collect some basic facts from the classification theory of projective surfaces.

\begin{lemma} \label{lemma-chinegative}
Let $S$ be a smooth projective surface such that $\chi(S, \sO_S)<0$. Then $S$ is uniruled
and the MRC fibration is a morphism $S \rightarrow C$ onto a curve of genus $g=1+\chi(S, \sO_S)$.\end{lemma}

\begin{proof}
The statement is invariant under MMP so we can assume that $S$ is relatively minimal surface.
If $K_S$ is nef we have $c_2(S) \geq 0$ \cite[VI,Table10]{BHPV04} and therefore $\chi(S, \sO_S) \geq 0$ by Noether's formula.
\end{proof}

\begin{lemma}\text{(Hodge index inequality)} \cite[Prop.2.5.1, Cor.2.5.4]{BS95} \label{hodge-index}
Let $S$ be a normal projective surface and let $A$ and $B$ be $\Q$-Cartier divisors on
$S$. If $A$ is nef and big we have
$$
(A \cdot B)^2 \geq A^2 \cdot B^2
$$
and if equality holds the classes of $A$ and $B$ are colinear in $N^1(S)_\Q$.
\end{lemma}

\begin{lemma} \label{lemma-ramificationformula}
Let $\holom{\varphi_S}{S}{U}$ be a generically finite morphism between normal projective surfaces
with $\Q$-Gorenstein singularities. Then we have
\begin{equation} \label{ramificationformula}
K_B \equiv \varphi_S^* K_U + R
\end{equation}
where $R$ is a $\Q$-divisor on $S$ such that $(\varphi_S)_* R$ is effective. In particular for every nef $\Q$-Cartier divisor $N$ on $U$ we have $R \cdot \varphi_S^* N \geq 0$.

Moreover $(\varphi_S)_* R=0$ if and only if the finite part $\holom{\varphi_{\bar S}}{\bar S}{U}$ of the Stein factorisation is quasi-\'etale.
\end{lemma}

\begin{remark*} We will see in the proof that for a finite morphism 
$\varphi_S$ the divisor $R$ is effective and $K_S \simeq \varphi_S^* K_U + R$.
\end{remark*}

\begin{proof}
Let $\holom{\nu}{S}{\bar S}$ and $\holom{\varphi_{\bar S}}{\bar S}{U}$ be the Stein factorisation of $\varphi_S$.
Since $\varphi_{\bar S}$ is finite, the pull-back $\varphi_{\bar S}^* K_U$ is well-defined as a Weil divisor by Lemma \ref{lemma-pullback-weil}
and $K_{\bar S} \simeq \varphi_{\bar S}^* K_U + \bar R$ with an effective divisor $\bar R$: this is clear on the smooth locus of $\bar S$, the formula extends since the singular locus has codimension at least two. The morphism $\nu$ is birational, so we have
$$
K_S \equiv \nu^* K_{\bar S} + E
$$
where $E$ is a $\nu$-exceptional divisor and $\nu^*K_{\bar S}$ is the Mumford pull-back of the Weil divisor $K_{\bar S}$. Set $R := \nu^* \bar R + E$, then it is clear that the formula holds.
Moreover 
$$
(\varphi_S)_* R = (\varphi_{\bar S})_* \nu_* (\nu^* \bar R + E) =  (\varphi_{\bar S})_*\bar R
$$ 
is effective and equal to zero if and only if $\bar R=0$, i.e.\ when $\varphi_{\bar S}$ is quasi-\'etale.
\end{proof}

\begin{lemma} \label{lemma-RR-rational-sings}
Let $U$ be a normal projective surface with rational singularities. Let $M_U$ be a nef and big Cartier divisor on $U$. Then we have
$$
h^0(U, \sO_U(K_U+M_U)) = \chi(U, \sO_U(K_U+M_U)) =   \frac{1}{2} M_U^2 + \frac{1}{2} K_U \cdot M_U + \chi(U,\sO_U).
$$
\end{lemma}

\begin{proof}
Let $\holom{\nu}{U'}{U}$ be a resolution of singularities, then 
$\chi(U, \sO_U) = \chi(U', \sO_{U'})$ since the singularities are rational. Moreover we have
$\nu_* \omega_{U'} \simeq \omega_U$ and $R^j \nu_* \omega_{U'} = 0$ for $j>0$
by Grauert-Riemenschneider vanishing. Since
$$
R^j \nu_* ( \omega_{U'} \otimes \nu^* M_U) \simeq R^j \nu_* \omega_{U'} \otimes M_U
$$
by \cite[III,Ex.8.1, Ex.8.3]{Har77} we have
$$
H^i(U,  \sO_U(K_U+M_U)) = H^i(U', \sO_{U'}(K_{U'}+\nu^* M_U))
$$
for all $i \in \N$. By  Kawamata-Viehweg vanishing on $U'$
there is no higher cohomology,
in particular $h^0(U, \sO_U(K_U+M_U)) = \chi(U, \sO_U(K_U+M_U))$.
By the Riemann-Roch formula on $U'$ we have
$$
\chi(U', \sO_{U'}(K_{U'}+\nu^* M_U)) = \frac{1}{2} (\nu^* M_U)^2 + \frac{1}{2} K_{U'} \cdot \nu^* M_U + \chi(U', \sO_{U'}),
$$
so the statement follows by the projection formula for divisors.
\end{proof}

The following lemma is a reformulation of results from Fujita's classification 
of varieties of low $\Delta$-genus \cite[Sect.5]{Fuj90}, \cite[Prop.3.1.2]{BS95}.

\begin{lemma} \label{lemma-low-degree}
Let $U$ be a normal projective surface with klt singularities and $p_g(U)=0$, and let $M_U$ be an ample and globally generated Cartier divisor. If $M_U^2 \leq 2$ one of the following holds:
\begin{itemize}
\item $(U, M_U) \simeq (\PP^2, H)$ where $H$ is a hyperplane section;
\item $(U, M_U) \simeq (Q^2, H)$ where $Q^2$ is a normal quadric and $H$ is a hyperplane section; or
\item $(U, M_U) \simeq (U, -K_U)$ and $U$ a del Pezzo surface of degree two with canonical singularities.
\end{itemize}
\end{lemma}

\begin{proof}
Since $M_U$ is ample and globally generated we have $h^0(U, \sO_U(M_U)) \geq 3$. 

If $M_U^2=1$, the global sections of $M_U$ define an isomorphism to $\PP^2$.
If $M_U^2=2$ and $h^0(U, \sO_U(M_U)) > 3$, the global sections of $M_U$ define an isomorphism onto a quadric.

If $M_U^2=2$ and $h^0(U, \sO_U(M_U)) = 3$
the linear system $|M_U|$ defines a finite morphism 
$\holom{g}{U}{\PP^2}$ of degree $2$. The trace morphism \cite[Prop.5.7]{KM98} gives a splitting
$$
g_* \sO_U \simeq \sO_{\PP^2} \oplus \sF
$$
and the rank one sheaf $\sF$ is reflexive since so is $g_* \sO_U$. Thus $\sF$ is invertible
and
$$
g_* \sO_U \simeq \sO_{\PP^2} \oplus \sO_{\PP^2}(-d)
$$
where $d$ is half the degree of the branch divisor (cf.\ \cite[V, Sect.22]{BHPV04}).
Since $p_g(U)=h^2(U, \sO_U)=0$
we get $h^2(\PP^2, \sO_{\PP^2}(-d))=0$ and therefore $d \leq 2$. 
The ramification formula \eqref{ramificationformula} yields
$$
K_U \simeq f^* K_{\PP^2} + R \simeq f^* K_{\PP^2} + f^* (dH) \simeq f^* (3-d) H
$$
If $d=1$ we obtain a quadric and if $d=2$ we obtain a del Pezzo surface of degree two.
Since $K_U$ is Cartier it is clear that $U$ has canonical singularities.
\end{proof}

\begin{lemma} \label{lemma-bound-negativity}
Let $S$ be a normal projective surface admitting a fibration 
$\holom{h}{S}{C}$ onto a smooth curve. Assume that $K_S \equiv h^* M$ with 
$M$ a line bundle on $C$. Then $\deg M \geq -2$.
\end{lemma}

\begin{proof}
We argue by contradiction and assume that $\deg M \leq -3$. Then
$-K_S \equiv b F$ where $F$ is a general $h$-fibre and $b \geq 3$. 
Let $\holom{\tau}{S'}{S}$ be the minimal resolution of $S$, then we have
$$
-K_{S'} \equiv b F + E
$$
where $E$ is an effective $\tau$-exceptional $\Q$-divisor.
Note that $\tau$ is an isomorphism in a neighbourhood of $F$, so we identify the general fibres of $h$ and $h \circ \tau$.

Let $\holom{\psi}{S'}{S_m}$ be a MMP over $C$. Since $K_{S'}$ is trivial on the general
fibre the outcome is a relative minimal model over $C$ and
$$
K_{S_m} = \psi_* K_{S'} = - \psi_* (bF+E) = - b F - \psi_* E
$$
is numerically trivial on every fibre of $\holom{h_m}{S_m}{C}$. 
In particular $\psi_* E$ is $h_m$-nef and by Zariski's lemma \cite[III,Lemma 8.2]{BHPV04}
we obtain $\psi_* E \equiv \lambda F$ with $\lambda \geq 0$, so
$$
-K_{S_m} \equiv (b+\lambda) F.
$$
Since $K_{S_m}$ is not nef and $S_m$ is a surface with Picard number at least two, there exists an elementary Mori contraction $\holom{\eta}{S_m}{T}$ such that $\dim T>0$.
By the classification of surfaces $\eta$ contracts
a rational curve $\Gamma$ such that $-K_{S_m} \cdot \Gamma \leq 2$.
Thus we have
$$
2 \geq -K_{S_m} \cdot \Gamma = (b+\lambda) F \cdot \Gamma \geq 3 F \cdot \Gamma.
$$ 
and therefore $F \cdot \Gamma=0$. Thus $\Gamma$ is contained in a $h_m$-fibre,
a contradiction to $K_{S_m}$ being $h_m$-nef.
\end{proof}

\subsection{Fibrations}

We will need the simplest case of a Weierstra\ss\ model of an elliptic fibration:

\begin{definition} \label{defn:weierstrass}
A flat Weierstra\ss\ fibration is a complex manifold $M$ admitting a flat projective fibration
$$
\holom{\phi}{M}{T}
$$
such that the general fibre is an elliptic curve and there is a section $E \subset M$ that is $\phi$-ample.
We say that $E$ is the distinguished section of the Weierstra\ss\ fibration.
\end{definition}

\begin{proposition} \label{prop:weierstrass} \cite[Thm.2.1]{Nak88}
Let $\holom{\phi}{M}{T}$ be flat Weierstra\ss\ fibration with distinguished section $E$.
Consider $L := \ma N_{E/M}^*$ as a line bundle on $T$. Set
$W:= \sO_T \oplus L^{-2} \oplus L^{-3}$ and denote by $\zeta_W$ the tautological class
on the projectivised bundle
$$
\holom{\Phi}{\PP(W)}{T}.
$$ 
Then we have an embedding $\holom{j}{M}{\PP(W)}$ such that $j^* \sO_{\PP(W)}(1) \simeq
\sO_M(3E)$ and $M$ is an element of the linear system $|3 \zeta_W + \Phi^* (6L)|$.
\end{proposition}

\begin{proof}
Nakayama shows all the properties in \cite[Thm.2.1]{Nak88} except that $j$ is only a birational morphism such that $j^* \sO_{\PP(W)}(1) \simeq
\sO_M(3E)$. Since we assume that $E$ is $\phi$-ample the morphism does not contract any curves and is therefore an isomorphism onto its image.
\end{proof}

\begin{fact} \label{fact:weierstrass}
In the situation of Proposition \ref{prop:weierstrass} we have $\phi_* \sO_{M}(E) \simeq \sO_{T}$ and 
$$
\ma N_{E/M} \simeq R^1 \phi_* \sO_{M} \simeq \phi_* (\omega_{M/T})^*.
$$
\end{fact}

\begin{proof}
This is part of the construction of the Weierstra\ss\ model, we recall the argument for the convenience of the reader:
the section $E \subset M$ is $\phi$-ample,
so  $R^1 \phi_* \sO_{M}(E)=0$ by the relative Kodaira vanishing theorem.
Pushing forward the exact sequence
$$
0 \rightarrow \sO_{M} \rightarrow \sO_{M}(E) \rightarrow \sO_E(E) \simeq \ma N_{E/M} \rightarrow 0.
$$
we obtain 
$$
0 \rightarrow \sO_{T} \rightarrow \phi_* \sO_{M}(E)  \rightarrow \ma N_{E/M} \rightarrow R^1 \phi_* \sO_{M} \rightarrow 0
$$
The last two sheaves are invertible, since the morphism is surjective it is an isomorphism. Therefore the injection $\sO_{T} \rightarrow \phi_* \sO_{M}(E)$ is also an isomorphism.

For the last isomorphism note that $\phi$ is a flat Gorenstein morphism, so we have  $R^1 \phi_* \sO_{M} \simeq \phi_* (\omega_{M/T})^*$ by relative Serre duality 
\cite[Thm.21]{Kle80}. 
\end{proof}

\begin{lemma} \label{lemma-factorisation-degree-one}
Let $M$ be a normal projective $\Q$-factorial variety with a fibration
$\holom{f}{M}{\Sigma}$ onto a normal projective surface. Let $F$ be a general fibre 
and set $d:=\dim F$. 
Assume that there exists
an $f$-ample Cartier divisor $L$ such that $L^d \cdot F=1$. 

Then there exists
an equidimensional fibration $\holom{g}{M}{\Sigma'}$ with integral fibres and a birational morphism $\holom{\nu}{\Sigma'}{\Sigma}$
such that $f = \nu \circ g$.

Assume moreover that $f$ admits a section $\holom{\sigma}{\Sigma}{M}$. Then $f$ is equidimensional. 
\end{lemma}

\begin{remark*}
If $M$ and the section $\sigma(\Sigma)$ are smooth, the equidimensional fibration
$f$ is flat \cite[IV, Ex.10.9]{Har77}. 
\end{remark*}

\begin{proof}
Let us first show how the first statement implies the second one: the composition 
$\holom{g \circ \sigma}{\Sigma}{\Sigma'}$ is an inverse to the birational map $\nu$. Thus $\nu$ is an isomorphism and $f = \sigma \circ g$ is equidimensional.

Let us now show the first statement. 
The general fibres of $f$ determine a unique irreducible component of the cycle
space $\Chow{M}$, and we denote by $\Sigma'$ its normalisation.
Let $\holom{\phi}{M'}{\Sigma'}$
be the normalisation of the universal family over the component and denote by $\holom{\mu}{M'}{M}$ the birational map given by the evaluation. The composition $f \circ \mu$ contracts every $\phi$-fibre, so
by the rigidity lemma there exists a birational map
$\holom{\nu}{\Sigma'}{\Sigma}$ such that $f \circ \mu = \nu \circ \phi$.
In summary we obtain a commutative diagram
$$
\xymatrix{
M' \ar[d]_{\phi} \ar[r]^{\mu} & M \ar[d]^f \\
\Sigma' \ar[r]^{\nu} & \Sigma 
}
$$
Note that $\mu$ is finite when 
restricted to a $\phi$-fibre so the pull-back $\mu^* L$ is an $\phi$-ample Cartier divisor.
Since $(\mu^* L)^d \cdot F=1$
all the cycle theoretic fibres (in the sense of \cite[I,Defn.3.10]{Kol96})  of $\phi$ are irreducible and reduced, i.e.\ integral. 

We will show that $\mu$ is an isomorphism, then $g:=\phi \circ \mu^{-1}$ provides the equidimensional factorisation.
Arguing by contradiction, assume that $\mu$ is not an isomorphism.
Since $X$ is $\Q$-factorial the exceptional locus has pure codimension one,
so there exists a prime divisor 
$D \subset M'$  that is $\mu$-exceptional.
Since $\mu$ is an isomorphism near a general $f$-fibre, the divisor $D$ is disjoint from the general $\phi$-fibre. The fibration $\phi$ being equidimensional we obtain that $\phi(D) \subset \Sigma'$ is an irreducible curve. Since $\phi$ has irreducible fibres we have
$D = \phi^* \phi(D)$ (cf. Lemma \ref{lemma-pullback-weil} for pull-back of Weil divisors in this case). 
Since $D$ is $\mu$-exceptional,
its image $\mu(D) \subset M$ has codimension at least two. Yet the evaluation map $\mu$ is finite when restricted to a $\phi$-fibre, so $\mu(D)$ has dimension at least $d=n-2$.
Thus we obtain that for every $t \in \phi(D)$, the cycle parametrised by $t$ is the same variety $\mu(D)$. Yet $\Sigma'$ is the normalisation of an irreducible component of $\Chow{M}$ so there are at most finitely many points parametrising the same cycle, a contradiction.
\end{proof}

\section{Elliptic Calabi-Yau threefolds}

Linear systems on K3 surfaces with non-empty base locus are completely understood by Mayer's theorem,
for linear system on Calabi-Yau threefolds our knowledge is much more limited.
Theorem \ref{thm-birational} provides a first step for the classification of fixed divisors
in the presence of an elliptic fibration, it will also be the starting point
of our investigation of $\varphi_{|-K_{X'}|}$ in Section \ref{section-structure-result}.

Before starting the proof of Theorem \ref{thm-birational} let us recall 
some properties of the total space and base of an elliptic fibration.

\begin{remarks} \label{remarks-thm-birational}
Let $Z$ be a normal $\Q$-factorial projective threefold with terminal singularities that is Calabi-Yau, i.e.\ $K_Z$ is trivial and $h^1(Z, \sO_Z)=h^2(Z, \sO_Z)=0$.
\begin{itemize}
\item Since $\sO_Z(K_Z) \simeq \sO_Z$, the threefold $Z$ is Gorenstein and therefore it is even factorial by \cite[Lemma 5.1]{Kaw88}.
\item By the canonical bundle formula, there exists a boundary divisor $\Delta$ on $U$ such that
$(U, \Delta)$ is klt and $K_Z \sim_\Q \varphi_Z^* (K_U + \Delta)$. Therefore, $(U, 0)$ is klt.
Arguing as in \cite[Thm.3.1]{Ogu93} it is not difficult to see that $U$ is rationally connected and 
$h^1(U, \sO_U)=h^2(U, \sO_U)=0$,
in particular we have $\chi(U, \sO_U)=1$.
\end{itemize}
\end{remarks}

\begin{proof}[Proof of Theorem \ref{thm-birational}]
The proof requires a number of case distinctions, we start by establishing the setup for all cases. Since $A$ is nef and big and $|A| = S + |M|$, the divisor $S$ is generically finite over $U$. We set 
$$
\holom{\varphi_S:= \varphi_Z|_S}{S}{U}
$$
and denote the degree of this morphism by $d \in \N$.

{\em Step 1. We have an isomorphism
\begin{equation}
\label{the-injection}
H^0(S, \sO_S(K_S+M_S)) \simeq H^0(U, \sO_U(K_U+M_U)).
\end{equation}
and an equality
\begin{equation}
\label{the-equality}
\frac{1}{2} (K_S+M_S) \cdot M_S + \chi(S, \sO_S) =  \frac{1}{2} (K_U+M_U) \cdot M_U + 1
\end{equation}
where $M_S := \varphi_S^* M_U$.
}

Since $U$ has klt singularities we have a natural pull-back morphism $\varphi_S^{[*]} \omega_U \rightarrow \omega_S$
\cite[Thm.4.3]{GKKP11}. The sheaf $\sO_U(M_U)$ is locally free, so this induces an injection
$$
\varphi_S^* : H^0(U, \sO_U(K_U+M_U))
\rightarrow 
H^0(S, \sO_S(K_S+\varphi_S^* M_U)) \simeq
H^0(S, \sO_S(K_S+M_S)).
$$
Thus we are left to construct the other inclusion.

Consider the exact sequence
$$
0 \rightarrow \sO_Z(\varphi_Z^* M_U) \rightarrow \sO_Z(A) \rightarrow \sO_S(A) \rightarrow 0.
$$
By Kawamata-Viehweg vanishing we have $H^1(Z, \sO_Z(A))=0$ and therefore 
a long exact sequence in cohomology
$$
0 \rightarrow H^0(Z, \sO_Z(\varphi_Z^* M_U)) \rightarrow H^0(Z, \sO_Z(A)) \rightarrow H^0(S, \sO_S(A)) \rightarrow H^1(Z, \sO_Z(\varphi_Z^* M_U)) \rightarrow 0
$$
 Since $S$ is in the base locus of $|A|$ the restriction map $H^0(Z, \sO_Z(A)) \rightarrow H^0(S, \sO_S(A))$ is zero, so we have an isomorphism
\begin{equation} \label{hzerotoh1}
H^0(S, \sO_S(A)) \simeq H^1(Z, \sO_Z(\varphi_Z^* M_U)).
\end{equation}
By the Leray spectral sequence we have an exact sequence
$$
0 \rightarrow H^1(U, \sO_U(M_U)) \rightarrow H^1(Z, \sO_Z(\varphi_Z^* M_U))
\rightarrow H^0(U, \sO_U(M_U) \otimes R^1 (\varphi_Z)_* \sO_Z) .
$$
Since $\omega_Z:=\sO_Z(K_Z) \simeq \sO_Z$ we have
$$
H^1(U, \sO_U(M_U)) \simeq H^1(U, (\varphi_Z)_* \omega_Z \otimes \sO_U(M_U)) = 0
$$
by Koll\'ar's vanishing theorem \cite[Thm.2.1]{Kol86} (applied on a resolution of $Z$). 
Moreover $R^1 (\varphi_Z)_* \sO_Z \simeq R^1 (\varphi_Z)_* \omega_Z$ is torsion-free and by relative duality 
$$
(R^1 (\varphi_Z)_* \sO_Z)^* \simeq (\varphi_Z)_* \omega_{Z/U} \simeq \omega_U^*
$$
where for the last isomorphism we used again $\omega_Z \simeq \sO_Z$. In particular we get an injection $R^1 (\varphi_Z)_* \sO_Z \hookrightarrow \omega_U$.
Thus we see that \eqref{hzerotoh1} yields an injection
$$
H^0(S, \sO_S(A)) \hookrightarrow H^0(U, \sO_U(K_U+M_U)).
$$
Actually $\omega_S \simeq \sO_S(S)$ by adjunction, so we can rewrite this as
$$
H^0(S,\sO_S(K_S+M_S)) \hookrightarrow H^0(U, \sO_U(K_U+M_U)).
$$
Therefore, both inclusions hold and we have \eqref{the-injection}.

By Remark \ref{remarks-thm-birational} we know that $U$ has rational singularities and
$\chi(U,\sO_U)=1$, so the Lemma \ref{lemma-RR-rational-sings} yields
$$
h^0(U, \sO_U(K_U+M_U)) = \frac{1}{2} M_U^2 + \frac{1}{2} K_U \cdot M_U + 1.
$$
Since $M_S$ is nef and big we have
by Kawamata-Viehweg and Riemann-Roch on the smooth surface $S$
$$
h^0(S, \sO_S(K_S+M_S))=\chi(S, \sO_S(K_S+M_S)) = 
\frac{1}{2} M_S^2 + \frac{1}{2} K_S \cdot M_S + \chi(S, \sO_S).
$$
Therefore \eqref{the-injection} implies \eqref{the-equality}.

{\bf We will now argue by contradiction and assume that $\boldsymbol{d>1}$.
Let 
$$\boldsymbol{
\holom{\nu}{S}{\bar S}, \qquad \holom{\varphi_{\bar S}}{\bar S}{U}
}$$ 
be the Stein factorisation of $\boldsymbol{\varphi_S}$. Set $\boldsymbol{M_{\bar S} := \varphi_{\bar S}^* M_U \simeq \nu_* M_S}$.
}

{\em Step 2.
The divisor $K_U+M_U$ is effective. Moreover 
we have 
\begin{equation} \label{the-inequality-ramified}
\frac{d}{2} (K_U+M_U) \cdot M_U + \chi(S, \sO_S) \leq  \frac{1}{2} (K_U+M_U) \cdot M_U + 1
\end{equation}
and if equality holds the morphism $\varphi_{\bar S}$ is quasi-\'etale.
}

Since $A|_S = K_S+M_S$ is nef and $M_S$ is nef and big we know by effective nonvanishing on smooth surfaces \cite[Thm.3.1]{Kaw00} that $h^0(S, \sO_S(K_S+M_S)) \neq 0$. By \eqref{the-injection} this implies that
$K_U + M_U$ is also effective.

By the ramification formula \eqref{ramificationformula} we have $K_S+M_S \equiv f^* (K_U+M_U)+R$ with $R$ a $\Q$-divisor such that $\nu_* R$ is the ramification divisor of $\varphi_{\bar S}$.
Thus we have 
$$
\frac{1}{2} (K_S+M_S) \cdot M_S + \chi(S, \sO_S) 
= 
\frac{d}{2} (K_U+M_U) \cdot M_U + R \cdot M_S + \chi(S, \sO_S). 
$$
Since $M_U$ is ample and $R \cdot M_S = \nu_* R \cdot \varphi_{\bar S}^* M_U$ we see that the intersection number is non-negative and equal to zero if and only if $\nu_* R=0$, if and only if $\varphi_{\bar S}$ is quasi-étale by Lemma \ref{lemma-ramificationformula}.
Plugging this into \eqref{the-equality} we obtain \eqref{the-inequality-ramified}.

{\em Step 3. We exclude the case $\chi(S, \sO_S) \geq 2$.
}
In this case the inequality  \eqref{the-inequality-ramified} simplifies to
$$
\frac{d}{2} (K_U+M_U) \cdot M_U  <  \frac{1}{2} (K_U+M_U) \cdot M_U. 
$$
Since $d>1$ we deduce that $(K_U+M_U) \cdot M_U<0$, a contradiction to the effectivity of $K_U+M_U$.

{\em Step 4. We exclude the case $\chi(S, \sO_S) = 1$.
}
In this case the inequality  \eqref{the-inequality-ramified} simplifies to
$$
\frac{d}{2} (K_U+M_U) \cdot M_U  \leq  \frac{1}{2} (K_U+M_U) \cdot M_U. 
$$
If the inequality is strict, we obtain a contradiction as in Step 3. If equality holds,
$d>1$ implies that $(K_U+M_U) \cdot M_U=0$, so the effective divisor $K_U+M_U$ is zero.
Moreover $K_{\bar S}+M_{\bar S}$ is zero since $\varphi_{\bar S}$ is quasi-\'etale by Step 2. Yet
$K_{\bar S}+M_{\bar S} = \nu_* (K_S+M_S) = \nu_* A|_S$ 
and $A|_S$ is nef and big by item ii.\ of our assumption. Thus its 
push-forward is not zero.

{\em Step 5. We exclude the case $\chi(S, \sO_S) < 0$.
}

By Lemma \ref{lemma-chinegative} the surface
$S$ admits an MRC fibration $\holom{\alpha}{S}{C}$
onto a smooth curve $C$ of genus $g \geq 2$ such that the general fibre $F \simeq \PP^1$. 
We set $l := M_S \cdot F$ and note that $l \geq 3$ since $-2+l=(K_S+M_S) \cdot F = A|_S \cdot F>0$. 
We claim that 
$$
d=2,\, g=2,\, l=3\quad\text{and}\quad h^0(U, \sO_U(K_U+M_U))=h^0(S, K_S+M_S)=3.
$$
{\em Proof of the claim.}
Set $m:=h^0(S, K_S+M_S)$. By Lemma \ref{lemma-RR-rational-sings}, with this notation the equality \eqref{the-equality} is 
$$
m  = \frac{1}{2} (K_U+M_U) \cdot M_U + 1
$$
and
$$
\frac{1}{2} (K_S+M_S) \cdot M_S = m+g-1.
$$
By the ramification formula \eqref{ramificationformula} we have
\begin{equation*}
\begin{aligned}
(K_U+M_U) \cdot M_U = & \frac{1}{d} f^* (K_U+M_U) \cdot f^* M_U \leq \\
\leq &\frac{1}{d} f^* (K_U+M_U) \cdot f^* M_U
+
\frac{1}{d} R \cdot  f^* M_U
=
\frac{1}{d} 
(K_S+M_S) \cdot M_S.
\end{aligned}
\end{equation*}
Thus we have
$$
m = \frac{1}{2} (K_U+M_U) \cdot M_U + 1
\leq \frac{1}{2d} 
(K_S+M_S) \cdot M_S
 + 1 = \frac{1}{d} \left(m+g-1\right) + 1 
$$
An elementary computation shows that
\begin{equation}
\label{elementary}
m (d-1) \leq g-1+ d.
\end{equation}
The direct image sheaf $\alpha_* \sO_S(K_S+M_S)$ has rank $l-1$, so
by Riemann-Roch on the curve $C$ we have
$$
m = h^0(C, \alpha_* \sO_S(K_S+M_S)) = c_1(\alpha_* \sO_S(K_S+M_S)) + (l-1)
(1-g).
$$
Since $\alpha_* \sO_S(K_S+M_S) \simeq \alpha_* \sO_S(K_{S/C}+M_S) \otimes \omega_C$
and $\alpha_* \sO_S(K_{S/C}+M_S)$ is ample by \cite[Cor.3.7]{Vie01} we obtain
\begin{equation}
\label{eqn2}
m = c_1(\alpha_* \sO_S(K_{S/C}+M_S)) + (l-1) (g-1) > (l-1) (g-1).
\end{equation}
Combined with \eqref{elementary} we obtain
\begin{equation}
\label{inequ1}
(l-1) (d-1) < 1 + \frac{d}{g-1}.
\end{equation}

Since $g \geq 2$ we obtain $(l-1) (d-1) < 1 + d$.
Then $l \geq 3$ implies $d<3$ and therefore $d=2$. Plugging this into \eqref{inequ1} we obtain $g=2$ and $l=3$. 
From \eqref{elementary} we have $m \leq 3$
and by \eqref{eqn2} we have $m >2$. Thus $m=3$ and \eqref{the-injection}
yields the last claim. 

\smallskip

Note that $K_{S/C}+M_S$ is nef
and big, since it is the pull-back of the tautological bundle
on $\PP(\alpha_* \sO_S(K_{S/C}+M_S))$. Moreover, $\alpha^* K_C\simeq2F$ yields
$$
8 =(K_S+M_S) \cdot M_S = (K_{S/C}+M_S) \cdot M_S + \alpha^* K_C\cdot M_S =
(K_{S/C}+M_S) \cdot M_S + 6.
$$
Thus, we get $(K_{S/C}+M_S) \cdot M_S=2$ and by the Hodge index inequality (Lemma \ref{hodge-index}) 
$$
4 \geq (K_{S/C}+M_S)^2 \cdot M_S^2.
$$
The first term is at least one, so we get $M_S^2 \leq 4$. Yet
$M_S^2 = 2 M_U^2$ and therefore $M_U^2 \leq 2$.
The polarised varieties $(U, M_U)$ with $M_U^2 \leq 2$ are classified
in Lemma \ref{lemma-low-degree}. It is easy to see that in none of the
cases $h^0(U, \sO_U(K_U+M_U))=3$.

{\em Step 6. We exclude the case $\chi(S, \sO_S) = 0$, thereby finishing the proof.
}

Set $m:=h^0(S, K_S+M_S)$. By Lemma \ref{lemma-RR-rational-sings}, with this notation the equality \eqref{the-equality} is
\begin{equation}
\label{easyRR}
\frac{1}{2} (K_S+M_S) \cdot M_S = m = \frac{1}{2} (K_U+M_U) \cdot M_U + 1. 
\end{equation}
As in Step 5 we use
the ramification formula \eqref{ramificationformula} to obtain
$$
m = \frac{1}{2} (K_U+M_U) \cdot M_U + 1
\leq \frac{1}{2d}
(K_S+M_S) \cdot M_S + 1 = \frac{m}{d} + 1 
$$
An elementary computation shows that
$$
m  \leq \frac{d}{d-1}.
$$
Since $d \geq 2$ we obtain $m \leq 2$ and even $m=1$ if $d>2$.
We make a case distinction:

{\em 1st case: Assume that $m=1$.}
By \eqref{easyRR} we obtain $(K_U+M_U) \cdot M_U=0$, so the effective divisor $K_U+M_U$
is zero. 

Moreover \eqref{easyRR} implies $(K_S+M_S) \cdot M_S=2$. Since $K_S+M_S$ is nef and big and $M_S$ is nef and big we get from Hodge index inequality that $M_S^2 \leq 4$.
Since $M_S^2 = d M_U^2$ 
we obtain $M_U^2 \leq 2$. Applying the classification from Lemma \ref{lemma-low-degree} and using $h^0(U, \sO_U(K_U+M_U))=1$
we see that $U$ is a del Pezzo surface of degree two, so $M_U^2=2$

Thus we have $M_S^2=4$
and $(K_S+M_S)^2=1$. Even more, since we have equality in the Hodge index inequality
we obtain that $K_S+M_S$ and $M_S$ are colinear in the Neron-Severi space \cite[Cor.2.5.4]{BS95}. It now follows easily that $K_S \equiv - \frac{1}{2} M_S$, so $S$ is weak Fano.
This contradicts $\chi(S, \sO_S)=0$.

{\em 2nd case: $m=2$.} This case is somewhat lengthy: we will compute various invariants in order to obtain a classification of the possibilities for $S$ and $U$, then derive a contradiction.

As seen above $m=2$ implies $d=2$ and by \eqref{easyRR} we obtain
$$
(K_S+M_S) \cdot M_S=4, \qquad (K_U+M_U) \cdot M_U=2.
$$
Moreover equality holds in \eqref{the-inequality-ramified} so the degree
two morphism $\holom{\varphi_{\bar S}}{\bar S}{U}$ is quasi-\'etale by Step 2.
If $M_U^2 \leq 2$ we can use the classification in Lemma \ref{lemma-low-degree}
to obtain a contradiction to $h^0(U, \sO_U(K_U+M_U)) = 2$. Thus we can assume from now
$$
M_U^2 \geq 3, \qquad M_S^2 = 2 M_U^2 \geq 6.
$$
From  $(K_U+M_U) \cdot M_U=2$ we deduce that $K_U \cdot M_U \leq -1$, 
so $K_U$ is not pseudoeffective. Since $\varphi_{\bar S}$ is quasi-\'etale we obtain that
$K_{\bar S}$ is not pseudoeffective, so $K_S$ is not pseudoeffective. Thus $q(S)=1$ and
the classification of surfaces shows that a minimal model $\holom{\tau}{S}{S_m}$
is a ruled surface
$$
\holom{\psi}{S_m}{E}
$$
over an elliptic curve $E$. In particular we have $K_{S_m}^2=0$ and therefore $K_S^2 \leq 0$.

{\em We show that $M_S^2=6$ and $K_S^2\in\{-1,0\}$.}
We know that $M_S^2=2M_U^2$ is an even integer, so arguing by contradiction we assume that
$M_S^2\geq 8$. Since $(K_S+M_S) \cdot M_S=4$, we then have $-K_S\cdot M_S \geq 4$.
Moreover, from
$$
1\leq (K_S+M_S)^2=K_S^2+(K_S+M_S) \cdot M_S + K_S\cdot M_S
$$
we get $K_S^2 \geq 1$. Yet we have just seen that $K_S^2 \leq 0$, so we have a contradiction.
Thus $M_S^2=6$ and therefore $K_S\cdot M_S=-2$. Now the inequality
$1\leq (K_S+M_S)^2$ implies $K_S^2 \geq -1$.

{\em We show that $\bar S$ is a ruled surface.} Since
$S_m$ is a ruled surface over an elliptic curve and $K_S^2\in\{-1,0\}$ we either have
$S \simeq S_m$ or $S \rightarrow S_m$ is the blow-up in one point. 
Recall now that $\bar S \rightarrow U$ is quasi-\'etale, so $\bar S$ is also klt \cite[Prop.5.20]{KM98}. Thus
the birational morphism $S \rightarrow \bar S$ only contracts rational curves.
Yet all the exceptional rational curves (if any exist) on $S$ are $(-1)$-curves, so $S \rightarrow \bar S$
is an isomorphism or a blow-up of a smooth point. Thus $\bar S$ is smooth and
$K_{\bar S}^2 \in \{ -1, 0\}$.

The double cover $\bar S \rightarrow U$ is Galois, so
the surface $U$ is the quotient of the surface $\bar S$ by the covering involution. Yet $\bar S$ is smooth, 
so $U$ has at most $A_1$-singularities. In particular $U$ is Gorenstein and $K_U^2$ is an integer. Thus $K_{\bar S}^2 = 2 K_U^2$ is an even integer. 
By the preceding paragraph we finally obtain $K_{\bar S}^2=0$, so
$\bar S$ is a ruled surface over an elliptic curve.

{\em We show that $U$ is a conic bundle $U \rightarrow \PP^1$ with at least one double fibre.}
Recall that the double cover $\holom{\varphi_{\bar S}}{\bar S}{U}$ is quasi-\'etale. It is not \'etale since
otherwise $\chi(\bar S, \sO_{\bar S})=0$ and $\chi(U,\sO_U)=1$ gives a contradiction. Moreover,
$\rho(U)>1$ since otherwise $U$ would be Fano (recall that $U$ is rationally connected with canonical singularities). Yet then its quasi-\'etale cover $\bar S$ is also Fano,
again a contradiction to $\chi(\bar S, \sO_{\bar S})=0$. 
Since $\rho(\bar S)=2$ we  must have $\rho(U)=2$ and therefore
an isomorphism between $N^1(U) \rightarrow N^1(\bar S)$. Thus if $F$ is a fibre of the ruling
$\bar S \rightarrow E$ it is a pull-back of nef and non-big class on $U$. Since $-K_U$ is positive on this class we obtain that $U$ is a conic bundle $U \rightarrow \PP^1$ and the ruling $\bar S \rightarrow E$ is obtained by base change.
Since $\rho(U)=2$ all the fibres of $U \rightarrow \PP^1$ are irreducible, yet there is at least one singular fibre since otherwise $U$ is smooth and therefore $\varphi_{\bar S}$ \'etale which we excluded.

{\em We compute the degree of $K_{\bar S}+M_{\bar S}$ on the fibres of the ruling.}
Recall that $A|_S = K_S + M_S$ is nef and big, so its push-forward
$K_{\bar S} + M_{\bar S} = \varphi_{\bar S}^* (K_U+M_U)$ is nef and big. Since $U$ is Gorenstein we obtain that $K_U+M_U$ is a Cartier divisor that is relatively ample for the Mori fibre space
$U \rightarrow \PP^1$. Since there is at least one singular fibre the divisor has at least degree two on the general fibre.

The $\PP^1$-bundle $\bar S \rightarrow E$ is obtained by base change from $U \rightarrow \PP^1$ so we finally obtain
$$
(K_{\bar S} + M_{\bar S}) \cdot F \geq 2.
$$ 

{\em We conclude by a computation on the ruled surface $\bar S$.}
We write $\bar S \simeq \PP(V)$ where $V$ a rank two vector bundle that is normalised in the sense of \cite[V, Prop.2.8]{Har77}.
We write
$$
M_{\bar S} \equiv \mu C_0 + \lambda F
$$
where $C_0$ is a section with $C_0^2=e \in \Z$ and $\mu, \lambda \in \Z$.
Note that $(K_{\bar S} + M_{\bar S}) \cdot F \geq 2$ is equivalent to $\mu \geq 4$.
Recall that $M_{\bar S}^2=M_S^2=6$, so 
$$
6 = (\mu C_0 + \lambda F)^2 = \mu^2 e + 2 \mu \lambda = \mu (\mu e + 2 \lambda).
$$
Since $\mu \geq 4$ and $\mu e + 2 \lambda$ is an integer the only possibility is
$\mu=6$ and $\mu e + 2 \lambda=1$. Yet then $1=\mu e + 2 \lambda = 6e+2 \lambda$ is even,
the final contradiction.
\end{proof}

In some special cases we can show a rational section to be regular:

\begin{proposition} \label{proposition-section}
Let $Z$ be a smooth Calabi-Yau threefold admitting a fibration 
$$
\holom{\varphi_Z}{Z}{U}
$$
onto a normal projective surface $U$. Assume that 
there exists a smooth prime divisor $S \subset Z$ that is a rational section for $\varphi_Z$ and $S$ is $\varphi_Z$-ample. 
Then $\varphi_Z$ is flat and $S$ is a section for $\varphi_Z$.
\end{proposition}

\begin{remark*}
While Theorem \ref{thm-birational} is likely to be true without the assumption that $S$
is smooth, this is not the case for Proposition \ref{proposition-section}: elliptic fibrations with rational, non-regular sections provide immediate counterexamples.
\end{remark*}

\begin{proof}[Proof of Proposition \ref{proposition-section}]
Since $S$ is $\varphi_Z$-ample and a rational section the fibration $\holom{\varphi_Z}{Z}{U}$
satisfies the conditions of Lemma \ref{lemma-factorisation-degree-one}.
Thus we have an equidimensional fibration $\holom{\phi_Z}{Z}{U'}$
and a birational map
$\holom{\nu}{U'}{U}$ such that $\varphi_Z = \nu \circ \phi_Z$.

{\em Step 1. We show that $U'$ is smooth and $S$ is a $\phi_Z$-section.} 
Let $t \in U'$ be an arbitrary point. 
Since $S$ is relatively ample of degree one the cycle theoretic fibre $\phi_Z^{[-1]}(t)$ is reduced, in particular it has a smooth point $y$.
Thus the fibration $\phi_Z$ is smooth in the point $y$ by \cite[I, Thm.6.5]{Kol96}, in particular it is flat. Yet $Z$ is smooth near $y$, so $U'$ is smooth near $t=\phi_Z(y)$. 

Since $U'$ is smooth the equidimensional map $\phi_Z$ is flat, in particular all the (scheme-theoretic) fibres are integral curves of arithmetic genus one. Assume now that the birational morphism $S \rightarrow U'$ is not an isomorphism over a point $t \in U'$, then $S$ contains the fibre $\fibre{\phi_Z}{t}$, a curve of arithmetic genus one. Yet $S \rightarrow U'$ is a birational map between smooth surfaces so the exceptional locus is the union of smooth rational curves by Castelnuovo's theorem, a contradiction.

{\em Step 2. We show that $\nu$ is an isomorphism.} 
Since $Z$ is Calabi-Yau the fibration $\varphi_Z$ is $K$-trivial and we know by Ambro's theorem \cite[Thm.0.2]{Amb05} that there exists a boundary divisor $\Delta$ on $U$ such that $(U, \Delta)$ is klt. In particular $U$ has rational singularities,
and the morphism $\nu$ is a resolution of singularities. Thus $R^1 \nu_* \sO_U=0$ which implies that every exceptional curve
is a smooth rational curve \cite[Cor.4.9]{Rei97}.

Arguing by contradiction let $C \simeq \PP^1$ be a $\nu$-exceptional curve.
We claim that $C^2 \leq -3$, i.e. $C$ is not a (-1) or (-2)-curve. In fact since $S \simeq U'$ the morphism $\nu$ identifies
to the birational map
$$
\holom{\varphi_Z|_S}{S}{U}.
$$ 
By assumption $S$ is $\varphi_Z$-ample so $K_S = (K_Z+S)|_S = S|_S$ is $\varphi_Z|_S$-ample. Thus we have $K_S \cdot C>0$ and therefore
$C^2 \leq -3$ by the adjunction formula.

Set now $\Sigma := \phi_Z^* C$ and consider the fibration $\holom{\phi_\Sigma}{\Sigma}{C}$. 

{\em 1st case.} If the general $\phi_\Sigma$-fibre is smooth, it is an elliptic curve. Since all the $\phi_Z$-fibres are reduced, the singular locus of $\Sigma$ is finite. 
The surface $\Sigma \subset Z$ is also Gorenstein, so it is normal and by adjunction
$$
K_\Sigma \simeq (K_Z+\Sigma)|_\Sigma = \Sigma|_\Sigma = (\phi_\Sigma)^* C|_C \simeq (\phi_\Sigma)^* \sO_{\PP^1}(-b)
$$
with $-b \leq -3$. This contradicts Lemma \ref{lemma-bound-negativity}.

{\em 2nd case.} If the general $\phi_\Sigma$-fibre is singular, it is a nodal or cuspidal cubic.
Thus $\phi_\Sigma$ has a section $C' \subset \Sigma$ determined by the unique singular point of the fibre.  Since $S$ is a $\phi_Z$-section and Cartier divisor in $Z$, it meets every $\phi_Z$-fibre in a smooth point. Thus we have $S \cdot C'=0$. 
Now recall that $\varphi_Z(\Sigma)=\nu(\phi_Z(\Sigma))=\nu(C)$ is a point.
By assumption the restriction of $S$ to the fibre component $\Sigma$ is ample, so
$S \cdot C'=0$ gives the final contradiction.
\end{proof}

\section{A structure result} \label{section-structure-result}

The goal of this section is to show Theorem \ref{theorem-nefdivisor}. 
We fix the notation for the whole section:

\begin{setup} \label{setup-general} {\rm
Let $X$ be a smooth Fano fourfold such that $h^0(X, \sO_X(-K_X)) \geq 4$.
Assume that the scheme-theoretic base locus $\BsAX$ is a smooth surface $B$.
Let
$$
\holom{\mu}{X'}{X}
$$
be the blow-up along $B$, and denote by $E$ the exceptional divisor. We have
$$
-K_{X'} \simeq \mu^*(-K_X) - E, 
$$
so
$$
H^0(X', \sO_{X'}(-K_{X'})) \simeq H^0(X, \sI_B \otimes \sO_X(-K_X)) \simeq H^0(X, \sO_X(-K_X))
$$
and $|-K_{X'}|$ is basepoint-free. We denote by
$$
\holom{\varphi:=\varphi_{|-K_{X'}|}}{X'}{T'}
$$
the Stein factorisation of the morphism defined by $|-K_{X'}|$.
By construction we have
\begin{equation}
\label{pullbackmu}
\mu^*(-K_X) \simeq -K_{X'} + E \simeq \varphi^* H + E
\end{equation}
where $H$ is an ample and globally generated Cartier divisor on $T'$.

Let $Y \in \AX$ be a general element, and let $Y' \subset X'$ be its strict transform.
Then 
$$
\holom{\mu_Y}{Y'}{Y}
$$
is a resolution of singularities and not an isomorphism by \cite[Thm.1.3]{HS25}.
In fact, since $Y$ has isolated singularities \cite[Thm.1.7]{HV11} and the fibres of the blow-up $\mu$ have dimension at most one, the resolution $\mu_Y$ is small.

The surface $B \subset Y$ is a prime Weil divisor, denote by 
$B'  \subset Y'$ its strict transform. Since $B' = E \cap Y'$ and $Y'$ is general
the surface $B'$ is smooth and the blow-up of the finite set $Y_{\sing} \subset B$.

On the normal threefold $Y$ we have a decomposition
$$
|-K_X|_Y| \simeq |M| + B
$$
where $M$ is the mobile part of the linear system. 
Note that 
$$
h^0(Y, \sO_Y(M)) = h^0(X, \sO_X(-K_X))-1 \geq 3,
$$
so $M \neq 0$.
Since $-K_X|_Y$ is Cartier and $\mu_Y$ is small, its strict transform on $Y'$ coincides with the pull-back. Thus we have
\begin{equation}
\label{onY}
\mu^*(-K_X)|_{Y'} \simeq M' + B'  \simeq -K_{X'}|_{Y'} + B' \simeq (\varphi^* H)|_{Y'} + B'
\end{equation}
where $|M'|$ is the strict transform of the mobile part $|M|$. 

We have an induced fibration
$$
\holom{\varphi_{Y'}}{Y'}{U}
$$
where $U$ is an general element of $|H|$. Since $H^1(X', \sO_{X'})=H^1(X, \sO_X)=0$ the restriction morphism
$H^0(X', \sO_{X'}(-K_{X'})) \rightarrow H^0(Y', \sO_{Y'}(-K_{X'}))$ is surjective, so
$\varphi_{Y'}$ is defined by the complete linear system $|-K_{X'}|_{Y'}|=|M'|$.
}
\end{setup}

\begin{remark*}
Heuberger \cite{Heu15} has shown that the general anticanonical divisor only 
has singularities of the form $x_1^2+x_2^2 + x_3^2 +x_4^2=0$ or
$x_1^2+x_2^2 + x_3^2 +x_4^3=0$, thereby restricting the geometry of the resolution $\mu_Y$.
We will not use this finer result, however we will use several times that $\mu_Y$ contracts at least one rational curve.
\end{remark*}

\begin{lemma} \label{lemma-rel-ample}
In the Setup \ref{setup-general} the divisor $\mu^*(-K_X)$ is $\varphi$-ample and therefore by \eqref{pullbackmu} the effective divisor $E$ is $\varphi$-ample.
\end{lemma}

\begin{proof}
The divisor $\mu^*(-K_X)$ is semiample and big and has degree zero exactly on the exceptional curves of the blow-up $\mu$. Thus if $F \subset X'$ is a $\varphi$-fibre such that $\mu^*(-K_X)|_F$ is not
ample, then $F$ contains a $\mu$-exceptional curve $l$. 
Since $-K_{X'} \simeq \varphi^* \sO_T'(1)$ we obtain $-K_{X'} \cdot l=0$.
Yet we know that $-K_{X'} \cdot l=1$ since
$\mu$ is a blow-up, a contradiction.
\end{proof}

\begin{lemma} \label{lemma-dimensionT}
In the Setup \ref{setup-general} the variety $T'$ has dimension three.
\end{lemma}

\begin{proof}
Since $4 \leq h^0(X, \sO_X(-K_X)) = h^0(X', \sO_{X'}(-K_{X'}))$ it is clear that $\dim T'>0$.
Note that $\dim T'$ is equal to the numerical dimension $\nu(-K_{X'})$.
We will argue by contradiction and exclude the cases $\dim T' \neq 3$ one-by-one. 

{\em 1st case. Assume $\dim T'=\nu(-K_{X'})=1$.}
Then the restriction of $-K_{X'}$ to $Y' \in |-K_{X'}|$ is numerically trivial. Yet $-K_{X'}|_{Y'} \simeq M' \neq 0$, a contradiction.

{\em 2nd case. Assume  $\dim T'=\nu(-K_{X'})=4$.}
Then the divisor $-K_{X'}$ is big and so is its restriction $M' \simeq -K_{X'}|_{Y'}$.  Thus we have $H^1(Y', \sO_{Y'}(M'))=0$ by Kawamata-Viehweg vanishing
and therefore the restriction morphism
$$
H^0(Y', \sO_{Y'}(-K_{X'})) \rightarrow H^0(B', \sO_{B'}(-K_{X'}) \simeq H^0(B', \sO_{B'}(K_{B'} + M'))
$$
is surjective and the zero map. Thus $H^0(B', \sO_{B'}(K_{B'} + M'))=0$ which is impossible by \cite[Thm.3.1]{Kaw00} if $M'|_{B'}$ is big. If 
$$
M'|_{B'} \simeq -K_{X'}|_{B'} \simeq (\varphi^* H)|_{B'}
$$
is not big, the prime divisor $B'$ is contracted by the birational map $\varphi_{Y'}$.
In particular $B'$ is not $\varphi_{Y'}$-nef by the negativity lemma, i.e.\ there
exists a curve $C \subset B'$ that is contracted by $\varphi_{Y'}$ and $B' \cdot C<0$. 
By \eqref{onY} we obtain
$$
\mu^*(-K_{X}) \cdot C =
((\varphi^* H)|_{Y'} + B') \cdot C = 
B' \cdot C<0,
$$
a contradiction to $\mu^*(-K_{X})$ being nef.

{\em 3rd case. Assume $\dim T'=\nu(-K_{X'})=2$.}
Since the exceptional divisor is $\varphi$-ample by Lemma \ref{lemma-rel-ample},
the morphism $E \rightarrow T'$ is surjective. 
Let $l \subset E$ be a fibre of the blow-up $\mu$.
Then we have
$$
1 = -K_{X'} \cdot l = \varphi^* H \cdot l = H \cdot \varphi_* l.
$$
Since $H$ is an ample Cartier divisor we see that the curves $\varphi_* l = \varphi(l)$
define an unsplit family of rational curves that has degree one with respect to $H$.

Now recall from the Setup \ref{setup-general} that the birational map $\holom{\mu_Y}{Y'}{Y}$
is not an isomorphism, so we can choose $l$ such that it is contained in $Y'$.
Since $U=\varphi(Y) \subset T'$ is an irreducible curve we obtain
$U = \varphi(l)$.
Thus we have
$$
H^2 = 
H \cdot U = H \cdot l = 1.
$$
Since $H$ is globally generated, the morphism $\varphi_{|H|}$ defines an isomorphism $T' \simeq \PP^2$ such that $H \simeq (\varphi_{|H|})^* \sO_{\PP^2}(1)$. In particular we have
$$
h^0(X, \sO_X(-K_X)) = h^0(X', \sO_{X'}(-K_{X'})) = h^0(T', \sO_{T'}(H)) = h^0(\PP^2, \sO_{\PP^2}(1))=3,
$$
contradicting our assumption $h^0(X, \sO_X(-K_X)) \geq 4$.
\end{proof}

\begin{remark} \label{remark-dimensionT}
In the proof of Lemma \ref{lemma-dimensionT} the assumption $h^0(X, \sO_X(-K_X)) \geq 4$ was only used in the third case. If $h^0(X, \sO_X(-K_X)) \geq 3$ the proof shows that either $\dim T'=3$
or we have a fibration $\holom{\varphi:=\varphi_{|-K_{X'}|}}{X'}{\PP^2}$
such that $-K_{X'} \simeq \varphi^* \sO_{\PP^2}(1)$.
\end{remark}

An outcome of the proof of Theorem \ref{theorem-nefdivisor} will be the following technical statement that we need
for later reference:

\begin{lemma} \label{lemma:technical}
In the situation of the Setup \ref{setup-general}, the morphism 
$$
\holom{\varphi:=\varphi_{|-K_{X'}|}}{X'}{T'}
$$
does not contract a divisor. Moreover there exists a finite set $N \subset T'$ such that 
$$
\fibre{\varphi}{T' \setminus N} \rightarrow T' \setminus N
$$
is a flat elliptic fibration with integral fibres.
\end{lemma}

\begin{proof}[Proof of Theorem \ref{theorem-nefdivisor} and Lemma \ref{lemma:technical}]
We use the notation from the Setup \ref{setup-general}.
Our main goal is to construct a nef divisor $A_X \rightarrow X$ of numerical dimension two.
This will be achieved in the four steps.

Since $\dim T'=3$, the general fibre of 
$\holom{\varphi}{X'}{T'}$ is an elliptic curve and $E \subset X'$ surjects onto $T$ since $E$ is $\varphi$-ample by Lemma \ref{lemma-rel-ample}.

{\em Step 1. We show that $\holom{\varphi_E}{E}{T'}$ is birational and the image of the exceptional locus is at most a finite set.}
The birational morphism
$\holom{\mu_Y}{Y'}{Y}$
is small, so if we set $\mu^*(-K_X)|_{Y'}$ it is a nef and big divisor such  that
the restriction to every prime divisor is big. By \eqref{onY}
we have a decomposition $|\mu^*(-K_X)|_{Y'}| \simeq |-K_{X'}|_{Y'}|+B'$ into mobile and fixed part
and we have seen in the Setup \ref{setup-general} that the induced elliptic fibration
$\holom{\varphi_{Y'}}{Y'}{U}$ is defined by the full linear system $|-K_{X'}|_{Y'}| = \varphi_{Y'}^* |H|_U|$. Since $B'$ is smooth we can apply
Theorem \ref{thm-birational} to obtain that $B'$ is a rational
section of $\varphi_{Y'}$. Since $B' = E \cap \fibre{\varphi}{U}$ this shows that $E \rightarrow T'$ is birational.

Let now $N \subset T'$ be the image of the exceptional locus of $\varphi_E$, arguing by contradiction we assume that $\dim N>0$. Then for a general element $U \in |H|$ the intersection $U \cap N$ is not empty, so if we set $Y' := \fibre{\varphi}{U} \in |-K_{X'}|$,
then $B' := Y' \cap E$ is a rational section of
$\holom{\varphi_{Y'}}{Y'}{U}$ and the map $U \dashrightarrow B'$ has at least one point of indeterminacy.
Yet $B'$ is $\varphi_{Y'}$-ample by Lemma \ref{lemma-rel-ample},
so we know by Proposition \ref{proposition-section} that $B'$ is
a regular section, so we get a contradiction.

{\em Step 2. We show that $\holom{\varphi}{X'}{T'}$ does not contract a divisor.}
We argue by contradiction and assume that $\varphi$ contracts a prime divisor $D$ onto a set of dimension at most one. Thus the fibres of $D \rightarrow \varphi(D)$ have dimension at least two. Since $E$ is $\varphi$-ample, the intersection $D \cap E$ is not empty
and contained in the exceptional locus of $\holom{\varphi_E}{E}{T'}$.
By Step 1 this shows that $\varphi(D)$ is a point.

We choose a general element $U \in |H|$ that does not contain the point $\varphi(D)$,
so the general anticanonical divisor $Y' := \fibre{\varphi}{U} \subset X'$ is disjoint from $D$. 
The morphism $D \rightarrow \mu(D)$ is finite since $\mu^*(-K_X)$
is ample on $D$ by Lemma \ref{lemma-rel-ample}. Thus $\mu(D)$ is a prime divisor in $X$,
it is not disjoint from $B$ since it contains the non-empty set $\mu(D \cap E) \subset B$.
Since $Y'$ and $D$ are disjoint, their images $Y := \mu(Y)$ and $\mu(D)$ intersect at most
along $B$. Since $B \subset Y$ and $\mu(D)$ is not disjoint from $B$ the intersection is not empty. Yet both divisors are Cartier, so $Y \cap \mu(D)$ has dimension at least two. 
In conclusion we have a set-theoretical equivality
$$
Y \cap \mu(D) = B.
$$
Since $\mu(D)$ is Cartier this shows that $B \subset Y$ is $\Q$-Cartier, a contradiction
to \cite[Thm.1.3]{HS25}.

{\em Step 3. Construction of the divisor $A_X$.}
Let $N \subset T'$ be the image of the $\varphi_E$-exceptional locus, so $N$ is finite by Step 1.
Since $E$ is $\varphi$-ample the fibration $\varphi$ is equidimensional over $T' \setminus N$.
\footnote{Together with Step 1 and 2 this finishes the proof of Lemma \ref{lemma:technical}.}

Let now $A \rightarrow B$ be a general hyperplane section, 
then $\mu_E^* A$ is a prime divisor on $E$, so its push-forward
$$
A_{T'} := (\varphi_E)_* (\mu_E^* A) \qquad \in \mbox{\rm Cl}(T')
$$
via the birational map $\varphi_E$ is a prime Weil divisor on $T'$.
For later use note that $A_{T'}$ is Cartier in the complement of $N$. 
 
By Step 2 the fibration $\varphi$ does not a contract a divisor, so by Lemma \ref{lemma-pullback-weil} the pull-back
$$
A_{X'} := \varphi^* A_{T'} \qquad \in \mbox{\rm Cl}(X')
$$
is well-defined. Since $X'$ is smooth the divisor $A_{X'}$ is Cartier, and 
$$
A_X := \mu_* A_{X'} \in \mbox{\rm Cl}(X')
$$
is also Cartier.

{\em Step 4. We show that $A_X$ is nef of numerical dimension two.}
We first show that
\begin{equation}
\label{AXpullback}
A_{X'} \simeq \mu^* A_X.
\end{equation}

{\em Proof of the claim.}
Let $Y' \subset X'$ be a general anticanonical divisor, so
$Y' = \varphi^* U$ with $U \subset T'$ a divisor that is disjoint from $N$. In particular 
$B' = E \cap Y'$ is a section of $\varphi_{Y'}$. 
The pull-back defined in Lemma \ref{lemma-pullback-weil} extends the pull-back of Cartier divisors. Since $A_{T'}$ is Cartier in the complement of $N \subset T'$ we have
$$
A_{X'}|_{Y'} \simeq \varphi_{Y'}^* (A_{T'}|_U).
$$
Yet by definition $A_{T'}|_U = (\varphi|_{B'})_* (\mu_{B'})^* A$
and $\varphi_{B'}$ is an isomorphism, so
\begin{equation}
\label{AXrestricted}
A_{X'}|_{Y'} \simeq  \varphi_{Y'}^* (\varphi_{B'})_*  (\mu_{B'})^* A
\qquad 
\mbox{and}
\qquad
A_{X'}|_{B'} \simeq  (\mu_{B'})^* A.
\end{equation}
Now recall that $Y' \rightarrow Y$ is a small modification that is not an isomorphism,
so $B' = E \cap Y'$ contains at least one of the exceptional curves $l$.
Thus \eqref{AXrestricted} yields 
$$
A_{X'}  \cdot l = (A_{X'})|_{B'} \cdot l  =   (\mu_{B'})^* A \cdot l = 0.
$$
Thus $A_{X'}$ has degree zero on the extremal ray contracted by $\mu$, so it is a pull-back
of some Cartier divisor on $X$ \cite[Thm.7.39 c)]{Deb01}. By definition $A_X = \mu_* A_{X'}$ so the Cartier divisor must be $A_X$.

In order to show that $A_X$ is nef, note that it is sufficient by Koll\'ar's theorem \cite[Theorem]{Bor91} to show that $A_X|_Y$ is nef for $Y \subset X$ a general anticanonical divisor.  Since $A_{X'} = \mu^* A_X$ this is equivalent
to showing that $A_{X'}|_{Y'}$ is nef.
Yet by \eqref{AXrestricted} we know that $A_{X'}|_{Y'}$ is the pull-back of a hyperplane section, so nef.  

A nef divisor on a Fano manifold is semiample, so the numerical dimension coincides with the Iitaka dimension and the nef dimension (cf. \cite{8authors02} for the notion of nef dimension).
Let us now compute this dimension: 
by \eqref{AXrestricted} the numerical dimension of $A_{X'}|_{Y'} \simeq \mu_Y^* (A_X|_Y)$ is two, so $A_X$ has numerical dimension at least two.

On the other hand if $C$ is a general fibre of the elliptic fibration $\varphi_{Y'}$ we have
$$
A_X \cdot \mu(C) = \mu^* A_X \cdot C = A_{X'} \cdot C= A_{X'}|_{Y'} \cdot C= 0
$$ 
by \eqref{AXrestricted}. The family of curves
$\mu(C)$ dominates $X$, so $A_X$ has nef dimension at most three. Yet if the nef dimension
is exactly three the general fibre of $\varphi_{|m A_X|}$ for $m \gg 0$ is exactly
the curve $\mu(C)$. Since $C$ is elliptic the curve $\mu(C)$ has positive genus. 
But $X$ is a Fano manifold, so this is impossible.

After this technical preparation we can finally come to the statements in our theorem:

{\em Proof of item i.}
A nef divisor on a Fano manifold is semiample, so by Step 4 some multiple of $A_X$
defines a fibration 
$$
\holom{\pi}{X}{V}
$$
onto a surface $V$. Thus $A_X \sim_\Q \pi^* H_V$ with $H_V$ an ample $\Q$-divisor on $V$. Combining \eqref{AXpullback} and \eqref{AXrestricted} we see that
$A_X|_B$ coincides with the hyperplane section $A$. In particular $A_X|_B \sim_\Q (\pi^* H_V)|_B$ is ample and therefore the morphism $B \rightarrow V$ is finite.

Let $F$ be a general $\pi$-fibre, so $F$ is a smooth del Pezzo surface.
Let $F' \subset X'$ be the strict transform under the blow-up $\mu$. Then
$F' \rightarrow F$ is the blow-up in the finite set $B \cap F$ and $E \cap F'$
is the exceptional locus.
We have seen in Step 4 that a general fibre $C_t = \fibre{\varphi}{t}$ of the elliptic fibration
satisfies $A_X \cdot \mu(C_t)=0$, so $F$ is dominated by a one-dimensional family
of curves $\mu(C_t)$. Since $E$ is a $\varphi$-section we have $E \cdot C_t=1$
and therefore $-K_F \cdot \mu(C_t) = -K_X \cdot \mu(C_t)=1$.
Since the nef divisor $C_t \subset F'$ meets the exceptional divisor $E \cap F'$ its
image $\mu(C_t)$ is nef and big. By the Hodge index inequality (Lemma \ref{hodge-index}) we have
$$
1 = (-K_F \cdot \mu(C_t))^2 \geq (-K_F)^2 \cdot {\mu(C_t)}^2 \geq (-K_F)^2 \geq 1.
$$
Thus we have $(-K_F)^2 = {\mu(C_t)}^2 = 1$ and $F$ is a del Pezzo surface of degree one.
Moeover equality holds in the Hodge index inequality, so
the divisors are colinear in $N^1(F) \simeq \mbox{Pic}(F)$.
This shows that $\mu(C_t) \in |-K_F|$.
Since the elliptic curve $\mu(C_t) \subset F$ passes through all the points in $B \cap F$
we see that $B \cap F$ is exactly one point. Therefore the finite morphism
$\pi_B: B \rightarrow V$ is birational, so an isomorphism.

The anticanonical divisor $-K_X$ is a $\pi$-ample line bundle such that $(-K_X)^2 \cdot F=1$.
Moreover $B$ is a $\pi$-section, so we conclude by Lemma \ref{lemma-factorisation-degree-one}.

{\em Proof of items ii.\ and iii.}
By the relative Kodaira vanishing theorem we have $R^j \pi_* \sO_X(-K_X)=0$ for all $j \geq 1$. Since $\pi$ is flat we obtain by cohomology and base change that
$$
h^0(F_v, \sO_{F_v}(-K_{X})) = h^0(F, \sO_F(-K_X))=2
$$
for every $v \in V$. Since $(-K_X)^2 \cdot F_v=1$ 
and $-K_X|_{F_v}$ is ample, every element
in the linear system $|-K_X|_{F_v}|$ has irreducible support. In particular  $|-K_X|_{F_v}|$
has no fixed part. Now $(-K_X)^2 \cdot F_v=1$ implies that its base scheme
is a reduced point $p_v$ and this point is contained in the smooth locus of $F_v$.
Since
$$
\mbox{Bs} |-K_X|_{F_v}| \subset \BsAX \cap F_v = B \cap F_v
$$
and $B$ is a $\pi$-section we obtain that $p_v$ is exactly $B \cap F_v$.
In conclusion we see that the image of the natural morphism
$$
\pi^* \pi_* \sO_X(-K_X) \rightarrow \sO_X(-K_X)
$$
is exactly the sheaf $\sI_B \otimes \sO_X(-K_X)$. Pulling-back via the principalisation $\mu$ we obtain a surjection
$$
\mu ^* \pi^* \pi_* \sO_X(-K_X) \twoheadrightarrow \mu^* (\sI_B \otimes \sO_X(-K_X))
\twoheadrightarrow \sO_{X'}(-E) \otimes \mu^* \sO_X(-K_X) \simeq \sO_{X'}(-K_{X'}).
$$
The quotient line bundle $\mu ^* \pi^* \pi_* \sO_X(-K_X) \rightarrow \sO_{X'}(-K_{X'})$
determines a morphism
$$
\holom{\phi}{X'}{T := \PP(\pi_* \sO_X(-K_X))} \rightarrow V
$$
such that for every $v \in V$ the restriction to $F_v'$ is
the blow-up of the basepoint $p_v$. In particular $\phi$ is equidimensional and its general fibre coincides with the fibration $\holom{\varphi}{X'}{T'}$.
The rigidity lemma gives a birational morphism
$\holom{\nu}{T}{T'}$ such that $\varphi = \nu \circ \phi$.

It is clear that $E$ is a rational $\phi$-section, so we have 
a birational morphism $\holom{\phi_E}{E}{T}$. Since both manifolds are $\PP^1$-bundles
over $V \simeq B$, they have the same Picard number. Thus $\phi_E$ is an isomorphism
and $E$ is a regular section. Since $E$ is $\varphi$-ample it is clear that it is also $\phi$-ample. By Lemma \ref{lemma-factorisation-degree-one} the fibres of $\phi$ are integral, so we have shown ii.

The existence of the commutative diagram in iii.\ is clear by construction. Note that
the isomorphism $\phi_E$ identifies the $\PP^1$-bundle structure $\mu_E$
to the structural morphism $\PP(\pi_* \sO_X(-K_X)) \rightarrow V$
so we denote it by $\mu_T$.
\end{proof}

\section{Examples} \label{section-examples}

In this section we present a general construction for Fano manifolds with large anticanonical base locus which is crucial for our classification theorem.
We will then apply the construction to certain Fano bundles over del Pezzo surfaces of degree $\geq 2$ to illustrate the non-trivial birational geometry of most of these varieties. Thanks to our classification in Section \ref{The classification}, all the examples presented here are precisely the Fano fourfolds appearing in Theorem \ref{theorem-main}; see Table \ref{table} for their numerical invariants.

\begin{proposition} \label{construction:general}
Let $\sF$ be a globally generated vector bundle of rank $r \geq 2$ over a projective manifold $V$. Denote by 
$$
\holom{\mu_T}{T:= \PP(\sF)}{V}
$$ 
the projectivisation and by $\zeta_T$ the tautological divisor
on $T$. Set
$$
W := \sO_T \oplus \sO_T(-2\zeta_T) \oplus \sO_T(-3\zeta_T)
$$
and denote by $\zeta_W$ the tautological divisor on 
$$
\holom{\Phi}{\PP(W)}{T}.
$$
Let $X' \subset \PP(W)$ be a general element of the linear system $|3 \zeta_W + \Phi^* (6\zeta_T)|$.
Then the following holds:
\begin{enumerate}
\item The variety  $X'$ is smooth and contains $E:=\PP(\sO_T(-3\zeta_T)) \simeq T$.
\item The morphism 
$$
\holom{\phi:=\Phi|_{X'}}{X'}{T}
$$
is a flat elliptic fibration with integral fibres. 
\item The section $E \subset X'$ is a $\phi$-ample divisor and $\ma N_{E/X'} \simeq \sO_T(-\zeta_T)$.
\item The fibration 
$$
\holom{\mu_E :=\mu_T}{E \simeq T}{V}
$$
extends to a bimeromorphic morphism
$$
\holom{\mu}{X'}{X}
$$
onto a compact complex manifold $X$ such that $\mu|_E = \mu_E$ and $X' \setminus E \simeq X \setminus B$, i.e.\ $X'$ is the blow-up of $X$ along $B := \mu(E) \simeq V$. 
\item The morphism $\mu_T \circ \phi$ factors through the blow-down $\mu$, and we obtain
a fibration
$$
\holom{\pi}{X}{V}
$$
such that $-K_{X}$ is $\pi$-ample and all the fibres are del Pezzo varieties of degree one.
\item The manifold $X$ is projective.
\item We have
\begin{equation}
\label{formulaanticanonical}
-K_{X'} \simeq \phi^* \bigl((r-1) \zeta_T + \mu_T^* (-K_V-\det \sF)\bigr).
\end{equation}
If $-K_V-\det \sF$ is ample, then 
$X$ is a Fano manifold.
\end{enumerate}
\end{proposition}

\begin{remark}\label{remark-restrictionB}
For later use, let us note that the proof will show that
$$
\sO_{B}(-K_X) \simeq \sO_V(-K_V-\det \sF)
$$
and
$$
\ma N_{B/X}^*\simeq \sF.
$$ 
\end{remark}

\begin{proof}
In this proof we will frequently identify divisors on $T$ and $E$ via the isomorphism
$\Phi|_E$.

{\em Proof of i.}
Set $L := \sO_W(3 \zeta_W + \Phi^*(6\zeta_T))$ and note that 
$$
L \simeq \sO_{\PP(\sO_T(2\zeta_T) \oplus \sO_T \oplus \sO_T(-\zeta_T))} (3). 
$$
Since $\zeta_T$ is globally generated, the base locus of the linear system $|L|$ is exactly $E$. 
Note also that
$$
\ma N^*_{E/\PP(W)} \otimes L_E \simeq 
\bigl(
\sO_E(3\zeta_T) \oplus \sO_E(\zeta_T)
\bigr) 
\otimes \sO_E(-3\zeta_T)
\simeq
\sO_E \oplus \sO_E(-2\zeta_T). 
$$
Thus a section of this vector bundle is either zero or does not vanish in any point. By 
\cite[Lemma 1.7.4]{BS95}
this means that a divisor $X' \in |L|$ is either smooth along the base locus or has $E$ as a component of its singular locus. In view of the exact sequence
$$
0 \rightarrow H^0(\PP(W), \sI_E^2 \otimes L) \rightarrow H^0(\PP(W), \sI_E \otimes L) \rightarrow  H^0(E,\ma N^*_{E/\PP(W)} \otimes L_E) \simeq \C
$$
there exists a smooth element $X' \in |L|$ if and only if the inclusion 
\begin{equation}
\label{inclusionideals}
H^0(\PP(W), \sI_E^2 \otimes L) \rightarrow H^0(\PP(W), \sI_E \otimes L)
\end{equation}
is strict. Let $\holom{\tau}{P}{\PP(W)}$ the blow-up of $E$, and denote by $G$ the exceptional divisor. Then we have
$$
H^0(P,  \sO_P(-2G) \otimes \tau^* L) \simeq H^0(\PP(W), \sI_E^2 \otimes L), \quad H^0(P, \sO_P(-G) \otimes \tau^* L) \simeq H^0(\PP(W), \sI_E \otimes L).
$$
Set
$$
\holom{\pi_K}{\PP(K) := \PP(\sO_T \oplus \sO_T(-2\zeta_T))}{T},
$$
and denote by $\zeta_K$ the tautological divisor on $\PP(K)$.
Then $\tau$ resolves the indeterminacies of the map 
$\PP(W) \dashrightarrow \PP(K)$ given by fibrewise projection from $E$. Therefore 
$\holom{q}{P}{\PP(K)}$ is a $\PP^1$-bundle and the exceptional divisor $G$ is a $q$-section.
An 
elementary computation shows that
\begin{equation}
\label{pushG}
q_*(\sO_P(G)) \simeq \sO_{\PP(K)} \oplus \sO_{\PP(K)}(-\zeta_K + \pi_K^*(-3\zeta_T)).
\end{equation}
We summarise the construction in a commutative diagram:
$$
\xymatrix{
P \ar[d]_{q} \ar[r]^{\tau} & \PP(W) \ar[d]^{\Phi} & \ar @{_{(}->}[l] X' \ar[ld]^\phi \\
\PP(K) \ar[r]^{\pi_K} & T \ar[d]^{\mu_T} & \\
& V &
}
$$
We have $\tau^* \zeta_W \simeq q^* \zeta_K + G$ and therefore
$$
\sO_P(-2G) \otimes \tau^* L \simeq \sO_P\bigl(G + q^* (3 \zeta_K + \pi_K^* (6\zeta_T))\bigr),
\quad
\sO_P(-G) \otimes \tau^* L \simeq \sO_P\bigl(2G + q^* (3 \zeta_K + \pi_K^* (6\zeta_T))\bigr).
$$
Pushing forward via $q_*$ and using \eqref{pushG} the inclusion \eqref{inclusionideals} becomes
\begin{multline}
H^0\bigl(\PP(K), \sO_{\PP(K)}(3 \zeta_K+\pi_K^* (6\zeta_T)) \oplus  \sO_{\PP(K)}(2 \zeta_K + \pi_K^* (3\zeta_T))\bigr)
\hookrightarrow \\
H^0\bigl(\PP(K), \sO_{\PP(K)}(3 \zeta_K+\pi_K^* (6\zeta_T)) \oplus  \sO_{\PP(K)}(2 \zeta_K + \pi_K^* (3\zeta_T))  \oplus \sO_{\PP(K)}(\zeta_K)\bigr).
\end{multline}
Since $H^0(\PP(K), \sO_{\PP(K)}(\zeta_K))=\C$ the inclusion is strict,
so a general element $X' \in |L|$ is smooth. 
The variety $E$ is contained in the base locus of $|L|$ so it is clear that $E \subset X'$.

{\em Proof of ii.\ and iii.}
The varieties $X'$ and $T$ being smooth, the flatness of $\phi$ is clear if we show that $\phi$
is equidimensional. The equation of 
$X' \in |3 \zeta_W + \Phi^* (6\zeta_T)|$ in $\PP(W)$ corresponds to a global section of the vector bundle
$$
S^3 W \otimes \sO_T(6\zeta_T) \simeq
S^3 \bigl(\sO_T \oplus \sO_T(-2\zeta_T) \oplus \sO_T(-3\zeta_T)\bigr) \otimes \sO_T(6\zeta_T) 
\supset \sO_T.
$$
For a general section, the coefficient corresponding to the trivial direct factor does not vanish in any point $t \in T$, so the cubic equation in $\PP(W_t) \simeq \PP^2$ is not zero.

Note that the restriction of $L$ to the divisor $\PP(
\sO_T(-2\zeta_T) \oplus \sO_T(-3\zeta_T)) \subset \PP(W)$ is
isomorphic to $\sO_{\PP(
\sO_T \oplus \sO_T(-\zeta_T))}(3)$, so it has $3E$ as its unique section. This shows that
$$
\sO_{\PP(W)}(1) \otimes \sO_{X'} \simeq \sO_{X'}(3E),
$$
so the section $E \subset X'$ is $\phi$-ample. Since $\phi$ is flat and has a relatively ample line bundle of degree one, all the fibres are integral.

Since $\ma N_{E/\PP(W)} \simeq \sO_E(-3\zeta_T) \oplus \sO_E(-\zeta_T)$ the normal bundle 
$\ma N_{E/X'}$ can be easily computed from the exact sequence
$$
0 \rightarrow \ma N_{E/X'} \rightarrow \ma N_{E/\PP(W)} \rightarrow \ma N_{X'/\PP(W)} \otimes \sO_E 
\simeq L_E \simeq \sO_T(-3\zeta_T) \rightarrow 0.
$$

{\em Proof of iv.}
The isomorphism $\ma N_{E/X'}^* \simeq \sO_T(\zeta_T)$ implies the existence of
the blow-down of $E$ along $\mu_E$ since we can apply the criterion of Fujiki and Nakano
\cite{FN72},\cite[Thm.3.1]{Pet94}.

{\em Proof of v.\ and vi.}
By the rigidity lemma the factorisation $\pi$ exists if every $\mu$-fibre is contracted
by $\mu_T \circ \phi$ onto a point. But this is clear since by construction $\mu$
is the extension of $\mu_E=\mu_T$ to a blow-down. 

For any point $v \in V$ the fibre $X'_v := \fibre{(\mu_T \circ \phi)}{v}$ has a flat elliptic
fibration with integral fibres over $\fibre{\mu_T}{v} \simeq \PP^{r-1}$, in particular it is integral. The intersection $X'_v \cap E \simeq \PP^{r-1}$ is contained in the smooth locus
of $X'_v$ and $\ma N_{E/X'_v\cap E} \simeq \sO_{\fibre{\mu_T}{v}}(-\zeta_T) \simeq
\sO_{\PP^{r-1}}(-1)$, so the restriction of $\mu$ to $X'_v$ is the blow-down to a smooth point. By \cite[13.7]{Fuj84III} this shows that $X_v$ is a del Pezzo variety of degree one.

Since all the $\pi$-fibres are integral and Fano, the anticanonical bundle $-K_X$ is $\pi$-ample. In particular the morphism $\pi$ is projective over a projective base $M$, so $X$ is projective.

{\em Proof of vii.}
The formula \eqref{formulaanticanonical} is a consequence of the adjunction formula
and the formula for the canonical of a projectivised bundle. Let us now show the last 
statement:
since $\zeta_T$ is the tautological class of a nef vector bundle, we know that
$\zeta_T + \mu_T^* (-K_V-\det \sF)$ is ample. 
Note that $B \subset X$ has codimension $r$, so we have
$$
\mu^*(-K_X) \simeq -K_{X'} + (r-1) E.
$$
Since $E$ is $\phi$-ample and $-K_{X'}$ is the pull-back of an ample divisor on $T$, we have
$\mu^*(-K_X) \cdot C>0$ for every curve $C \subset X'$ that is not contained in $E$. Moreover
\eqref{formulaanticanonical} and iii.\ imply
$$
\sO_E(\mu^*(-K_X)) \simeq \sO_E(-K_{X'} + E) \simeq \sO_T(\mu_T^* (-K_V-\det \sF)),
$$
so $\mu^*(-K_X)$ is nef and has degree zero exactly on the fibres of $\mu$.
This shows that $-K_X$ is nef and has strictly positive degree on all curves,
it also proves Remark \ref{remark-restrictionB}.

Note now that
$$
{\mu^*(-K_X)}^{\dim X} \geq (-K_{X'})^{\dim X-1} \cdot E > 0
$$
since $-K_{X'}$ is the pull-back of an ample divisor on $T \simeq E$. Thus $-K_X$ is nef and big and strictly positive on all the curves. By the basepoint-free theorem we obtain that $-K_X$ is ample.
\end{proof}

We now focus on the case that is relevant for us:

\begin{proposition} \label{construction:rk2}
In the situation of Proposition \ref{construction:general} assume that
$$
r=2
$$
and that $-K_V-\det \sF$ is ample and globally generated. Then the following holds:
\begin{enumerate}
\item We have 
$B = \BsAX.$
\item We have
$
E + \phi^* \zeta_T \simeq \mu^* A \simeq \mu^* (-K_{X/V}+\det \sF),
$
with $A$ a semiample divisor on $X$. 
\item Let 
$$
\holom{\psi}{X}{\bar X}
$$
be the morphism with connected fibres defined by some positive multiple of $A$. Then we have $A \simeq \psi^* \bar A$
with $\bar A$ an ample divisor on $\bar X$.
\item The morphism $\psi$ contracts the base locus $B$ onto a point $p$ and
$\mbox{\rm Bs}(|\bar A|)=p$.
\item Let $N \subset T$ be a subvariety such that $\sO_N(\zeta_T) \simeq \sO_N$. Then
$$
\fibre{\phi}{N} \simeq N \times C
$$
with $C$ an integral curve of arithmetic genus one and 
\begin{equation}
\label{restrictionEN}
\sO_{\fibre{\phi}{N}}(E) \simeq p_C^* \sO_C(p).
\end{equation}
The morphism
$\psi \circ \mu$ contracts 
$\fibre{\phi}{N}$ onto a curve $\bar C \subset \bar X$ such that $\bar A \cdot \bar C=1$.
\item Let $k$ be the numerical dimension of the nef divisor $\zeta_T$. Then
$\bar X$ has dimension $k+1$. In particular $\psi$ is birational if and only if $\zeta_T$ is big.
In this case let $D_T \subset T$ be the null locus (cf.\ Remark \ref{remark-nulllocus})
of the divisor $\zeta_T$. Then the exceptional locus of $\psi$ is
$$
B \cup \mu(\fibre{\phi}{D_T}).
$$
\end{enumerate}
\end{proposition}

\begin{remarks}
\begin{itemize}
\item The morphism $\psi$ contracts an extremal face in the Mori cone of $X$, but 
as one can see from the examples below it is not an elementary contraction
unless $V \simeq \PP^2$.
\item If $V$ a del Pezzo surface 
of degree one and $\sF=\sO_{V}^{\oplus 2}$, the divisor $-K_V-\det\sF$ is not globally generated,
so Proposition \ref{construction:rk2} does not apply. Nevertheless our construction 
yields the Fano fourfold from \cite[Ex.2.12]{HV11} where the base locus $\BsAX$ is a reducible surface. 
\item A posteriori, Theorem \ref{theorem-main} shows that any smooth Fano fourfold with $h^0(-K_X) \geq 4$ and such that the scheme-theoretic base locus $\BsAX$ is a smooth irreducible surface $B$ is uniquely determined by the data $(B,\, \ma N_{B/X}^*)$. Therefore we will describe some of the examples below, which are constructed through Propositions \ref{construction:general} and \ref{construction:rk2}, by explicitly producing a smooth Fano fourfold as above with the prescribed data $(B,\, \ma N_{B/X}^*)$.
\end{itemize} 
\end{remarks}

\begin{proof}

{\em Proof of i.}
By assumption $-K_V-\det \sF$ is ample and globally generated, so \eqref{formulaanticanonical}
shows that $-K_{X'}$ is the pull-back of an ample and globally generated divisor.
Since $B \subset X$ has codimension $r=2$, we have
$$
H^0(X', \sO_{X'}(-K_{X'})) \simeq H^0(X, \sI_B \otimes \sO_{X}(-K_{X})) \subset H^0(X, \sO_{X}(-K_{X})).
$$
So it is sufficient to show that $H^0(X', \sO_{X'}(-K_{X'})) \simeq H^0(X, \sO_{X}(-K_{X}))$.
Yet by Fact \ref{fact:weierstrass} we have $\phi_* \sO_{X'}(E) \simeq \sO_T$, so
by the projection formula and \eqref{formulaanticanonical}
$$
\phi_* \mu^* \sO_X(-K_X) \simeq \phi_* \sO_{X'}(-K_{X'} + E) \simeq 
\sO_T\bigl(\zeta_T + \mu_T^* (-K_V-\det \sF)\bigr) \otimes \phi_* \sO_{X'}(E) \simeq  \phi_*\sO_{X'}(-K_{X'}).
$$ 
In particular the two line bundles have the same space of global sections.

{\em Proof of ii.\ and iii.} The divisor $\zeta_T$ is nef and $E$ is effective, so 
$(E + \phi^* \zeta_T) \cdot C \geq 0$ for every curve $C \not\subset E$.
Moreover 
\begin{equation}
\label{restrictionAE}
\sO_E(E + \phi^* \zeta_T) \simeq \sO_E
\end{equation}
by item iii.\ of Proposition \ref{construction:general},
so $E + \phi^* \zeta_T$ is nef. The divisor has degree zero on the extremal ray contracted
by the blow-down $\mu$, so there exists a divisor $A$ on $X$ such that
 $E + \phi^* \zeta_T \simeq \mu^* A$. Since $X$ is Fano by item vii.\ of Proposition \ref{construction:general}, the nef divisor $A$ is semiample.
 The isomorphism  $E + \phi^* \zeta_T \simeq \mu^* A \simeq \mu^* (-K_{X/V}+\det \sF)$
 follows easily from \eqref{formulaanticanonical}.

Using again $X$ Fano a standard application of the basepoint-free theorem shows that $A \simeq \psi^* \bar A$ with $\bar A$
a Cartier divisor.
 
{\em Proof of iv.} 
Since $B=\mu(E)$  
and $E + \phi^* \zeta_T \simeq \mu^* A$, we see that \eqref{restrictionAE} implies $\sO_B(A) \simeq \sO_B$. Therefore, $\psi$ contracts $B$ onto a point $p$.
For the second statement observe that
by construction 
$$
H^0(\bar X, \sO_{\bar X}(\bar A)) \simeq H^0(X', \sO_{X'}(E+\phi^* \zeta_T)) 
$$
and $\zeta_T$ is globally generated, so the base locus of $\bar A$ is contained
in the image of $E$, i.e.\ the image of $B$. 

{\em Proof of v.} 
Let now $N \subset T$ be a subvariety such that $\sO_N(\zeta_T) \simeq \sO_N$.
Then, $W_N \simeq \sO_N^{\oplus 3}$ and so $\PP(W_N)$ is a product $N \times \PP^2$. The restriction of  $3 \zeta_W + \Phi^* (6\zeta_T)$ to $\PP(W_N)$ is isomorphic
to $p_{\PP^2}^* \sO_{\PP^2}(3)$, so 
$$
\fibre{\phi}{N} \simeq  X' \cap \PP(W_N) \simeq N \times C 
$$
with $C$ an integral cubic in $\PP^2$. 
Moreover $E \cap \PP(W_N) = \PP(\sO_N) \simeq N \times p$ 
with $p$ a point in $C$, so we obtain \eqref{restrictionEN}.

By \eqref{restrictionEN}
we have $\sO_{\fibre{\phi}{N}}(\mu^* A) \simeq p_C^* \sO_C(p)$.
Thus the variety gets contracted onto a curve $\bar C \simeq C$ and 
$\bar A \cdot \bar C=1$.
 
{\em Proof of vi.} Since $\psi$ is defined by some multiple of $A$ it is sufficient to determine
the numerical dimension of $E+\phi^* \zeta_T = \mu^* A$. From \eqref{restrictionAE} we deduce that for all $m \in \N$ one has
\begin{equation} \label{eqn-ezeta-powers}
(E+\phi^* \zeta_T)^m = E \cdot \phi^* \zeta_T^{m-1} + \phi^* \zeta_T^{m}.
\end{equation}
Therefore $(E+\phi^* \zeta_T)^{k+1} = E \cdot \phi^* \zeta_T^k \neq 0$
and $(E+\phi^* \zeta_T)^{k+2}=0$.

For the second statement, note that the exceptional locus of $\psi$ is exactly the null locus of $A$,
equivalently the image of the null locus of $\mu^* A$. By item v.\ we have
$\fibre{\phi}{\mbox{\rm Null}(\zeta_T)} \subset \mbox{\rm Null}(\mu^* A)$ and we claim that
$$
\mbox{\rm Null}(\mu^* A) = E \cup \fibre{\phi}{\mbox{\rm Null}(\zeta_T)}.
$$
For this we show that if
$C' \subset X'$ is an irreducible curve such that $\psi \circ \mu$ contracts $C'$, then
$C' \subset E$ or $\zeta_T \cdot \phi(C')=0$:
indeed if
$$
0= \mu^* A \cdot C' = (E+\phi^* \zeta_T) \cdot C', 
$$
and $C' \not\subset E$, then $E \cdot C' \geq 0$ and $\phi^* \zeta_T \cdot C' \geq 0$
since $\zeta_T$ is nef. Thus equality holds and $\phi(C') \subset \mbox{\rm Null}(\zeta_T)$. 
\end{proof}

\begin{remark} \label{remark-fourfold-formulas}
In the situation of Proposition \ref{construction:rk2} assume that $\dim V=2$, so that $X$ is a smooth Fano fourfold. If $\psi$ is birational, then
$$
\bar A^{4} = \zeta_T^{3} = c_1(\sF)^2-c_2(\sF).
$$
Indeed $E+\phi^* \zeta_T=\mu^* \psi^* \bar A$ together with \eqref{eqn-ezeta-powers}
implies $\bar A^{4} = \zeta_T^{3}$. Since $\sF$ has rank two, the latter is equal
to the Segre class $c_1(\sF)^2-c_2(\sF)$.
\end{remark}

We will now apply our construction to Fano bundles over del Pezzo surfaces, these examples will
eventually correspond to the families appearing in Theorem \ref{theorem-main} and Corollary \ref{corollary-main}. 
We denote by $[\# ]$ the corresponding number in Table \ref{table}.

\begin{example}$\boldsymbol{[\# 1]}$ \label{example-new2}
{\rm 
Set $V:=\PP^2$ and $\sF := \sO_{\PP^2} \oplus \sO_{\PP^2}(2)$. Then the conditions of the Propositions \ref{construction:general} and \ref{construction:rk2} are satisfied and we obtain a smooth Fano fourfold 
$X$ such that $\BsAX \simeq \PP^2$. Let
$$
\holom{\psi}{X}{\bar X}
$$ 
be the morphism constructed in Proposition \ref{construction:rk2}.
Then $\psi$ is a divisorial Mori contraction with exceptional divisor 
$$
D \simeq \PP^2 \times C
$$
and for every non-trivial fibre the normal bundle is $\sO_{\PP^2}  \oplus \sO_{\PP^2} (-2)$.
}
\end{example}

\begin{proof}
Since $\sF$ is nef and big the morphism $\psi$ is birational by item vi.\ of Proposition \ref{construction:rk2}, so we are left to describe its exceptional locus.
The null locus of $\zeta_T$ is the negative section $D_T:=\PP(\sO_T) \subset T$,
by item v.\ of Proposition \ref{construction:rk2} we have
$$
\fibre{\phi}{D_T} \simeq D_T \times C \simeq \PP^2 \times C.
$$
By \eqref{restrictionEN} the intersection with the exceptional divisor $E$ is transversal.
Since $D_T$ is a $\mu_T$-section, the intersection $\fibre{\phi}{D_T} \cap E$ is a section of $\mu_E$. Therefore $\mu$ maps $\fibre{\phi}{D_T}$ isomorphically onto its image $D \subset X$.
By item v.\ of Proposition \ref{construction:rk2} we know that $D$ is contracted by $\psi$
onto a curve $\bar C\simeq C$. Since $B \subset D$ we obtain by item vi.\ of the proposition
that $\psi$ is the contraction of the divisor $D$.

Since $\ma N_{D_T/T} \simeq \sO_{\PP^2}(-2)$, we have $\ma N_{\fibre{\phi}{D_T}/X'} \simeq p_{\PP^2}^* \sO_{\PP^2}(-2)$
and therefore 
$$
\ma N_{D/X} \simeq p_{\PP^2}^* \sO_{\PP^2}(-2) \otimes p_C^* \sO_C(p).
$$
This determines $\ma N_{B/X}^*$.
\end{proof}

\begin{remark*}
A short computation shows that 
$K_X \sim_\Q \psi^* K_{\bar X} + \frac{1}{2} D$
and $-K_{\bar X}$
has degree $\frac{3}{2}$ on the curve $\bar C \subset \bar X$.
By the item iv.\ in Lemma \ref{lemma:fano} we obtain 
that $\bar X$ is a singular Fano variety with 
singularities of type $\frac{1}{2}(1,1,1) \times \C$ along
a curve.

Moreover, $\rho(X)\leq 4$: $X$ has an elementary divisorial contraction sending a divisor onto a curve, thus it follows from \cite[Thm.3.3]{Casa12} that $\rho(X)\leq 5$, while \cite[Thm.1.8]{CRS} yields a contradiction when $\rho(X)=5$.
\end{remark*}

 \begin{example}$\boldsymbol{[\# 2]}$ \label{example-new1}
Set $V:=\PP^2$ and $\sF := \sO_{\PP^2} \oplus \sO_{\PP^2}(1)$. Then the conditions of the Propositions \ref{construction:general} and \ref{construction:rk2} are satisfied and we obtain a smooth Fano fourfold 
$X$ such that $\BsAX \simeq \PP^2$. Let
$$
\holom{\psi}{X}{\bar X}
$$ 
be the morphism constructed in Proposition \ref{construction:rk2}.
Then $\psi$ is a divisorial Mori contraction with exceptional divisor 
$$
D \simeq \PP^2 \times C
$$
and for every non-trivial fibre the normal bundle is $\sO_{\PP^2}  \oplus \sO_{\PP^2} (-1)$.
\end{example}

\begin{proof}
The proof is completely analogous to the proof of Example \ref{example-new2}, the only difference is that $\ma N_{D_T/T} \simeq \sO_{\PP^2}(-1)$, leading to the difference of the normal bundles of the fibres. 
\end{proof}

\begin{remark*}
A short computation shows that 
$K_X \simeq \psi^* K_{\bar X} + 2 D$
and $-K_{\bar X}$
has degree $3$ on the curve $\bar C \subset \bar X$.
By item iv.\ in Lemma \ref{lemma:fano} we obtain 
that $\bar X$ is a Fano manifold. In fact $X$ is the fourfold constructed
in \cite[Ex.3.4]{HS25}.
\end{remark*}

\begin{example}$\boldsymbol{[\# 3, \# 10, \# 15, \# 17-22]}$ \label{example-trivial}
Let $V$ be a smooth del Pezzo surface of degree $\geq 2$ and set $\sF := \sO_V^{\oplus 2}$. Then the conditions of the Propositions \ref{construction:general} and \ref{construction:rk2} are satisfied and we obtain a smooth Fano fourfold 
$X$ such that $\BsAX \simeq V$. Let
$$
\holom{\psi}{X}{\bar X}
$$ 
be the morphism constructed in Proposition \ref{construction:rk2}.
Then ${\bar X}$ is a del Pezzo of degree surface one, $X\simeq V \times {\bar X}$ and $\pi$, $\psi$ are the projections.
\end{example}
\begin{proof}
In this setting, $\zeta_T$ has numerical dimension one, so $\bar X$ is a surface by item vi.\ of Proposition \ref{construction:rk2}. Since $\psi$ maps a general $\pi$-fibre isomorphically onto its image, it is clear that $\bar X$ is a del Pezzo surface of degree one. We conclude by considering
the morphism $\pi \times \psi$.
\end{proof}

\begin{example}$\boldsymbol{[\# 4]}$ \label{example-new0}
Set $V:=\PP^2$ and $\sF := \sO_{\PP^2}(1)^{\oplus 2}$. Then the conditions of the Propositions \ref{construction:general} and \ref{construction:rk2} are satisfied and we obtain a smooth Fano fourfold 
$X$ such that $\BsAX \simeq \PP^2$. Let
$$
\holom{\psi}{X}{\bar X}
$$ 
be the morphism constructed in Proposition \ref{construction:rk2}.
Then $\psi$ is a small Mori contraction with exceptional locus $B$. 
\end{example}

\begin{proof}
Since $\sF$ is ample, the null locus of $\zeta_T$ is empty. 
By item vi.\ of Proposition \ref{construction:rk2}
that the exceptional locus of $\psi$ is
exactly the surface $B$.
\end{proof}

\begin{example}$\boldsymbol{[\# 5]}$ \label{example:flagvariety}
Set $V:=\PP^2$ and $\sF := T_{\PP^2}(-1)$. Then the conditions of the Propositions \ref{construction:general} and \ref{construction:rk2} are satisfied and we obtain a smooth Fano fourfold 
$X$ such that $\BsAX \simeq \PP^2$. Let
$$
\holom{\psi}{X}{\bar X}
$$ 
be the morphism constructed in Proposition \ref{construction:rk2}.
Then $\bar X$ is a smooth del Pezzo threefold of degree one and $B$ is a higher-dimensional fibre of $\psi$.
\end{example}

\begin{proof}
We have $T \simeq \PP(T_{\PP^2}(-1)) \subset \PP^2 \times \PP^2$ and the tautological divisor
$\zeta_T$ defines the projection 
$$
\holom{\varphi_{|\zeta_T|}}{T}{\PP^2}
$$
onto the second factor. Thus the numerical dimension of $\zeta_T$ is two and $\dim \bar X=3$ by item vi.\ of Proposition \ref{construction:rk2}.
In particular the general $\psi$-fibre is $\PP^1$ and $B$ is a higher-dimensional fibre.

Let us now show that the ample divisor $\bar A$ on $\bar X$ has degree one:
note that 
$$
-K_{X'} \cdot (\mu^* \psi^* \bar A)^3 =
-K_{X'} \cdot  (E+\phi^* \zeta_T)^3 = -K_{X'} \cdot  E \cdot \phi^* \zeta_T^2,
$$
and the cycle $E \cdot \phi^* \zeta_T^2$ is represented by
a fibre of the fibration
$$
\holom{\varphi_{|\zeta_T|}}{T}{\PP^2}.
$$
Therefore \eqref{formulaanticanonical} implies that
$-K_{X'} \cdot  E \cdot \phi^* \zeta_T^2=2$.
Yet the homology class of $(\mu^* \psi^* \bar A)^3$ is $\bar A^3$ times the class of a general fibre of $\psi \circ \mu$ and $-K_{X'}$ has degree two on a general fibre. Thus we have $\bar A^3=1$. 

By item iv.\ of Proposition \ref{construction:rk2} the point $p=\psi(B)$ is the unique basepoint of the linear system $|\bar A|$. Since $\bar A^3=1$, we see that $p$ is in the smooth locus of $\bar X$. It is not difficult to see that $\psi$ is a conic bundle over $\bar X \setminus p$, so $\bar X$ is smooth. 

We also have
$$
h^0(\bar X, \sO_{\bar X}(\bar A)) \simeq h^0(X', \sO_{X'}(E+\phi^* \zeta_T)) \simeq h^0(\PP^2, \sF) = 3,
$$
so the $\Delta$-genus $\Delta(\bar X, \bar A)=3-h^0(\bar X, \bar A)+\bar A^3$ is one
\cite[Defn.2.2]{Fuj90}. Moreover
if $D_1, D_2 \in |E+\phi^* \zeta_T|$ are two general elements, 
they are of the form
$$
D_i = E + (\varphi_{|\zeta_T|} \circ \phi)^* l_i
$$
with $l_i$ a line in $\PP^2$. Thus their intersection is
the union of $E$ 
and 
$
\fibre{(\varphi_{|\zeta_T|} \circ \phi)}{s}
$
with $s \in \PP^2$ a point. By item v.\ of Proposition \ref{construction:rk2} the surface
$\fibre{(\varphi_{|\zeta_T|} \circ \phi)}{s}$ is mapped onto an elliptic curve in $\bar X$.
This shows that two general elements $\bar D_i \in |\bar A|$ intersect along an elliptic curve
so the sectional genus $g(\bar X, \bar A)$ is also one \cite[(2.1)]{Fuj90}. 
Therefore Fujita's classification \cite[Cor.6.2]{Fuj90} yields that $\bar X$ is a del Pezzo threefold of degree one.
\end{proof}

\begin{example}$\boldsymbol{[\# 6]}$ \label{example:quadric}
Set $V:=\PP^2$ and 
$\sF $ is given by the exact sequence
$$
0 \rightarrow \sO_{\PP^2} \rightarrow T_{\PP^2}(-1) \oplus \sO_{\PP^2}(1) \rightarrow \sF \rightarrow 0.
$$
Then the conditions of the Propositions \ref{construction:general} and \ref{construction:rk2} are satisfied and we obtain a smooth Fano fourfold 
$X$ such that $\BsAX \simeq \PP^2$. Let
$$
\holom{\psi}{X}{\bar X}
$$ 
be the morphism constructed in Proposition \ref{construction:rk2}.
Then $\bar X$ is a singular Gorenstein Fano fourfold with Picard number one and index two.
A general element $D \in |\bar A|$ is the complete intersection of a quadric and a sextic
in $\PP(1^4,2,3)$.
\end{example}

\begin{proof}
The image of the morphism $\holom{\varphi_{|\zeta_T|}}{T}{\PP^4}$ is 
a smooth quadric and $T$ is the blow-up of the quadric along a line \cite{MM81} \cite[Table 1]{AW98}, so $\psi$ is birational and $\bar A^4=2$ by Remark \ref{remark-fourfold-formulas}.
Denote by $D_T \subset T$ the exceptional divisor of $\varphi_{|\zeta_T|}$: a computation shows
that $D_T \simeq \zeta_T - \mu_T^* H$ where $H$ is the hyperplane class on $V \simeq \PP^2$.
The divisor $D_T$ is the null locus of $\zeta_T$, so $\psi \circ \mu$ contracts exactly
$E \cup \fibre{\phi}{D_T}$ and $\psi$ is a divisorial Mori contraction.

Thus we have
$$
\bar A = \psi_* \mu_* (E+\phi^* \zeta_T) = \psi_* \mu_* (\phi^* \zeta_T)
$$
and by \eqref{formulaanticanonical}
$$
-K_{\bar X} = \psi_* \mu_*(-K_{X'}) = \psi_* \mu_* (\phi^* (\zeta_T+\mu_T^* H))
=  \psi_* \mu_* (\phi^* (2 \zeta_T-D_T)) = \psi_* \mu_* (2 \phi^* \zeta_T).
$$
Thus $\bar X$ is Fano of index two. Since $\psi$ is a divisorial contraction, the variety
$\bar X$ has terminal singularities. Therefore we can apply Mella's theorem \cite[Thm.1]{Mel99}
to obtain that a general element in $D \in |\bar A|$ is a Gorenstein Fano variety with at most terminal singularities. The restriction map $H^0(\bar X, \sO_{\bar X}(\bar A))
\rightarrow H^0(D, \sO_D(\bar A))$ is surjective, so by item iv.\ in Proposition \ref{construction:rk2} the linear system $|-K_D| = |\bar A|_D|$ has a unique base point.
By \cite[Thm.1.1]{JR06} this yields the description of $D$ as a complete intersection in a weighted projective space, in particular $\rho(D)=1$.
By \cite[Thm]{AW98} the unique singular point of $\bar X$, i.e.\ the point $p$, is a hypersurface
singularity. Thus we can apply the generalised Lefschetz hyperplane theorem
\cite[Thm. 3.1.17, Rem.3.1.34]{Laz04a} to obtain that $\rho(\bar X)=1$.
\end{proof}

\begin{example}$\boldsymbol{[\# 7]}$ \label{example:twisted}
Set $V:=\PP^2$ and 
$\sF$ is given by the exact sequence
$$
0 \rightarrow \sO_{\PP^2}(-1)^{\oplus 2} \rightarrow \sO_{\PP^2}^{\oplus 4} \rightarrow \sF \rightarrow 0.
$$
Then the conditions of the Propositions \ref{construction:general} and \ref{construction:rk2} are satisfied and we obtain a smooth Fano fourfold 
$X$ such that $\BsAX \simeq \PP^2$. Let
$$
\holom{\psi}{X}{\bar X}
$$ 
be the morphism constructed in Proposition \ref{construction:rk2}.
Then $\bar X$ is a smooth del Pezzo fourfold of degree one and $\psi$ is the blow-up along
a normal surface $S$ that is a set-theoretical intersection of a divisor $D_2 \in |2\bar A|$
and a divisor $D_3 \in |3 \bar A|$.
\end{example}

\begin{proof}
The morphism $\holom{\varphi_{|\zeta_T|}}{T}{\PP^3}$ is the blow-up along a twisted cubic \cite{MM81} \cite[Table 1]{AW98}, so $\psi$ is birational
and $\bar A^4=1$ by Remark \ref{remark-fourfold-formulas}.
By item iv.\ of Proposition \ref{construction:rk2}, the point $p=\psi(B)$ is the unique basepoint of the linear system $|\bar A|$. Since $\bar A^4=1$, we see that $p$ is in the smooth locus of $\bar X$. It is not difficult to see that $\psi$ is a smooth blow-up over $\bar X \setminus p$, so $\bar X$ is smooth. Since $h^0(\bar X, \sO_{\bar X}(\bar A))=4$ we obtain that
$\Delta(\bar X, \bar A)=1$.

A general element of $|\bar A|$ is the image of a general element of $|\phi^* \zeta_T|$.
Since three general elements of $|\zeta_T|$ intersect in exactly one point, we obtain that
three general elements of $|\bar A|$ intersect exactly along a smooth elliptic curve. 
Thus the sectional genus of $(\bar X, \bar A)$ is one and $\bar X$ is del Pezzo fourfold of degree one \cite[Cor.6.2]{Fuj90}.

The null-locus of $|\zeta_T|$ is the exceptional divisor over the twisted cubic $C \subset \PP^3$ so the exceptional locus of $\psi$ is $\mu(\phi^*(\fibre{\varphi_{|\zeta_T|}}{C}))$. 
Now recall that $C$ is the set-theoretic intersection of a quadric surface $S_2$ and a cubic  $S_3$ in $\PP^3$.
Set $D_i := \psi(\mu(\phi^*(\varphi_{|\zeta_T|}^* S_i)))$, then $D_i \in |i \bar A|$
and 
$$
\psi (\mu(\phi^*(\fibre{\varphi_{|\zeta_T|}}{C}))) \subset D_2 \cap D_3.
$$
Since $S_2 \cap S_3$ has multiplicity two along the twisted cubic, the right hand side has multiplicity two along a surface. We leave it to the reader to verify that we have a set-theoretical equality.
\end{proof}

\begin{example}$\boldsymbol{[\# 8]}$ \label{example:bidegreeonetwo}
Set $V:=\PP^2$ and $\sF$ is given by the exact sequence
$$
0 \rightarrow \sO_{\PP^2}(-2) \rightarrow \sO_{\PP^2}^{\oplus 3} \rightarrow \sF \rightarrow 0.
$$
Then the conditions of the Propositions \ref{construction:general} and \ref{construction:rk2} are satisfied and we obtain a smooth Fano fourfold 
$X$ such that $\BsAX \simeq \PP^2$. Let
$$
\holom{\psi}{X}{\bar X}
$$ 
be the morphism constructed in Proposition \ref{construction:rk2}.
Then $\dim \bar X=3$ and $B$ is a higher-dimensional fibre of $\psi$.
\end{example}

\begin{proof}
This is completely analogous to the proof for the Example \ref{example:flagvariety}.
\end{proof}

\begin{remark} \label{remark:flag-bidegreeonetwo}
In the situation of Example \ref{example:flagvariety} (resp.\ Example \ref{example:bidegreeonetwo}) the morphism
$$
\holom{\psi \times \pi}{X}{\bar X \times \PP^2}
$$
defines an embedding of $X$ as an element of the linear system $|p_{\bar X}^* \bar A + p_{\PP^2}^* H|$
(resp. $|p_{\bar X}^* \bar A + 2 p_{\PP^2}^* H|$) with $H$ the hyperplane divisor on $\PP^2$.
This provides an alternative construction for these examples, moreover we can apply the Lefschetz hyperplane theorem to obtain that $\rho(X)=2$.
\end{remark}

\begin{example}$\boldsymbol{[\# 9]}$ \label{example:birational-only 1 exceptional}
Set $V := \PP^1 \times \PP^1$ and 
$$
\sF := \sO_{\PP^1 \times \PP^1} \oplus \sO_{\PP^1 \times \PP^1}(1,1).
$$
Then the conditions of Proposition \ref{construction:general} and Proposition \ref{construction:rk2} are satisfied, so we obtain a Fano fourfold
with $\BsAX \simeq \PP^1 \times \PP^1$.  Let
$$
\holom{\psi}{X}{\bar X}
$$ 
be the morphism constructed in Proposition \ref{construction:rk2}.
Then $\psi$ is birational of relative Picard number two sending a divisor onto an elliptic curve $C$. Moreover, $\bar X$ is a Gorenstein Fano fourfold with Picard number one and index two, and ${\rm Sing}(\bar X)=C$.
A general element $D \in |\bar A|$ is the complete intersection of a quadric and a sextic
in $\PP(1^4,2,3)$.
\end{example}

\begin{proof}
The numerical dimension of $\zeta_T$ is three, so $\psi$ defines a birational contraction by item vi.\ in Proposition \ref{construction:rk2}. 
The null locus of $\zeta_T$ is the surface $D_T:=\PP(\sO_{\PP^1 \times \PP^1}) \subset \PP(\sF)$, so by item v.\ the divisor
$$
\fibre{\phi}{D_T} \simeq D_T \times C
$$
gets contracted onto an elliptic curve $C$. We show that $\psi$ factors as a smooth blow-up and a small contraction.
Set $D:=\mu(\phi^{-1}(D_T))$: then $D \simeq \fibre{\phi}{D_T}$ is smooth and $\mu^*D\simeq \phi^*D_T +E$ implies that
$$
\ma N_{D/X} \simeq p_{\PP^1 \times \PP^1}^* \sO_{\PP^1 \times \PP^1}(-1,-1) + p_C^* \sO_C(p)
$$
with $p \in C$ a point. We have $\sO_{D_T}(-K_{X'}) \simeq p_{\PP^1 \times \PP^1}^* \sO_{\PP^1 \times \PP^1}(1,1)$ and therefore
$$
\sO_D(-K_X) \simeq p_{\PP^1 \times \PP^1}^* \sO_{\PP^1 \times \PP^1}(1,1) + p_C^* \sO_C(p)
$$

By item vi.\ in Proposition \ref{construction:rk2}, the divisor $D$ is the exceptional locus of $\psi$ and $\Ne{\psi}=\R^+[F_1] + \R^+[{F_2}]$ is generated by the fibres $F_1, F_2$ of the two rulings of $B$, each corresponding to an elementary divisorial contraction $f_i\colon X \to X_i$ sending $D$ to a surface $S_i\subset X_i$.
Since $D \simeq \PP^1 \times \PP^1 \times C$
it is clear that the $f_i$ are smooth blow-ups along surfaces $S_i \simeq \PP^1 \times C$.
Therefore, $\psi$ has a decomposition
$$
\xymatrix{
 X \ar[r]^{f_1} \ar[d]_{f_2} \ar[rd]^{\psi} & X_1 \ar[d]^{h_1}\\
 X_2 \ar[r]_{h_2} &  \bar X
}
$$
where $h_i$ is the small contraction of $S_i$ onto $C\subset \bar X$, in particular ${\rm Sing}(\bar X)=C$. Notice that $-K_X \simeq f_i^*(-K_{X_i}) -D$ implies that $-K_{X_i}$ is $h_i$-trivial, so that $-K_{X_i}\simeq h_i^*(-K_{\bar X})$.

Now, $-K_X-D$ is trivial on the general $\pi$-fibre, thus it is the pull-back of some divisor $\sO_{\PP^1\times\PP^1}(a,b)$; restricting to $B$ we obtain $-K_X-D\simeq \pi^*\sO_{\PP^1\times\PP^1}(2,2)$, thus $D\simeq -K_{X/V}$. By item ii.\ in Proposition 
\ref{construction:rk2} we have
$$
D+\pi^*\sO_{\PP^1\times\PP^1}(1,1)\simeq\psi^*{\bar A}
$$
and so
$$
\psi^*(-K_{\bar X})\simeq -K_X+D\simeq 2D+\pi^*\sO_{\PP^1\times\PP^1}(2,2)\simeq 2\psi^*{\bar A},
$$
from which we obtain that $-K_{\bar X}$ is a Gorenstein Fano fourfold of even index. From Table \ref{table} we have that  $-K_X^4=22$ and it is easy to show that $-K_{\bar X}^4=32$, so that $\bar A^4=2$ and the index of $\bar X$ is two.

Finally, note that $\psi$ is a log-resolution of $\bar X$, so $-K_X \simeq \psi^*(-K_{\bar X}) -D$
shows that $\bar X$ has terminal singularities. In fact from
$\ma N_{S_i/X_i} \simeq \bigl(p_{\PP^1}^* \sO_{\PP^1}(-1) \otimes p_C^* \sO_C(p)\bigr)^{\oplus 2}$
we deduce that  $\bar X$ has ordinary double points along
$C$, so hypersurface singularities. Now the description of a general $D \in |\bar A|$ 
and the Picard number of $\bar X$ is obtained as in the proof of Example \ref{example:quadric}.
\end{proof}

\begin{example}$\boldsymbol{[\# 11]}$ \label{example:notelementary}
Set $V := \PP^1 \times \PP^1$ and 
$$
\sF := \sO_{\PP^1 \times \PP^1} \oplus \sO_{\PP^1 \times \PP^1}(1,0).
$$
Then the conditions of Proposition \ref{construction:general} and Proposition \ref{construction:rk2} are satisfied, so we obtain a Fano fourfold
with $\BsAX \simeq \PP^1 \times \PP^1$.  Let
$$
\holom{\psi}{X}{\bar X}
$$ 
be the morphism constructed in Proposition \ref{construction:rk2}.
Then $\dim \bar X=3$ and $\psi$ has relative Picard number two. Moreover $\bar X$ is a smooth sextic in $\PP(1,1,1,2,3)$  and $X\simeq \PP^1\times {\rm Bl}_{\bar C}\bar X$, with $\bar C \subset \bar X$ an elliptic curve.
\end{example}

\begin{proof}
The numerical dimension of $\zeta_T$ is two, so $\psi$ defines a fibration onto a threefold
by item vi.\ in Proposition \ref{construction:rk2}. 
Note that the restriction of $\zeta_T$ to the surface $N:=\PP(\sO_{\PP^1 \times \PP^1}) \subset \PP(\sF)$ is trivial, so by item v.\ the divisor
$\fibre{\phi}{N}$ gets contracted onto a curve $C$, and thus $\psi$ is not an elementary contraction.

In fact $X\simeq \PP^1\times {\rm Bl}_{\bar C}\bar X$, with $\bar C \subset \bar X$ an elliptic curve in a smooth sextic in $\PP(1,1,1,2,3)$ is a smooth Fano fourfold with $B:= \BsAX \simeq \PP^1 \times \PP^1$ and $\ma N_{B /X}^*\simeq \sF$.
\end{proof}

\begin{example}$\boldsymbol{[\# 12]}$ \label{example:blowup surfaces}
Set $V := \PP^1 \times \PP^1$ and 
$$
\sF := \sO_{\PP^1 \times \PP^1}(1,0) \oplus \sO_{\PP^1 \times \PP^1}(0,1).
$$
Then the conditions of Proposition \ref{construction:general} and Proposition \ref{construction:rk2} are satisfied, so we obtain a Fano fourfold
with $\BsAX \simeq \PP^1 \times \PP^1$.  Let
$$
\holom{\psi}{X}{\bar X}
$$ 
be the morphism constructed in Proposition \ref{construction:rk2}.
Then ${\bar X}$ is a del Pezzo fourfold of degree one with $-K_{\bar X}\simeq 3 \bar A$, and $\psi$ is the blow-up along two smooth del Pezzo surfaces of degree one $S_1$ and $S_2$ meeting at $\mbox{\rm Bs}(|\bar A|)=p$.
\end{example}

\begin{proof}
The threefold $T$ is item 25 in \cite[Table 3]{MM81}, so it can be realized as the blow-up of $\PP^3$ along two disjoint lines; the exceptional locus is the disjoint union of divisors $D_{T,1}$ and $D_{T,2}$ that correspond to the sections of $\mu_T\colon T\to \PP^1\times \PP^1$ given by the projections of $\sF$ onto its two factors: one easily checks that the null locus of $\zeta_T$
is $D_{T,1} \cup D_{T,2}$.

By item vi.\ in Proposition \ref{construction:rk2}, $\psi$ is birational with exceptional locus is $\mu(\fibre{\phi}{D_{T_1}} \cup \fibre{\phi}{D_{T_2}})$. Set $D_1=\fibre{\phi}{D_{T_1}}$ and $D_2=\fibre{\phi}{D_{T_2}}$: then by construction $D_1\cap D_2=B$ and the fibres $F_1, F_2$ of the two rulings of $B$ generate the contraction of $D_1$ and $D_2$, that is $\Ne{\psi}=\R^+[F_1] + \R^+[{F_2}]$, where $\R^+[F_1]$ yields the contraction of $D_1$ and $\R^+[F_2]$ yields the contraction of $D_2$. It follows that $\psi(\Exc(\psi))$ is a union of surfaces $S_1\cup S_2$ meeting at $p:=\psi(B)$.

By construction, $\psi$ factors as
$$
\psi\colon X \xrightarrow{f_i} X_i \xrightarrow{h_i} \bar X
$$
where $f_i$ is the contraction of $D_i$, $i=1,2$, and $h_j$ is the contraction of $f_i(D_j)$, $i\neq j\in\{1,2\}$. In particular $\bar X$ is smooth. Observe now 
that $D_{T,1} \simeq \zeta_T-\mu_T^* \sO_{\PP^1 \times \PP^1}(1,0)$
and $D_{T,2} \simeq \zeta_T-\mu_T^* \sO_{\PP^1 \times \PP^1}(0,1)$.
Therefore one has by \eqref{formulaanticanonical} 
\begin{multline*}
-K_{\bar X} \simeq \psi_* \mu_*(-K_{X'}) \simeq \psi_* \mu_* (\phi^* (\zeta_T+\mu_T^* \sO_{\PP^1 \times \PP^1}(1,1))) \\
\simeq  \psi_* \mu_* (\phi^* (3 \zeta_T-(D_{T,1}+D_{T,2}))) \simeq \psi_* \mu_* (3 \phi^* \zeta_T).
\end{multline*}
Since $\bar A \simeq \psi_* \mu_* (E+\phi^* \zeta_T) \simeq \psi_* \mu_* (\phi^* \zeta_T)$,
this shows that $\bar X$ is a del Pezzo fourfold. By Remark \ref{remark-fourfold-formulas} we have $\bar A^4=1$, so it has degree one.

In order to see that the surface $S_i$ are del Pezzo surfaces of degree one, we choose
a line $l \subset D_{T,i}$ such that $\sO_l(\zeta_T) \simeq \sO_{\PP^1}(1)$. Following the computation from the proof of Proposition \ref{construction:general} one sees
that $\fibre{\phi}{l} \rightarrow l$ is minimal rational elliptic surface and $\fibre{\phi}{l} \rightarrow \psi(\fibre{\phi}{l})$
is the contraction of one section, i.e.\ the image is a del Pezzo surface of degree one.
\end{proof}

\begin{example}$\boldsymbol{[\# 13]}$ \label{example-conicbundle}
Set $V:=\PP^1\times \PP^1$ and $\sF$ is given by the exact sequence
$$
0 \rightarrow \sO_{\PP^1 \times \PP^1}(-1,-1) \rightarrow \sO_{\PP^1 \times \PP^1}^{\oplus 3} \rightarrow \sF \rightarrow 0.
$$
Then the conditions of the Propositions \ref{construction:general} and \ref{construction:rk2} are satisfied and we obtain a smooth Fano fourfold 
$X$ such that $\BsAX \simeq \PP^1\times \PP^1$. Let
$$
\holom{\psi}{X}{\bar X}
$$ 
be the morphism constructed in Proposition \ref{construction:rk2}.
Then $\bar X$ is a smooth del Pezzo threefold of degree one, and $\psi$ factors through the blow-up $X\to\bar X\times \PP^1$ along a smooth K3 surface $S$ that is the complete intersection of two general elements in $|\sO_{\bar X \times \PP^1}(1,1)|$. 
\end{example}

\begin{proof}
Let $\bar X$ be a smooth del Pezzo threefold of degree one, and let
$$
Z_i \in |\sO_{\bar X \times \PP^1}(1,1)|
$$
be two general elements. Then $S:=Z_1 \cap Z_2$ is a smooth K3 surface that is the base scheme
of the pencil $\mathfrak d$ generated by $Z_1$ and $Z_2$; the restriction of the projection $p_{\PP^1}\colon \bar X \times \PP^1 \to \PP^1$ to $S$ yields an elliptic fibration $\eta\colon S \to \PP^1$ with a section $\ell = \mbox{\rm Bs}(|\sO_{\bar X \times \PP^1}(1,1)|) = p \times \PP^1$
Let
$$
\holom{f}{X}{\bar X \times \PP^1}
$$
be the blow-up along $S$. We show that $X$ is a smooth Fano fourfold with $B:= \BsAX \simeq \PP^1 \times \PP^1$ and $\ma N_{B /X}^*\simeq \sF$. 

Denote by $Z_{i,X}\simeq Z_i$ the strict transform of $Z_i$, $i=1,2$: then,
$$
-K_X\simeq f^*Z_j + Z_{i,X},
$$
the divisor $f^*Z_j$ is nef and big, and $Z_{i,X}$ is nef of numerical dimension one. 
It is easy to see that $-K_X$ is strictly nef, so that the basepoint-free theorem yields $-K_X$ ample. Consider now the fibration 
$\sigma:=p_{\PP^1}\circ f\colon X \to \PP^1$: 
then, the divisor
$$\sigma^*\sO_{\PP^1}(1)+ Z_{i,X}$$
is nef (thus semiample) of numerical dimension two, so it yields a morphism $\pi \colon X \to Y$ onto a normal surface. 

Let $D$ be the exceptional divisor of $f$: we have $D\simeq S\times \PP^1$ and the restriction of
$\pi$ to $D$ is exactly $\eta \times id_{\PP^1}$. In particular $Y \simeq \PP^1 \times \PP^1$.
Note also that a general $\sigma$-fibre $G$ is isomorphic to the blow-up of $\bar X$ along the elliptic curve $\bar X \cap S$, i.e.\ $G$ is a smooth Fano threefold with anticanonical base locus $\{pt\}\times\PP^1\subset B$, and $\pi_{|G}\colon G\to \PP^1$ is a smooth fibration with fibre a del Pezzo surface of degree one. This shows that $\BsAX=\fibre{f}{l} \simeq \PP^1 \times \PP^1$.

Finally, observe that $X$ has a natural embedding $\iota\colon X\hookrightarrow W:=\bar X\times \PP^1 \times\PP^1$ as the graph of the rational map
$$
|\mathfrak d|\colon \bar X \times \PP^1 \dashrightarrow \PP^1
$$
so that $X\in |p_{\bar X}^*\sO_{\bar X}(1)\otimes p_{\PP^1 \times \PP^1}^*\sO_{\PP^1 \times \PP^1}(1,1)|$ and $B\simeq \{p\} \times \PP^1 \times \PP^1$. Therefore, the inclusions $B\subset X \subset W$ give the exact sequence
$$
0\to {\ma N_{X/W}^*}_{|B} \to {\ma N_{B/W}^*} \to \ma N_{B /X}^* \to 0
$$
which is precisely
$$
0 \rightarrow \sO_{\PP^1 \times \PP^1}(-1,-1) \rightarrow \sO_{\PP^1 \times \PP^1}^{\oplus 3} \rightarrow \ma N_{B /X}^* \rightarrow 0.
$$
\end{proof}

\begin{example}$\boldsymbol{[\# 14 -16]}$ \label{example:Hirzebruch}
Set $V := \mathbb F_1$, with blow-up morphism $g\colon \mathbb F_1 \to \PP^2$ and 
$$
\sF := g^*\sF_1
$$
with $\sF_1$ as in Examples \ref{example-trivial}, \ref{example-new1} and \ref{example:flagvariety}.
Then the conditions of Proposition \ref{construction:general} and Proposition \ref{construction:rk2} are satisfied, so we obtain a Fano fourfold
with $\BsAX \simeq \mathbb F_1$.  Let
$$
\holom{\psi}{X}{\bar X}
$$ 
be the morphism constructed in Proposition \ref{construction:rk2}.
Then ${\bar X}$ factors through $\holom{\psi_1}{X_1}{\bar X}$, where $X_1$ and $\psi_1$ is the Fano fourfold from Examples \ref{example-trivial}, \ref{example-new1} and \ref{example:flagvariety} respectively, and $\holom{f}{X}{X_1}$ is the blow-up along the fibre of $X_1 \to \PP^2$ over the point $g(\Exc(g))$.
\end{example}
\begin{proof}
This will follow from Lemma \ref{lem:birational} and its proof.
\end{proof}

\section{The classification}\label{The classification}

In this whole section we will use the notation introduced in Theorem \ref{theorem-nefdivisor}.
The next steps of our classification will be built on an observation of Casagrande's:

\begin{lemma} \label{lemma-lifting} \cite[Lemma 2.6]{Casa08} 
In the situation of Theorem \ref{theorem-nefdivisor},
let 
$\holom{g}{V}{V_1}$ be the contraction of an extremal ray $R \subset \NE{V}$.
Then there exists a contraction $\holom{f}{X}{X_1}$ 
of an extremal ray in $\NE{X}$ such that
\begin{enumerate}[(i)]
\item ${\rm NE}(\pi)\cap{\rm NE}(f)=\{0\}$;
\item $\pi_* ({\rm NE}(f)) = R$.
\end{enumerate}
Moreover, we have a commutative diagramm
$$
\xymatrix{
X \ar[r]^{f} \ar[d]_{\pi} & X_1 \ar[d]^{\pi_1} \\
V \ar[r]_{g} & V_1 
}
$$
and we call $f$ a lifting of the contraction $g$.
\end{lemma}

\subsection{Birational geometry}

\begin{lemma}\label{lem:birational}
In the situation of Theorem \ref{theorem-nefdivisor}, let $e \subset V$ be a smooth rational curve such that $e^2<0$ and denote by $\holom{g}{V}{V_1}$ its contraction onto a point.
Then the following holds:
\begin{enumerate}
\item The rational curve $e \subset V$ is a $(-1)$-curve.
\item We have $\fibre{\pi}{e} \simeq e \times F_1$ with $F_1 \simeq{\pi_1}^{-1}(g(e))$ a  smooth del Pezzo surface of degree one;
\item the lifting $\holom{f}{X}{X_1}$ 
is the blow-down of $\fibre{\pi}{e}$ to $F_1 \subset X_1$ and 
$f_{|\fibre{\pi}{e}}$ is the projection onto the second factor;
\item the restriction of $f$ to $B$ is the contraction of $\sigma(e)$;
\item the fourfold $X_1$ satisfies the assumptions of Theorem \ref{theorem-main}, i.e.\
$X_1$ is a smooth Fano fourfold such that $h^0(X_1, -K_{X_1}) \geq 4$ and
the scheme-theoretical base locus 
$$
\mbox{Bs}(|-K_{X_1}|)=f(B)=:B_1
$$ 
is a smooth surface.
\item the surface $B_1 \subset X_1$ is a $\pi_1$-section and $\ma N^*_{B/X}\simeq f_B^*(\ma N^*_{B_1 / X_1})$.
\end{enumerate}
\end{lemma}

\begin{proof}
Set $D := \pi^{-1}(e)$, then $D$ is a prime divisor since $\pi$ has integral fibres.
Let $\holom{f}{X}{X_1}$ be a lifting of $g$ given by Lemma \ref{lemma-lifting},
and let $C$ be any curve contracted by $f$. Since $\pi_* ({\rm NE}(f)) = R$
and $R=\R^+ [e]$ contains a unique irreducible curve, we have $\pi(C)=e$. In particular
$$
D \cdot C = \pi^* e \cdot C = e \cdot \pi_* C < 0
$$
since $e^2<0$. 
Thus, $f$ is birational with exceptional locus contained in $D$ 
and the divisor  $-D$ is $f$-ample. 
Since  $-D|_D = \pi^* (-e|_e)$ has numerical dimension one, we obtain that
the fibres of $f$ have dimension at most one.  By \cite[Thm.2.3]{And85} this shows that $f$ is a smooth blow-up along a surface $F_1 \subset X_1$. In particular $X_1$ is smooth and $D=\mbox{Exc}(f)$ 
is smooth.
It is then clear that $\pi_1$ is equidimensional, thus flat, and that $F_1$ is the $\pi_1$-fibre over $g(e)$. Applying generic smoothness to the fibration $\pi_{|D}$, we obtain that
$D$ contains a smooth fibre $F$ of $\pi$. 

Let $C \subset D$ be a non-trivial fibre of the smooth blow-up $f$, then
$$
-K_X \cdot C = 1, \qquad -K_D \cdot C = 2,
$$
so by the adjunction formula $-D|_D \cdot C = 1$. We have $\sO_e(e) \simeq \sO_{\PP^1}(-m)$ with $m \geq 1$, so 
$$
1 = -D|_D \cdot C = \pi^* (-e|_e)  \cdot C = 
m F \cdot C
$$
implies that $m=1$ and that
the general fibres of $f_{|D}$ and $\pi_{|D}$ meet transversally in one point.
Therefore, the morphism
$$
\pi_{|D} \times f_{|D} \colon D\to e \times F_1  
$$ 
is birational, and since $\rho(D)=\rho(S)+1=\rho(e \times F_1)$ it is an isomorphism. This concludes the proof of the first three items.

{\em Proof of iv.}
Observe that $\sO_{D}(-K_X) \simeq p_e^* \sO_{\PP^1}(1)
\otimes p_{F_1}^* \sO_S(-K_{F_1})$, so if $p$ is the unique basepoint of $|-K_{F_1}|$ we have
$$
e \times p = \mbox{Bs}(|\sO_D(-K_X)|) \subset \BsAX \cap D = \sigma(e).
$$
Since $\sigma(e)$ is an irreducible curve, we see that $\sigma(e) = e \times p$ is contracted by $f$.

{\em Proof of v.}
We have $\sO_{D}(-K_X) \simeq p_e^* \sO_{\PP^1}(1)
\otimes p_{F_1}^* \sO_S(-K_{F_1})$
and $\ma N_{D/X} \simeq p_e^* \sO_{\PP^1}(-1)$, so
$$
\sO_D\bigl(f^* (-K_{X_1})\bigr) \simeq p_{F_1}^* \sO_S(-K_{F_1}).
$$
Hence $-K_{X_1}$ is ample on $F_1$ and therefore ample by  Lemma \ref{lemma:fano}.
We have
$$
h^0(X_1, -K_{X_1}) \geq h^0(X, -K_{X}) \geq 4,
$$
so we are left to show that the anticanonical base locus is $B_1$. Yet this is clear since
the general fibre of the fibration $\pi_1$ is a del Pezzo surface of degree one.

{\em Proof of vi.}
Since $B$ is a $\pi$-section, it is clear that $B_1$ is a $\pi_1$-section.
Denote by $\holom{\mu_1}{X_1'}{X_1}$ the blow-up of $B_1$ and by $E_1$ its exceptional divisor. 

We have $\fibre{f}{\sI_{B_1}} = \sI_{B}$ so the composition
$f \circ \mu$ is a principalisation of the ideal sheaf $\sI_{B_1}$. By the universal property
of the blow-up \cite[II,Prop.7.14]{Har77},
there exists a morphism $\holom{f'}{X'}{X_1'}$ such that $f \circ \mu=\mu_1 \circ f'$.
The exceptional divisor $E_1$ does not contain the locus blown-up by $f'$, so
we have $E \simeq (f')^* E_1$.
Since $\sO_E(-E)$ (resp.\ $\sO_{E_1}(-E_1)$) is the tautological divisor
of  $\PP(\ma N^*_{B/X})$ (resp.\ $\PP(N^*_{B_1 / X_1})$), this shows the last statement.
\end{proof}

\begin{corollary} \label{cor:birational}
In the situation of Theorem \ref{theorem-nefdivisor}, the following holds:
\begin{enumerate}
\item Let $V \rightarrow V_k$ be a MMP, i.e.\ a sequence of $k$ blow-downs of $(-1)$-curves. Then there exists a Fano fourfold $X_k \rightarrow V_k$ satisfying the conditions of Theorem \ref{theorem-nefdivisor} such that $X \simeq V \times_{V_k} X_k$.
\item The surface $B \simeq V$ is del Pezzo.
\end{enumerate}
\end{corollary}

\begin{proof}
Applying Lemma  \ref{lem:birational} to every step of the MMP we obtain a commutative diagram
$$
\xymatrix{
X \ar[r]_{f} \ar[d]_{\pi}  & X_1 \ar[d]^{\pi_1} \ar@{.>}[r] & X_{k} \ar[d]^{\pi_{k}}  \\
V \ar[r]^{g} & V_1 \ar@{.>}[r] & V_k 
}
$$
Since $X_i$ is the blow-up of $X_{i+1}$ along a smooth $\pi_{i+1}$-fibre, the first statement follows easily.

For ii.\ we proceed in two steps.

{\em Step 1. We show the statement assuming $(-K_V)^2 \geq 0$.} Since $(-K_V)^2 \geq 0$ and $V$ is rationally connected,
it follows from the Riemann-Roch formula that there exists an effective divisor $D \in |-K_V|$. Arguing by contradiction we assume that $-K_V$ is not nef. Then there exists an irreducible curve $C \subset D$ such that
$-K_V \cdot C<0$. Since $(-K_V)^2 \geq 0$, the support of the divisor $D$ is not equal to $C$. For every anticanonical divisor we have $p_a(D)=1$, so $p_a(C)=0$ and therefore $C \simeq \PP^1$. Since $-K_V$ is effective and $-K_V \cdot C<0$ the rational curve $C$ is not a nef divisor, i.e.\ we have $C^2<0$. Yet by item i.\ of Lemma \ref{lem:birational} this implies that $C$ is a $(-1)$-curve. Therefore $-K_V \cdot C=1>0$, a contradiction. 

If $-K_V$ is nef and $(-K_V)^2=0$, the surface $V$ contains infinitely many $(-1)$-curves \cite[Sect.5, Cor.3]{Dol83}. Yet $\NE{V} = \pi_* \NEX$ is polyhedral, so this is not possible.
If $-K_V$ is nef and big, but not ample, the surface $V$ contains a $(-2)$-curve, again a contradiction to item i. of Lemma \ref{lem:birational}.

{\em Step 2. Reduction to the first case.}
We argue by contradiction and assume that $(-K_V)^2 = -k < 0$. Then we run a MMP with $k$ steps to obtain a surface $V \rightarrow V_k$ such that $(-K_{V_k})^2=0$. By item i.\ there exists a smooth Fano
fourfold $X_k \rightarrow V_k$ satisfying the conditions of Theorem \ref{theorem-nefdivisor}.
Since $(-K_{V_k})^2 \geq 0$, we can apply the first step to obtain that $V_k$ is del Pezzo.
Yet this contradicts $(-K_{V_k})^2=0$.
\end{proof}

Corollary \ref{cor:birational} allows to focus the classification on the case
where $V \simeq \PP^1 \times \PP^1$ or $V \simeq \PP^2$.
In these cases we want to show the existence of a lifting that contracts the base locus $B$, and we want to use the lifting to determine $\ma N_{B/X}^*$. 
We start by introducing the main technical tool for the next steps:

\begin{lemma} \label{lemma:setup-classA}
In the situation of Theorem \ref{theorem-nefdivisor}, denote by $\zeta_T$ the divisor on $T$ that identifies to 
$\sO_E(-E) \simeq \ma N^*_{E/X'} \simeq \sO_{\PP(\ma N_{B/X}^*)}(1)$ under the isomorphism $\phi_E$. Then the following holds:
\begin{enumerate}
\item The divisor $\zeta_T$ is pseudoeffective.
\item We have
$$
E + \phi^* \zeta_T \simeq \mu^* A
$$
with $A$ Cartier divisor on $X$ such that $\sO_B(A) \simeq \sO_B$.
\item The divisor $\zeta_T$ is nef if and only if $A$ is nef.
\item If $A$ is not nef, there exists a prime divisor $D_T \subset T$ such that the restriction $\sO_{D_T}(\zeta_T)$ is not pseudoeffective.
\item If for some prime divisor $D \subset X$ the restriction $\sO_D(A)$ is not pseudoeffective, then $D = \mu_*(\phi^* D_T)$
for a  prime divisor $D_T \subset T$ such that the restriction $\sO_{D_T}(\zeta_T)$ is not pseudoeffective.
\end{enumerate}
\end{lemma}

\begin{proof}
By Fact \ref{fact:weierstrass} we have $\sO_E(-E) \simeq \phi_* \omega_{X'/T}$ and the latter is pseudoeffective
by Viehweg's theorem \cite{Vie82}. This shows that $\zeta_T$ is pseudoeffective.

The restriction of  $\phi^* \zeta_T$ to $E$ is isomorphic to $\sO_E(-E)$, so we have
$\sO_E(E + \phi^* \zeta_T) \simeq \sO_E$. In particular $E + \phi^* \zeta_T$ has degree zero on the extremal ray contracted by $\mu$. This shows the existence of $A$. Moreover, $\sO_B(A) \simeq \sO_B$ since its pull-back to $E$ is trivial.

{\em Proof of iii.}
If $\zeta_T$ is nef, any curve $C' \subset X'$ such that $(E + \phi^* \zeta_T) \cdot C' < 0$ would be contained in $E$. Yet we have
just seen that $\sO_E(E + \phi^* \zeta_T) \simeq \sO_E$ so this is not possible.
Vice versa, assume that there exists a curve $C \subset T$ such that $\zeta_T \cdot C<0$. Denote by $W_C$ the restriction of the vector bundle $W$ introduced in Proposition \ref{prop:weierstrass} to the curve $C$, and let $C' :=\PP(\sO_C)$
be the curve corresponding to the quotient $W_C \rightarrow \sO_C$. The divisor class of $X' \subset \PP(W)$ is
$3 \zeta_W +\Phi^* (6 \zeta_T)$, so we have $X' \cdot C'<0$. Thus the curve $C'$ is contained in $X'$ and it is disjoint from
from the section $E = \PP(\sO_T(-3 \zeta_T))$. Therefore 
$$
\mu^* A \cdot C' = (E+\phi^* \zeta_T) \cdot C' = \zeta_T \cdot C < 0.
$$

{\em Proof of iv.} Assume that for every divisor $D \subset T$ the restriction $\sO_D(\zeta_T)$ is pseudoeffective. Since $\dim T=3$,
it follows from the divisorial Zariski decomposition that there exist at most countably many curves $C \subset T$ such that
$\zeta_T \cdot C<0$. The anticanonical system $|-K_{X'}|$ is the pull-back of a globally generated linear system on $T$, 
so a very general element $Y' \in |-K_{X'}|$ fibres over a surface $U \subset T$ such that
$\sO_U(\zeta_T)$ is nef. Now the proof of iii.\ shows that 
$$
\sO_{Y'}(E + \phi^* \zeta_T) \simeq \sO_{Y'}(\mu^* A) 
$$ 
is nef. Therefore the restriction $\sO_Y(A)$ is nef and we conclude by \cite[Theorem]{Bor91}  that $A$ is nef.

{\em Proof of v.} Denote by $D' \subset X'$ the strict transform of $D \subset X$. By assumption,
the restriction $E + \phi^* \zeta_T \simeq \mu^* A$ to $D'$ is not pseudoeffective. Since $E \neq D'$ this implies that
the restriction of $\phi^* \zeta_T$ to $D'$ is not pseudoeffective. We know by i.\ that $\zeta_T$ is pseudoeffective, so the 
morphism $D' \rightarrow T$ can't be surjective. Since $\phi$ is flat with integral fibres, we obtain that
$D' = \phi^* D_T$ with $D_T$ a prime divisor on $T$. Moreover the restriction of $\zeta_T$ to $D_T$ is not pseudoeffective.
\end{proof}

Our goal will be to show that $A$ and therefore $\ma N_{B/X}^*$ is nef. Lemma \ref{lemma:setup-classA} essentially reduces this to excluding the existence of certain negative divisors $D_T \subset T$. 
This is surprisingly challenging and will require further geometric input.

\subsection{The case $V \simeq \PP^2$}

We start with a technical lemma, then apply it in our situation.

\begin{lemma} \label{lemma:split}
Let $\sF \rightarrow \PP^2$ be a vector bundle of rank two that is pseudoeffective but not nef,
and denote by $\holom{p}{\PP(\sF)}{\PP^2}$ the projectivisation.
Assume that there exists a non-constant morphism $\holom{j}{S  \simeq \PP^2}{\PP(\sF)}$.
Then 
$$
\sF \simeq W_1 \oplus W_2
$$
with $W_i \simeq \sO_{\PP^2}(a_i)$ with $a_1 \geq 0$ and $a_2<0$. 
\end{lemma}

\begin{proof}
It is elementary to see that the composition $\holom{\alpha := p \circ j}{S }{\PP^2}$ is a 
surjective finite morphism. We will first show the statement for $\alpha^* \sF$, then descend the properties via a stability argument.
The pull-back
$$
\holom{p'}{\PP(\alpha^* \sF)}{S \simeq \PP^2}
$$
has a section, so we obtain an extension
$$
0 \rightarrow L_1 \rightarrow \alpha^* \sF \rightarrow L_2 \rightarrow 0
$$
where the $L_i$ are line bundles on $\PP^2$. Since $H^1(\PP^2, L_1 \otimes L_2^*)=0$
the extension splits and $\alpha^* \sF \simeq L_1 \oplus L_2$.
Now recall that $\alpha^* \sF$ is pseudoeffective but not nef, so up to renumbering we have
$L_i \simeq \sO_{\PP^2}(b_i)$ with $b_1 \geq 0$ and $b_2<0$. Thus $\alpha^* \sF$
is not $\alpha^* H$-semistable. Yet this implies that $\sF$ itself is not $H$-semistable \cite[Lemma 6.4.12]{Laz04b} and we denote by $F \subset \sF$ the destabilising rank one subsheaf. Since $\alpha^* F$ destabilises
$\alpha^* \sF$ we have $\alpha^* F=L_1$ and therefore $F \subset \sF$ is a subbundle.
Now we can argue as before to obtain the statement.
\end{proof}

\begin{lemma} \label{lemma:conormal-nef}
In the situation of Theorem \ref{theorem-nefdivisor}, assume that $B \simeq V \simeq \PP^2$.
Assume also that there exists a non-constant morphism $\holom{j}{S  \simeq \PP^2}{T}$.
Then $\ma N_{B/X}^*$ is nef.
\end{lemma}

\begin{proof}
We argue by contradiction and assume that $\ma N_{B/X}^*$ is not nef. Since $\ma N_{B/X}^*$ is pseudoeffective by Lemma \ref{lemma:setup-classA}, 
we can apply Lemma \ref{lemma:split} to obtain 
$$
\ma N_{B/X}^* \simeq \sO_{\PP^2}(a_1) \oplus \sO_{\PP^2}(a_2)
$$
with $a_1\geq0$ and $a_2<0$. We have $\sO_B(-K_X) \simeq \sO_{\PP^2}(b)$
with $b \geq 1$, and using the normal sequence for $B \subset X$ we see that
$b=3-a_1-a_2$. Since $b \geq 1$,  this implies
$a_1 \geq -a_2 + 2 \geq 3$.

{\em Step 1. We show that $a_1=3$ and $b=-a_2$.}
Consider the subvariety $B ':= \PP(\sO_{\PP^2}(a_2)) \subset E \subset X'$. Then we have
$\sO_{B'}(-E) \simeq \sO_{\PP(\ma N_{B/X}^*)}(1) \otimes \sO_{B'} \simeq \sO_{\PP^2}(a_2)$ and therefore
$$
\sO_{B'}(-K_{X'}) \simeq \sO_{B'}(-\mu^* K_{X} - E) \simeq \sO_{\PP^2}(b+a_2).
$$
Since $-K_{X'}$ is nef, we obtain that $b \geq -a_2$. Since $b=3-a_1-a_2$, this implies $a_1 \leq 3$.
Yet $a_1 \geq 3$ so all the inequalities are equalities.

{\em Step 2. We obtain a contradiction by showing that $X'$ is not smooth.}
We have $T \simeq E \simeq \PP(\ma N_{B/X}^*)$ and we denote by 
$N \subset T$ the section corresponding to the negative quotient 
$\ma N_{B/X}^* \rightarrow \sO_{\PP^2}(-b)$.
Denote by $\zeta_T$ the tautological class on $\PP(\ma N_{B/X}^*)$, then we have
$\zeta_T \simeq N + \mu_T^*(3H)$
with $H$ the hyperplane class on $V \simeq \PP^2$.
Using the notation of Proposition \ref{prop:weierstrass} we consider 
the surface $S \subset \PP(W)$ determined by the quotient
$$
W_N  \simeq \sO_{\PP^2} \oplus \sO_{\PP^2}(2b) \oplus \sO_{\PP^2}(3b) \twoheadrightarrow \sO_{\PP^2}
$$
Then, we have an exact sequence
$$
0 \rightarrow \ma N_{\PP(W_N)/\PP(W)}^* \otimes \sO_S \rightarrow \ma N_{S/\PP(W)}^* \rightarrow \ma N_{S/\PP(W_N)}^*    \rightarrow 0
$$
which, using $\sO_N(\zeta_T) \simeq \sO_{\PP^2}(-b)$, can be simplified to 
$$
0 \rightarrow \sO_{\PP^2}(b+3)   \rightarrow \ma N_{S/\PP(W)}^* \rightarrow \sO_{\PP^2}(2b) \oplus \sO_{\PP^2}(3b) 
\rightarrow 0.
$$
By Proposition \ref{prop:weierstrass}, the divisor class of $X' \subset \PP(W)$
is $3 \zeta_W + \Phi^*(6 \zeta_T)$. In particular $\sO_S(X') \simeq \sO_{\PP^2}(-6b)$,
so the surface $S$ is contained in $X'$.
Tensoring the exact sequence with $\sO_S(X')$ we obtain an extension
$$
0 \rightarrow \sO_{\PP^2}(-5b+3)   \rightarrow \ma N_{S/\PP(W)}^* \otimes \sO_S(X') \rightarrow \sO_{\PP^2}(-4b) \oplus \sO_{\PP^2}(-3b) 
\rightarrow 0.
$$
Since $b \geq 1$, this implies $H^0(S, \ma N_{S/\PP(W)}^* \otimes \sO_S(X'))=0$.
By \cite[Lemma 1.7.4]{BS95} this implies that $X'$ is singular along $S$, a contradiction.
\end{proof}

\begin{proposition} \label{prop:P2}
In the situation of Theorem \ref{theorem-nefdivisor}, assume that $B \simeq V \simeq \PP^2$.
Then $\ma N_{B/X}^*$ is nef.
\end{proposition}

\begin{proof} We argue by contradiction assume that $\ma N_{B/X}^*$ is not nef. 
We use the notation of Lemma \ref{lemma:setup-classA}
and consider the class
$$
E + \phi^* \zeta_T \simeq \mu^* A.
$$
By item iii.\ of Lemma \ref{lemma:setup-classA}  we know that $A$ is not nef, 
so there exists a Mori contraction $$\holom{f}{X}{X_1}$$ such that $A \cdot C<0$ for every contracted curve $C$. 
Note that $B$ is not contracted by $f$: by item ii.\ of Lemma \ref{lemma:setup-classA} the restriction of $A$ to $B$ is trivial,
so $B$ does not contain any $A$-negative curve.

{\em 1st case. Assume that $f$ is small.} 
Let $S \simeq \PP^2$ be a surface contracted by $f$ \cite[Thm.1.1]{Kaw89}. If $S \cap B = \emptyset$, the
strict transform $S' \subset X'$ is isomorphic to $\PP^2$. Since $\phi$ is equidimensional of relative dimension one, the morphism $\phi\colon S' \simeq \PP^2 \rightarrow T$ is not constant. By Lemma \ref{lemma:conormal-nef} this implies that $\ma N_{B/X}^*$ is nef, a contradiction.

Thus we have $S \cap B \neq \emptyset$. Note first that the intersection is finite: otherwise $f$ contracts a curve in $S \cap B \subset B$ and therefore every curve in $B \simeq \PP^2$. Yet we know that $B$ is not contracted by $f$.

If $S \cap B$
is finite, let $d_B \subset S \simeq \PP^2$ be a general line passing through a point $p \in S \cap B$,
and let $d_B' \subset X'$ be its strict transform. Then $E \cdot d_B' > 0$ and therefore
$-K_X \cdot d_B=1$ implies that $-K_{X'} \cdot d_B'=0$. Thus, if $S' \subset X'$ is the strict transform of $S$, it is covered by
$-K_{X'}$-trivial curves. On the other hand a general line $d \subset S$ is disjoint from $S \cap B$, so
$-K_{X'} \cdot d' = -K_X \cdot d=1$. 

Since $-K_{X'} \simeq \varphi^* A$ with $A$ an ample divisor
on $T'$ (here we use the fibration $\varphi$ from Setup \ref{setup-general}, not $\phi$ from Theorem \ref{theorem-nefdivisor}) we see that $\varphi$
contracts the surface $S'$ onto a curve in $T'$.
The smooth rational curves $d_B' \subset S'$ are contained in general fibres of $S' \rightarrow \varphi|_{S'} \subset T'$. 
By Lemma \ref{lemma:technical}, 
the fibration $\varphi$ is flat with integral fibres over the complement of finitely many points in $T'$.
Thus $d_B'$ is a curve of arithmetic genus one. Yet $d_B' \simeq d_B$ is a line in $\PP^2$, so we have reached a contradiction. 

{\em 2nd case. Assume that $f$ is divisorial.}
Denote by $D$ the exceptional divisor, then by item v.\ of Lemma \ref{lemma:setup-classA} we have $D =\mu(\fibre{\phi}{D_T}$ with $D_T \subset T$ a prime divisor such that the restriction of $\zeta_T$ to $D_T$ is not pseudoeffective. In particular $D_T$ is not nef.
Since $T \rightarrow V \simeq \PP^2$ is a $\PP^1$-bundle, this implies that $D_T$ is not a pull-back from $V$, so
$D_T$ surjects onto $V$. Therefore $D$ contains the base locus $B$.
We have seen that $B$ is not contracted by $f$, so we have $\dim f(D) \geq \dim f(B) \geq 2$.
Thus the exceptional divisor $D$ is contracted onto a surface and $B \simeq \PP^2$ has a finite surjective map onto $f(D)$. A general fibre of $f$ is a smooth rational curve $d$ such that $-K_X \cdot d=1$ and its strict transform $d' \subset X'$ satifies $E \cdot d' = \mbox{length}(B \cap d)>0$. Thus we see that $-K_{X'} \cdot d'=0$ and the divisor $D' \subset X'$ is covered by $-K_{X'}$-trivial curves. By Lemma \ref{lemma:technical} 
the fibration $\varphi$ does not contract a divisor, so $\varphi(D')$ has dimension at least one
and $d'$ is contained in a general fibre of $D' \rightarrow \varphi(D')$.
By the same lemma,
the fibration $\varphi$ is flat with integral fibres over the complement of finitely many points in $T'$.
Thus $d'$ is a curve of arithmetic genus one. Yet $d' \simeq d \simeq \PP^1$, so we have reached a contradiction. 
\end{proof}

\begin{lemma} \label{lemma:P2-classifyNB}
In the situation of Theorem \ref{theorem-nefdivisor}, assume that $B \simeq V \simeq \PP^2$. 
Then $\ma N_{B/X}^*$ is isomorphic to a vector bundle $\sF$ of the following form:
$$
\sO_{\PP^2}^{\oplus 2}, \qquad  T_{\PP^2}(-1), \qquad 0 \rightarrow  \sO_{\PP^2}(-2) \rightarrow \sO_{\PP^2}^{\oplus 3} \rightarrow \sF \rightarrow 0 
$$
or
$$
\sO_{\PP^2}(1)^{\oplus 2}, \qquad \sO_{\PP^2} \oplus \sO_{\PP^2}(1), \qquad \sO_{\PP^2} \oplus \sO_{\PP^2}(2), 
$$
or
$$
0 \rightarrow \sO_{\PP^2}(-1)^{\oplus 2} \rightarrow \sO_{\PP^2}^{\oplus 4} \rightarrow \sF \rightarrow 0
$$
or
$$
0 \rightarrow \sO_{\PP^2} \rightarrow T_{\PP^2}(-1) \oplus \sO_{\PP^2}(1) \rightarrow \sF \rightarrow 0.
$$
In all the cases the vector bundle $\sF \simeq \ma N_{B/X}^*$ is globally generated and
$-K_{B} - \det \sF$ is ample and globally generated.
\end{lemma}

\begin{proof}
Since $\ma N_{B/X}^*$ is nef by Proposition \ref{prop:P2}, we know by item iii.\ of Lemma \ref{lemma:setup-classA} that
$$
E+ \phi^* \zeta_T \simeq \mu^* A 
$$
is nef and therefore semiample since $A$ is a nef divisor on a Fano manifold. 
We denote by  
$$
\holom{\psi}{X}{\bar X}
$$
the morphism defined by some multiple of $A$.

By the canonical bundle formula \cite{Amb05} we can find a boundary divisor $\Delta$ such that
$(T, \Delta)$ is klt and $-(K_{T}+\Delta) = \phi_* (-K_{X'})$. Since $-K_{X'}$ is the pull-back of a nef and big divisor on $T$, this shows that $(T, \Delta)$ is a weak Fano threefold
and the nef divisor $\zeta_T$ is semiample. We denote by
$$
\holom{\eta}{T}{\bar T}
$$ 
the morphism defined by some multiple of $\zeta_T$. The threefold $T$ has Picard number two, so $\eta$ is either an isomorphism
or a contraction of a (not necessarily $K_T$-negative) extremal ray.

The classification will be obtained by relating the geometric properties of these two maps. In fact, since  the restriction of $E+ \phi^* \zeta_T$ to $E \subset X'$ is trivial by Lemma \ref{lemma:setup-classA}, we have
\begin{equation}\label{self-intersections}
\mu^*(A^k) = (E+ \phi^* \zeta_T)^k = (E+ \phi^* \zeta_T) \cdot \zeta_T^{k-1} 
\end{equation}
for all $k \in \N$. We make a case distinction in terms of the null locus of $\zeta_T$ (cf.\ Remark \ref{remark-nulllocus}).

{\em 1st case. Assume that $\zeta_T$ is not big, so $\zeta_T^3=0$.}
In this case $\eta$ is of fibre type. Since $(T, \Delta)$ is weak Fano,
we see that $-K_{T}$ has positive degree on the general $\eta$-fibre. Yet $\rho(T)=2$
and $-K_{T}$ of positive degree on both extremal rays imply that $T$ is a smooth Fano threefold. By the Mori-Mukai classification \cite{MM81} (cf.\ also \cite[Thm.]{SW90} for the case of $\PP^1$-bundles) we have three possibilities:
$$
T \simeq \PP^1 \times \PP^2, \qquad T \simeq \PP(T_{\PP^2})
$$
or $T \subset \PP^2 \times \PP^2$ is a divisor of bidegree $(1,2)$. 
Since we assume that $\zeta_T$ is not big, we see that in the first two cases
$\ma N_{B/X}^* \simeq \sO_{\PP^2}^{\oplus 2}$ and $\ma N_{B/X}^* \simeq T_{\PP^2}(-1)$, respectively.
In the last case consider the exact sequence
$$
0 \rightarrow \sO_{\PP^2 \times \PP^2}(-T) \rightarrow 
\sO_{\PP^2 \times \PP^2} \rightarrow \sO_{T} \rightarrow 0. 
$$
Twisting the exact sequence by $p_2^* \sO_{\PP^2}(1)$ and pushing forward via $p_1$, we obtain
$T \simeq \PP(\sF)$ where $\sF$ is given by 
$$
0 \rightarrow \sO_{\PP^2}(-2) \rightarrow \sO_{\PP^2}^{\oplus 3} \rightarrow \sF \rightarrow 0.
$$
The vector bundle $\sF$ is nef and not big, while any twist $\sF(d)$ with $d \neq 0$
is either big or not nef. Thus $\ma N_{B/X}^* \simeq \sF$.
In conclusion we obtain the first three cases.

{\em 2nd case. Assume that $\zeta_T$ is big and its null locus contains no divisor.}
In this case the null locus $\mbox{\rm Null}(\zeta_T)$ is a finite union of curves $C_i \subset T$.
Therefore $\mbox{\rm Null}(E+\phi^*\zeta_T)$ is contained in $E \cup \cup_i \fibre{\phi}{C_i}$,
and
$$
\mbox{\rm Null}(A) \subset B \cup \mu(\cup_i \fibre{\phi}{C_i})
$$
is contained in a union of surfaces.
Thus $\psi$ is small and $N_{B/X}^* \simeq  \sO_{\PP^2}(1)^{\oplus 2}$ by \cite[Thm.1.1]{Kaw89}.

{\em 3rd case. Assume that $\zeta_T$ is big and its null locus contains a divisor $D_T$.}
Since $\rho(T)=2$ the exceptional locus of $\eta$ coincides with the divisor $D_T$
and $\mbox{\rm Null}(\zeta_T)=D_T$. Therefore
 $\mbox{\rm Null}(E+\phi^*\zeta_T)$ is contained in $E \cup \fibre{\phi}{D_T}$,
and
$$
\mbox{\rm Null}(A) \subset B \cup \mu(\fibre{\phi}{D_T}).
$$
Now observe that the exceptional divisor $D_T \subset T$ is not a pull-back from $V \simeq \PP^2$, so
$\fibre{\phi}{D_T}$ has positive degree on the fibres of the ruling $E \rightarrow B$. Thus we have
$B \subset \mu(\fibre{\phi}{D_T})=:D$, i.e.\ the exceptional locus of $\psi$ is contained in the prime divisor $D$.
If $\psi$ is small we obtain again $\ma N_{B/X}^* \simeq  \sO_{\PP^2}(1)^{\oplus 2}$.
For the rest of the proof we assume that the exceptional locus of $\psi$ is equal to $D$.

{\em Subcase a: the image $\eta(D_T)$ is a point.}
This is equivalent to the property that the restriction of $\zeta_T$ to $D_T$ is numerically trivial. By the same computation as
in item v.\ of Proposition \ref{construction:rk2} we obtain that
$\fibre{\phi}{D_T} \simeq D_T \times C$ with $C$ a curve of arithmetic genus one. Moreover the restriction
of $\mu^* A$ to $\fibre{\phi}{D_T}$ identifies to $D_T \times pt$, so the numerical dimension of $\sO_D(A)$ is one.
This shows that $\psi$ contracts $D \simeq D_T \times C$ onto the curve $C$. Since $B$ is contracted by $\psi$
we have $D_T \simeq \PP^2$ and $D \simeq \PP^2 \times C$. In particular $\psi$ is an elementary Mori contraction
of type $(3,1)$ and by  Takagi's theorem \cite[Main Theorem]{Tak99} the curve $\psi(D)$ is smooth and the fibration
$$
D \rightarrow f(D)
$$
a projective bundle. Now it is straightforward to see that  $\ma N_{B/X}^* \simeq \sO_{\PP^2} \oplus \sO_{\PP^2}(1)$
or $\ma N_{B/X}^* \simeq \sO_{\PP^2} \oplus \sO_{\PP^2}(2)$.

{\em Subcase b: the image $\eta(D_T)$ is a curve.}
This is equivalent to the property that the restriction of $\zeta_T$ to $D_T$ has numerical dimension one.
From there we deduce that the restriction of $\mu^* A$ to $\fibre{\phi}{D_T}$ has numerical dimension two,
so $\psi$ contracts the divisor $D$ onto a surface in $\bar X$. Since $B \subset D$ is mapped onto a point,
it is contained in a higher-dimensional fibre $\fibre{\psi}{\bar x}$.
If $\fibre{\psi}{\bar x}$ is irreducible, we obtain  from \cite[Thm.]{AW98}
that the conormal bundle is one of the last two items of our list.
If $\fibre{\psi}{\bar x}$ is reducible, we obtain from the same result that
it is a reducible quadric $Q \simeq B \cup B' \simeq \PP^2 \cup \PP^2$
and the conormal bundle of the irreducible component $B$ is either
 $T_{\PP^2}(-1)$ or $\sO_{\PP^2} \oplus \sO_{\PP^2}(1)$.
\end{proof}

\begin{corollary} \label{corollary:F1-classifyNB}
In the situation of Theorem \ref{theorem-nefdivisor}, assume that $B \simeq V \simeq \mathbb F_1$ is the first Hirzebruch surface
and denote by $\holom{g}{V}{\PP^2}$ the blow-down of the $(-1)$-curve. 
Then $\ma N_{B/X}^*$ is isomorphic to a vector bundle $\sF$ of the following form:
$$
g^* \sO_{\PP^2}^{\oplus 2}, \qquad  g^* T_{\PP^2}(-1), \qquad g^*(\sO_{\PP^2} \oplus \sO_{\PP^2}(1)).
$$
In all the cases the vector bundle $\sF \simeq \ma N_{B/X}^*$ is globally generated and
$-K_{B} - \det \sF$ is ample and globally generated.
\end{corollary}

\begin{proof}
We  know by Lemma \ref{lem:birational} that $X$ is a blow-up $\holom{f}{X}{X_1}$ with $X_1$ a smooth Fano fourfold satisfying
the assumption of Theorem \ref{theorem-nefdivisor} such that $B_1 := \mbox{\rm Bs}(|-K_{X_1}|) \simeq \PP^2$. Moreover by item vi.\ of the same lemma we have $\ma N_{B/X}^* \simeq f_B^* \ma N_{B_1/X_1}^*$, so it is clear that $\ma N_{B/X}^*$ is the pull-back of one of the bundles appearing in Lemma \ref{lemma:P2-classifyNB}. 
By Remark \ref{remark-restrictionB} we have
$$
\sO_{B_1}(-K_{X_1}) \simeq \sO_{\PP^2}(-K_{\PP^2}-\det \ma N_{B_1/X_1}^*). 
$$
A look at the list in Lemma \ref{lemma:P2-classifyNB} shows that
for the bundles different
from 
$$
\sO_{\PP^2}^{\oplus 2}, \qquad  T_{\PP^2}(-1), \qquad \sO_{\PP^2} \oplus \sO_{\PP^2}(1)
$$ the restriction of $-K_{X_1}$ to the base locus $B_1$ is isomorphic to $\sO_{\PP^2}(1)$. Since $X$ is the blow-up of $X_1$ along a $\pi$-fibre, this implies that the restriction
of $-K_X$ to $B$ is isomorphic to $\sO_{\mathbb F_1}(l)$ with $l$ a line of the ruling. Yet $X$ is Fano, so this is not possible.
\end{proof}

\subsection{The case $V \simeq \PP^1 \times \PP^1$}

We start with some technical lemmas.

\begin{lemma} \label{lemma:ruling-technical}
Let $\sF \rightarrow V \simeq \PP^1 \times \PP^1$ be a rank two vector bundle such that
for every line $d_i \subset V$ the restriction $\sF \otimes \sO_{d_i}$
is nef. Set 
$$
\holom{\mu_T}{T:=\PP(\sF)}{\PP^1 \times \PP^1}
$$
and denote by $\zeta_T$ the tautological class. Assume that there exists a prime divisor
$$
D_T \simeq \PP^1 \times \PP^1 \subset \PP(\sF)
$$
that maps surjectively onto the base $V$. Then the restriction of
$\zeta_T$ to $D_T$ is pseudoeffective.
\end{lemma}

\begin{proof}
We argue by contradiction that 
the restriction of
$\zeta_T$ to $D_T$ is not pseudoeffective, so
$\sO_{D_T}(\zeta_T) \simeq \sO_{\PP^1 \times \PP^1}(a,b)$ with $a<0$ or $b<0$. Assume that $a<0$, the other case is analogous.

Since the morphism $\mu_T|_{D_T}$ is surjective and $\rho(D_T)=\rho(V)$,
the pull-back of divisors induces a bijection between the nef cones. Dually, the push-forward
of curves induces a bijection the Mori cones, so a line $l_1 := \PP^1 \times pt. \subset D_T$
is mapped surjectively onto a line in one of the rulings of $V$, say $d_1$.
By assumption $\sF \otimes \sO_{d_1}$ is nef, so $\zeta_T$
is nef on the curve $l_1$. Yet 
$$
\sO_{D_T}(\zeta_T) \otimes \sO_{l_1} \simeq \sO_{\PP^1}(a)
$$
is antiample, a contradiction.
\end{proof}

\begin{lemma} \label{lemma:ruling-technical2}
Let $\sF$ be a nef rank two vector bundle on $\PP^1 \times \PP^1$
such that the restriction
to every fibre of the projections $p_i$ is isomorphic to $\sO_{\PP^1} \oplus \sO_{\PP^1}(1)$.
Then $\sF$ is of the following form:
$$
\sO_{\PP^1 \times \PP^1} \oplus \sO_{\PP^1 \times \PP^1}(1,1), \qquad \sO_{\PP^1 \times \PP^1}(1,0) \oplus \sO_{\PP^1 \times \PP^1}(0,1)
$$
or
$$
0 \rightarrow \sO_{\PP^1 \times \PP^1}(-1,-1) \rightarrow \sO_{\PP^1 \times \PP^1}^{\oplus 3} \rightarrow \sF \rightarrow 0.
$$
\end{lemma}

\begin{proof}
Note first that $(p_1)_* (\sF \otimes \sO_{\PP^1 \times \PP^1}(0,-1))$ is a line bundle
and the natural morphism
$$
p_1^* (p_1)_* (\sF \otimes \sO_{\PP^1 \times \PP^1}(0,-1))
\rightarrow \sF \otimes \sO_{\PP^1 \times \PP^1}(0,-1)
$$
is injective in every point. Thus we can write $\sF$ as an extension of line bundles
$$
0 \rightarrow \sO_{\PP^1 \times \PP^1}(a,1) \rightarrow \sF \rightarrow \sO_{\PP^1 \times \PP^1}(b,0) \rightarrow 0
$$
and $a+b=1$ since the restriction to the $p_2$-fibres has degree one. Since $\sF$ is nef, we obviously have $b \geq 0$ and we claim that we also have $b \leq 2$. Admitting the claim, it is not difficult to study the extension and deduce the result.

{\em Proof of the claim.} We have $c_1(\sF)=\sO_{\PP^1 \times \PP^1}(1,1)$
and therefore $c_1^2(\sF)=2$. Moreover
$$
c_2(\sF) = c_1(\sO_{\PP^1 \times \PP^1}(1-b,1)) \cdot c_1(\sO_{\PP^1 \times \PP^1}(b,0)) =b.
$$
Thus if $b \geq 3$ we obtain $c_2(\sF)>c_1^2(\sF)$, contradicting the Chern class inequalities for nef vector bundles \cite[Ex.8.3.4]{Laz04b}.
\end{proof}

\begin{lemma} \label{lemma:divisor-contraction}
In the situation of Theorem \ref{theorem-nefdivisor}, let $\holom{f}{X}{X_1}$ be a divisorial Mori contraction of an extremal ray on $X$ with exceptional
divisor $D$. Then $f(D)$ is not a point.
\end{lemma}

\begin{proof}
We argue by contradiction.
Let $F$ be a $\pi$-fibre such that the intersection $D \cap F$ is not finite. Then there
exists a curve $C \subset D \cap F$, so the ray $\R^+ [C]$ is contracted by $\pi$.
Since $f$ is the contraction of $\R^+ [C]$, this shows that $D$ is contracted by $\pi$
onto a point. Yet $\pi$ is equidimensional with two-dimensional fibres.
\end{proof}

\begin{lemma} \label{lemma:ruling-preparation}
In the situation of Theorem \ref{theorem-nefdivisor}, assume that $B \simeq V \simeq \PP^1 \times \PP^1$. 
Then the following statements hold:
\begin{enumerate}
\item $X$ does not contain a $\PP^2$ or a quadric cone of dimension two.
\item $X$ does not admit a small Mori contraction.
\item Let $A$ be a pseudoeffective divisor on $X$ such that the restriction to every prime divisor is pseudoeffective. Then $A$ is nef.
\end{enumerate}
\end{lemma}

\begin{proof}

{\em Proof of i.\ and ii.}
Note that ii.\ follows immediately from i.\ and \cite[Thm.1.1]{Kaw89}.

For the proof of i.\ note that first that the morphism $\holom{\pi|_S}{S}{V \simeq \PP^1 \times \PP^1}$ must be constant since $\rho(S)=1$. Yet we know that all the $\pi$-fibres
are integral and satisfy $(-K_F)^2 =1$. Clearly this does not hold for $\PP^2$ or a quadric.
 
{\em Proof of iii.} 
If $A$ is not nef, there exists by the cone theorem a $K_X$-negative extremal ray $\R^+ [\Gamma]$ such that $A \cdot \Gamma<0$. Since $A$ is pseudoeffective, the contraction is birational.
If the contraction is divisorial, the curves in the extremal ray cover the exceptional divisor,
so the restriction of $A$ to the exceptional divisor is not pseudoeffective. Thus the contraction is small, a possibility excluded by ii.
\end{proof}

The next step in our study is to understand the structure of the lifting given by Lemma \ref{lemma-lifting}:

\begin{lemma}\label{lem:ruling}
In the situation of Theorem \ref{theorem-nefdivisor}, assume that $B \simeq V \simeq \PP^1 \times \PP^1$. 
Let $\holom{g}{V}{V_1 \simeq \PP^1}$ be one of the rulings, and 
let $G \subset X$ be a general fibre of $g \circ \pi$. Then one of the following holds:
\begin{itemize}
\item $G \simeq \PP^1 \times S$ with $S$ a del Pezzo surface of degree one. Denote by
$\holom{f_G}{G}{S}$ the projection onto the second factor.
\item $G$ is the blow-up $\holom{f_G}{G}{G_0}$ along an elliptic curve, where $G_0 \subset \PP(1,1,1,2,3)$
is a smooth sextic. 
\end{itemize}
Moreover there exists a lift 
$\holom{f}{X}{X_1}$ of $g$ such that the restriction of $f$ to $G$ is
\begin{itemize}
\item the projection $f_G$, or
\item the blow-up $f_G$.
\end{itemize} 
In particular the restriction of $f$ to the base locus $B$ coincides with $g$. In the first
case $f$ defines a conic bundle structure, in the second case $f$ is the blow-up along a smooth surface.
\end{lemma}

\begin{proof}
The general fibre $G$ is a smooth Fano threefold with a fibration
$\holom{\pi_G:=\pi_{|G}}{G}{\PP^1}$ with general fibre a del Pezzo surface of degree one.
Thus, $\mbox{Bs}(|-K_G|)$ is not empty and we obtain the description of $G$ by Theorem \ref{thm:iskovskikh}. 
Since
$$
\mbox{Bs}(|-K_G|) = \mbox{Bs}(|-K_X|_G|) \subset \BsAX \cap G = B \cap G
$$
and $\mbox{Bs}(|-K_G|)$ is a curve,
we see that equality holds. In particular $f_G$ contracts $B \cap G$ onto a point.
Our goal is to show that the morphism $f_G$ extends to a morphism  $f$. Since
$f_G$ contracts $B \cap G$ it is then clear that $f$ is then a lift of $g$.

Denote by $\zeta_V \rightarrow V$ an ample tautological divisor for the ruling $\holom{g}{V}{\PP^1}$.

{\em 1st case. Assume that $G \simeq \PP^1 \times S$.} 
The restriction of the divisor $-K_X - 2 \pi^* \zeta_V$ to $G$ is nef and a supporting
divisor for the contraction of $f_G$. If we show that $-K_X - 2 \pi^* \zeta_V$ is
$(g \circ \pi)$-nef, we obtain the existence of $f$ by the relative basepoint-free theorem.
Arguing by contradiction we assume that this is not the case.
Since $-K_X$ is $(g \circ \pi)$-ample, the relative cone theorem yields the existence
of a $K_X$-negative extremal ray $\Gamma$ in $\mbox{NE}(X/V_1)$ that is $-K_X - 2 \pi^* \zeta_V$-negative. Since $-K_X - 2 \pi^* \zeta_V$ is nef on the general fibre $G$, 
the contraction of $\Gamma$ is contained in some special fibre $G_0$ of $g \circ \pi$, in particular the contraction is birational. 
By Lemma \ref{lemma:ruling-preparation} the contraction is divisorial and contracts the divisor $G_0$.
Yet $G_0 \simeq \pi^* g^* \sO_{\PP^1}(1)$ is a nef divisor, so this is impossible.

{\em 2nd case. $G$ is a blow-up $\holom{f_G}{G}{G_0}$.} 
Denote by $D_G$ the exceptional divisor of the blow-up $f_G$. From the description of $G$ we know that $\rho(G)=2$ and $D_G = -K_G - 2 F$, where $F$ is a general fibre of $G \rightarrow \PP^1$. Thus the restriction of the divisor $-K_X - 2 \pi^* \zeta_V$ to $G$ coincides with $D_G$. By construction the class is not $(g \circ \pi)$-nef, so
the relative cone theorem yields the existence
of a $K_X$-negative extremal ray $\Gamma$ in $\mbox{NE}(X/V_1)$ that is $(-K_X - 2 \pi^* \zeta_V)$-negative. It is clear that contraction of $\Gamma$ is not of fibre type,
so by Lemma \ref{lemma:ruling-preparation} it is divisorial with exceptional divisor $D$. The restriction of $D$ to the general fibre $G$ is exactly $D_G$. Since $D_G \simeq \PP^1\times X$ with $C$ an elliptic curve, it is clear
that the only way to contract $D_G$ is the blow-up $f_G$.

For the last statement just note that the fibres of $f$ have dimension at most one since $f$ is a lift of $g$. Thus the structure of $f$ is given by Ando's theorem \cite{And85}.
\end{proof}

\begin{lemma}\label{lem:ruling3}
In the situation of Theorem \ref{theorem-nefdivisor}, assume that $B \simeq V \simeq \PP^1 \times \PP^1$ and let $\holom{g_1}{V}{V_1 \simeq \PP^1}$ be one of the rulings. Assume that the lifting $\holom{f}{X}{X_1}$ is a conic bundle.
Then we have
$$
X \simeq X_1 \times \PP^1
$$
and $X_1$ is a smooth Fano threefold with anticanonical base locus. In this case we have
$$
\ma N_{B/X}^* \simeq \sO_{\PP^1 \times \PP^1}^{\oplus 2}, \qquad \mbox{or} \qquad
\ma N_{B/X}^* \simeq \sO_{\PP^1 \times \PP^1} \oplus \sO_{\PP^1 \times \PP^1}(1,0).
$$
\end{lemma}

\begin{remark*} Since we choose to order the product $X_1 \times \PP^1$ (and not $\PP^1 \times X_1$), the case
$\ma N_{B/X}^* \simeq \sO_{\PP^1 \times \PP^1} \oplus \sO_{\PP^1 \times \PP^1}(0,1)$ does not appear in the list,
the corresponding varieties are isomorphic.
\end{remark*}

\begin{proof}
By the proof Lemma \ref{lem:ruling} we know that the general $f$-fibre $d$ is of the form $\PP^1 \times \{p\} \subset G$, where $G$ is a fibre of $g_1 \circ \pi$. In particular $\pi_*(d)$
is reduced and meets a fibre of the second ruling $\holom{g_2}{V}{\PP^1}$ in exactly one point.
In other words the general $f$-fibre and the general $(g_2 \circ \pi)$-fibre meet transversally in one point. Thus the morphism
$$
\holom{\eta:= f \times (g_2 \circ \pi)}{X}{X_1 \times \PP^1}
$$
is birational. Since $X_1$ is the image of a conic bundle, it is smooth. Thus the exceptional locus of $\eta$ is divisorial. Yet $\rho(X)=\rho(X_1)+1=\rho(X_1 \times \PP^1)$, so the exceptional locus is empty and $\eta$ an isomorphism.
\end{proof}

\begin{proposition}\label{prop:quadric-nef}
In the situation of Theorem \ref{theorem-nefdivisor}, assume that $B \simeq V \simeq \PP^1 \times \PP^1$. 
Then $\ma N^*_{B / X}$ is nef.
\end{proposition}

\begin{proof}[Proof of Proposition \ref{prop:quadric-nef}]
By Lemma \ref{lem:ruling3} we can restrict to the case where
the two liftings of the rulings on $V$ are birational contractions.
Therefore for every line $d_i \subset B \simeq \PP^1 \times \PP^1$ the restriction
$$
\ma N_{B/X}^* \otimes \sO_{d_i} \simeq \sO_{\PP^1} \oplus \sO_{\PP^1}(1)
$$
is nef. We argue by contradiction and assume that $\ma N_{B/X}^*$ is not nef.

{\em Step 1. Setup.}
We consider the divisor
$$
E + \phi^* \zeta_T \simeq \mu^* A
$$
introduced in Lemma \ref{lemma:setup-classA}, by the lemma it is not nef. By item iii.\ of Lemma \ref{lemma:ruling-preparation} there exists
a prime divisor $D \subset X$ such that the restriction of $A$ to $D$ is not pseudoeffective. 
By item v.\ of Lemma \ref{lemma:setup-classA} there exists 
a prime divisor $D_T \subset T$ such that the restriction of $\zeta_T$ to $D_T$ is not pseudoeffective.
Moreover $D = \mu(\fibre{\phi}{D_T})$.

Note that $D_T$ maps surjectively onto $\PP^1 \times \PP^1$ since otherwise
$D_T$ is the pull-back of a nef divisor in $\PP^1 \times \PP^1$,
contradicting the fact that the restriction of the pseudoeffective divisor $\zeta_T$ to $D_T$ is not pseudoeffective.
By Lemma \ref{lemma:ruling-technical} we have a contradiction if we show that $D_T \simeq \PP^1 \times \PP^1$.

{\em Step 2. We prove that $D_T \simeq \PP^1 \times \PP^1$.}
Using the notation of Proposition \ref{prop:weierstrass} we consider the surface $S := \PP(\sO_{D_T}) \subset \PP(W)$. Then
$$
\sO_S(3 \zeta_W+ \Phi^* (6\zeta_T)) \simeq \sO_{D_T}(6 \zeta_T)
$$
is not pseudoeffective, so $S$ is contained in $X'$.
Note that $S \subset \PP(W_{D_T})$ is a locally complete intersection of codimension
two and
$$
\ma N_{S/\PP(W_{D_T})}^* \simeq \sO_{D_T}(-2 \zeta_T) \oplus \sO_{D_T}(-3 \zeta_T),
$$
so every global section of
$$
\ma N_{S/\PP(W_{D_T})}^* \otimes \sO_S(X') \simeq \sO_{D_T}(4 \zeta_T) \oplus \sO_{D_T}(3 \zeta_T)
$$
vanishes. By \cite[Lemma 1.7.4]{BS95} (applied over
the smooth locus of $D_T$) this shows that $\fibre{\phi}{D_T}$ is singular in every point of the surface $S$.
Since $S$ is disjoint from the exceptional divisor $E := \PP(\sO_{D_T}(-3 \zeta_T))$,
we can identify $S$ to its image in $D \subset X$
and obtain that $S$ is contained in the singular locus of $D$.

The divisor $D$ is not nef, since the restriction of the pseudoeffective divisor $A$ to
$D$ is not pseudoeffective. By the cone theorem there exists
a $K_X$-negative extremal ray $\R^+ [\Gamma]$ such that $D \cdot \Gamma<0$.
By item ii.\ of Lemma \ref{lemma:ruling-preparation} the contraction is not small,
so we obtain an extremal contraction $\holom{\eta}{X}{\bar X}$ with exceptional divisor
$D$. By Lemma \ref{lemma:divisor-contraction} the image of $D$ is not a point, so we have two cases.

{\em 1st case: $\eta(D)$ is a surface.} By
\cite[Thm.]{AW98} the morphism $D \rightarrow \eta(D)$ is a $\PP^1$-bundle
over the complement of finitely many points and the singular locus of $D$ maps onto
a finite set. Thus $S \subset D_{sing}$ is contained in a higher-dimensional fibre
of $D \rightarrow \eta(D)$ and therefore isomorphic to $\PP^2$ or a quadric
\cite[Thm.]{AW98}. By item i.\ of Lemma \ref{lemma:ruling-preparation} this
implies $S \simeq \PP^1 \times \PP^1$.

{\em 2nd case: $\eta(D)$ is a curve.}
By \cite[Main Theorem]{Tak99} the fibration $D \rightarrow \eta(D)$
is a $\PP^2$-bundle or quadric bundle such that the general fibre is irreducible.
An irreducible quadric surface has at most one singular point,
so a two-dimensional component of $D_{\sing}$ does not surject
onto $\eta(D)$. Therefore 
$S \subset D_{sing}$ is contained in a fibre
of $D \rightarrow \eta(D)$ and therefore isomorphic to $\PP^2$ or a quadric.
As in the first case this shows $S \simeq \PP^1 \times \PP^1$.

We conclude by noticing that
$$
D_T \simeq \PP(\sO_{D_T}) = S \simeq \PP^1 \times \PP^1.
$$
\end{proof}

\begin{lemma} \label{lemma:quadric-classifyNB}
In the situation of Theorem \ref{theorem-nefdivisor}, assume that $B \simeq V \simeq \PP^1 \times \PP^1$. 
Then $\ma N_{B/X}^*$ is isomorphic to a vector bundle $\sF$ of the following form:
$$
\sO_{\PP^1 \times \PP^1}^{\oplus 2}, \qquad 
\sO_{\PP^1 \times \PP^1} \oplus \sO_{\PP^1 \times \PP^1}(1,0)
$$
or
$$
\sO_{\PP^1 \times \PP^1} \oplus \sO_{\PP^1 \times \PP^1}(1,1), \qquad \sO_{\PP^1 \times \PP^1}(1,0) \oplus \sO_{\PP^1 \times \PP^1}(0,1)
$$
or
$$
0 \rightarrow \sO_{\PP^1 \times \PP^1}(-1,-1) \rightarrow \sO_{\PP^1 \times \PP^1}^{\oplus 3} \rightarrow \sF \rightarrow 0.
$$
In all the cases the vector bundle $\sF \simeq \ma N_{B/X}^*$ is globally generated and
$-K_{B} - \det \sF$ is ample and globally generated.
\end{lemma} 

\begin{proof}
Let $f_i$ be the liftings of the rulings $g_i$ constructed in Lemma \ref{lem:ruling}. If one of the $f_i$ is of fibre type we obtain
the classification by Lemma \ref{lem:ruling3}. If both liftings are birational we know that
 for every line $d_i \subset B \simeq \PP^1 \times \PP^1$ the restriction is
$$
\ma N_{B/X}^* \otimes \sO_{d_i} \simeq \sO_{\PP^1} \oplus \sO_{\PP^1}(1).
$$
Since $\ma N_{B/X}^*$ is nef by Proposition \ref{prop:quadric-nef}, we can conclude with Lemma \ref{lemma:ruling-technical2}.
\end{proof}

\section{Proof of the main theorem}

\begin{corollary}\label{corollary:high-degree}
In the situation of Theorem \ref{theorem-nefdivisor}, assume that $V \simeq B$ is a del Pezzo surface with
$\rho(B) \geq 3$. Then $\ma N_{B/X}^* \simeq \sO_{B}^{\oplus 2}$.
\end{corollary}

\begin{proof}
Applying Lemma \ref{lem:birational} we obtain a sequence 
$$
\xymatrix{
X \ar[r]_{f} \ar[d]_{\pi} \ar@/^1pc/[rrr]^{h} & X_1 \ar[d]^{\pi_1} \ar@{.>}[r] & X_{k-1} \ar[d]^{\pi_{k-1}} \ar[r] & X_k \ar[d]^{\pi_k} \\
V \ar[r]^{g} & V_1 \ar@{.>}[r] & V_{k-1} \ar[r] & V_k
}
$$
with $V_k$ isomorphic to $\PP^2$ or $\PP^1 \times \PP^1$ and
$\ma N_{B/X}^* \simeq h^*(\ma N_{B_k/X_k}^*)$. By item v.\ of the Lemma, the manifold $X_k$ satisfies the conditions of Theorem 
\ref{theorem-nefdivisor}. Thus we can use the classification of Lemma \ref{lemma:P2-classifyNB} and Lemma \ref{lemma:quadric-classifyNB} to obtain that $\ma N_{B/X}^*$ is nef.

Now note that Lemma \ref{lem:birational} does not depend on the choice of the exceptional curves in $V$, so $\ma N^*_{B / X}$ (and therefore $\det \ma N_{B/X}^*$) is trivial on all the $(-1)$-curves curves of $B$. For a del Pezzo surface $B$ with $\rho(B) \geq 3$
the Mori cone is generated by the $(-1)$-curves, so $\det \ma N_{B/X}^*$ is trivial. Thus the nef vector bundle
$\ma N_{B/X}^*$ is numerically flat. Since $B$ is simply connected, we can use \cite[Thm.1.18]{DPS94} to see that $\ma N_{B/X}^*$ is trivial.
\end{proof}

\begin{proof}[Proof of Theorem \ref{theorem-main}]
We already know by Theorem \ref{theorem-nefdivisor} that we have a flat elliptic fibration
$$
\holom{\phi}{X'}{T}
$$
such that the exceptional divisor $E$ is a $\phi$-ample section. By Proposition \ref{prop:weierstrass} 
this shows that $X'$ is a Weierstra\ss\ fibration with distinguished section $E$.

By Corollary \ref{cor:birational}, the surface $B$ is del Pezzo and from the classification 
of $\ma N_{B/X}^*$ in Lemma \ref{lemma:P2-classifyNB}, Corollary \ref{corollary:F1-classifyNB}, Lemma \ref{lemma:quadric-classifyNB}
and Corollary \ref{corollary:high-degree} we see that the list of possible normal bundles
coincides with the list of rank two Fano bundles in \cite[Thm.]{SW90}. In particular $E \simeq \PP(\ma N_{B/X}^*)$ is a Fano threefold.

From our classification, one sees easily that the pairs $(B,\, \ma N_{B/X}^*)$ satisfy the conditions of Proposition \ref{construction:general},
i.e.\ $\ma N_{B/X}^*$ is globally generated and  $-K_B - \det \ma N_{B/X}^*$ is ample.
Thus we can use the construction to recover $X'$ and $X$. If $B$ is not a del Pezzo surface of degree one, we can also
check that $-K_B - \det\ma N_{B/X}^*$ is globally generated, so $(B,\, \ma N_{B/X}^*)$ satisfies the conditions of Proposition \ref{construction:rk2}
and $\BsAX \simeq B$ is a smooth irreducible surface.
\end{proof}

\begin{proof}[Proof of Corollary \ref{corollary-main}]
By \cite[Thm.]{SW90} (which is of course contained in the famous list of \cite{MM81})
there are exactly 22 families of smooth Fano threefolds with a $\PP^1$-bundle structure $\PP(\sF) \rightarrow B$.
If $\rho(B) \geq 3$, we know from Corollary \ref{corollary:high-degree} that the unique possibility 
for $\ma N_{B/X}^*$ is $\sO_B^{\oplus 2}$, so $X \simeq V \times F$ with $F$ a del Pezzo of degree one.
In particular the case where $V$ is a del Pezzo surface of degree one does not appear, since then $\BsAX$ is reducible.

If $\rho(B) \leq 2$, we see from Lemma \ref{lemma:P2-classifyNB}, Corollary \ref{corollary:F1-classifyNB} and Lemma \ref{lemma:quadric-classifyNB} that all the other threefolds in \cite[Thm.]{SW90} appear in our list and
$T \simeq \PP^1 \times \PP^2$ appears twice, as the projectivisation of $\sO_{\PP^2}^{\oplus 2}$ and
$\sO_{\PP^2}(1)^{\oplus 2}$.
In conclusion we obtain $22-1+1=22$ families.
\end{proof}

\section{Numerical invariants and final table}
In the notation of Proposition \ref{construction:general}, let $\ma F\simeq\ma N_{B/X}^*$ under the isomorphism $B\simeq V$. We have $-K_{X'}\simeq \phi^*(-K_T-\zeta_T)$ by \eqref{formulaanticanonical} and $-K_T-\zeta_T$ is a tautological divisor for $$\PP\bigl(\ma F\otimes \sO(-K_V-\det \ma F)\bigr),$$ thus
$$h^0(X,-K_X)=h^0(X', -K_{X'})=h^0\bigl(V, \ma F\otimes \sO(-K_V-\det \ma F)\bigr)=h^0\bigl(V, \ma F^*\otimes \sO(-K_V)\bigr).$$
Moreover, $K_{X'}^4=0$ and \cite[Lemma 3.2]{CR} yield
$$K_X^4=6K_B^2+8K_B\cdot c_1(\ma F)+3c_1(\ma F)^2-c_2(\ma F).$$
To compute $K_{X}^2\cdot c_2(X)$, recall that the Hirzebruch-Riemann-Roch for a smooth Fano fourfold $X$ gives
$$h^0(X, -K_X)=1+\frac{1}{12}\bigl(2K_X^4+K_{X}^2\cdot c_2(X)\bigr).$$

It is not clear to us how to determine the relative Picard number of the elliptic fibration
$\holom{\phi}{X}{T}$ (equivalently, of $\holom{\pi}{X}{V}$), so the Picard number of $X$ is obtained from the description of morphism
$\holom{\psi}{X}{\bar X}$ and $\bar X$ in Section \ref{section-examples}. For the items $\# 1$
and $\# 4$, our description does not allow to determine the Picard number.

In the following table we collect the numerical invariants of the Fano fourfolds from Theorem \ref{theorem-main}, which we list based on the data $(B,\, \ma N_{B/X}^*)$.
We denote by $g\colon \F_1 \to \PP^2$ the blow-up at a point, and by $S_d$ the del Pezzo surface of degree $d$.

\begin{table}[h]
    $$\begin{array}{||c|c|c|c|c|c||}
\hline\hline
\rule{0pt}{2.5ex} \# & (B,\, \ma N_{B/X}^*)  & \rho_{X} & K_{X}^4 &  K_{X}^2\cdot c_2(X) & h^0(X,-K_{X})\\
    \hline\hline

\rule{0pt}{2.5ex} 1 & (\PP^2, \sO\oplus\sO(2)) & \leq 4 & 18 & 108 & 13 \\ \hline

\rule{0pt}{2.5ex} 2 & (\PP^2, \sO\oplus\sO(1)) & 2 & 33 & 114 & 16 \\ \hline

\rule{0pt}{2.5ex} 3 & (\PP^2,\sO^{\oplus 2}) & 10 & 54 & 120 & 20 \\ \hline

\rule{0pt}{2.5ex} 4 & (\PP^2, \sO(1)^{\oplus 2}) & ? & 17 & 98 & 12 \\ \hline

\rule{0pt}{2.5ex} 5 & (\PP^2, T_{\PP^2}(-1)) & 2 & 32 & 104 & 15 \\ \hline

\rule{0pt}{4ex} 6 & \makecell{(\PP^2, \ma F) \\ 0 \to \sO \rightarrow T_{\PP^2}(-1) \oplus \sO(1) \rightarrow \sF \to 0} & 2 & 16 & 88 & 11 \\ \hline

\rule{0pt}{4ex} 7 & \makecell{ (\PP^2, \ma F) \\ 0 \to \sO(-1)^{\oplus 2} \rightarrow \sO^{\oplus 4} \rightarrow \sF \to 0} & 2 & 15 & 78 & 10 \\ \hline

\rule{0pt}{4ex} 8 & \makecell{(\PP^2, \ma F) \\ 0 \rightarrow  \sO(-2) \rightarrow \sO^{\oplus 3} \rightarrow \sF \to 0} & 2 & 14 & 138 & 9 \\ \hline

\rule{0pt}{2.5ex} 9 & (\PP^1 \times \PP^1, \sO \oplus \sO(1,1)) & 3 & 22 & 100 & 13 \\ \hline

\rule{0pt}{2.5ex} 10 & (\PP^1 \times \PP^1, \sO^{\oplus 2}) & 11 & 48 & 108 & 18 \\ \hline

\rule{0pt}{2.5ex} 11 & (\PP^1 \times \PP^1, \sO \oplus \sO(1,0)) & 3 & 32 & 104 & 15 \\ \hline

\rule{0pt}{2.5ex} 12 & (\PP^1 \times \PP^1, \sO(1,0) \oplus \sO(0,1)) & 3 & 21 & 90 & 12 \\ \hline

\rule{0pt}{4ex} 13 & \makecell{(\PP^1 \times \PP^1, \ma F) \\ 0 \rightarrow \sO(-1,-1) \rightarrow \sO^{\oplus 3} \rightarrow \sF \rightarrow 0} & 3 & 20 & 80 & 11 \\ \hline

\rule{0pt}{2.5ex} 14 & (\F_1, g^*(\sO \oplus \sO(1)) & 3 & 27 & 102 & 14 \\ \hline

\rule{0pt}{2.5ex} 15 & (\F_1, \sO^{\oplus 2}) & 11 & 48 & 108 & 18 \\ \hline

\rule{0pt}{2.5ex}  16 & (\F_1, g^* T_{\PP^2}(-1)) & 3 & 26 & 92 & 13 \\ \hline

\rule{0pt}{2.5ex} 17 & (S_7, \sO^{\oplus 2}) & 12 & 42 & 96 & 16 \\ \hline

\rule{0pt}{2.5ex} 18 & (S_6, \sO^{\oplus 2}) & 13 & 36 & 84 & 14 \\ \hline

\rule{0pt}{2.5ex} 19 & (S_5, \sO^{\oplus 2}) & 14 & 30 & 72 & 12 \\ \hline

\rule{0pt}{2.5ex} 20 & (S_4, \sO^{\oplus 2}) & 15 & 24 & 60 & 10 \\ \hline

\rule{0pt}{2.5ex} 21 & (S_3, \sO^{\oplus 2}) & 16 & 18 & 48 & 8 \\ \hline

\rule{0pt}{2.5ex} 22 & (S_2, \sO^{\oplus 2}) & 17 & 12 & 36 & 6 \\
      
\hline\hline
   \end{array}$$

 \bigskip

   \caption{Numerical invariants of the Fano fourfolds from Theorem \ref{theorem-main}}\label{table}
 \end{table}


\newcommand{\etalchar}[1]{$^{#1}$}
\providecommand{\bysame}{\leavevmode\hbox to3em{\hrulefill}\thinspace}
\providecommand{\MR}{\relax\ifhmode\unskip\space\fi MR }
\providecommand{\MRhref}[2]{%
  \href{http://www.ams.org/mathscinet-getitem?mr=#1}{#2}
}
\providecommand{\href}[2]{#2}

\end{document}